\documentclass[letterpaper, reqno,oneside]{amsart}

\usepackage{amsmath}
\usepackage{amsthm}
\usepackage{graphicx} 
\usepackage{geometry}    
\usepackage[parfill]{parskip}    
\usepackage{graphicx}
\usepackage{amssymb}
\usepackage{epstopdf}
\usepackage{multicol}
\usepackage{hyperref}
\usepackage{cleveref}
\usepackage{MnSymbol}
\usepackage{bm}
\usepackage{soul}
\usepackage{pgf,tikz,tikz-cd}
\usepackage{mathrsfs}
\usetikzlibrary{arrows}
\usepackage{color}

\DeclareMathOperator{\PGL}{PGL}
\DeclareMathOperator{\diag}{diag}
\newtheorem*{theorem*}{Theorem}
\newtheorem*{question*}{Question}

\newcommand{\ba}{\mathbf{a}}
\newcommand{\bb}{\mathbf{b}}
\newcommand{\IOMAT}[1]{\left[
I \, \vert \, 0\right]_{#1}}
\newcommand{\OIMAT}[1]{\left[
0 \, \vert \, I\right]_{#1}}

\newtheorem{theorem}{Theorem}
\newtheorem{lemma}[theorem]{Lemma}
\newtheorem{question}{Question}
\newtheorem{definition}[theorem]{Definition}
\newtheorem{proposition}[theorem]{Proposition}
\newtheorem{corollary}[theorem]{Corollary}
\theoremstyle{remark}
\newtheorem{remark}{Remark}
\newtheorem{example}[theorem]{Example}

\Crefname{question}{Question}{Question}

\newcommand{\rank}{\textup{rank}}
\newcommand{\nul}{\textup{null}}
\newcommand{\RR}{\mathbb{R}}
\newcommand{\NN}{\mathbb{N}}
\newcommand{\PP}{\mathbb{P}}
\newcommand{\CC}{\mathbb{C}}
\newcommand{\ZZ}{\mathbb{Z}}

\newcommand{\SL}{\textup{SL}}
\newcommand{\M}{\mathcal{M}}
\newcommand{\X}{\mathcal{X}}
\newcommand{\Y}{\mathcal{Y}}
\newcommand{\Pcal}{\mathcal{P}}
\newcommand{\Qcal}{\mathcal{Q}}
\newcommand{\ones}{\mathbf{1}}
\newcommand{\twos}{\mathbf{2}}

\newcommand{\FE}[1]{\mathcal{FC}_{#1}}
\newcommand{\EPV}[1]{\mathcal{C}_{#1}}
\newcommand{\ORBIT}[2]{{\mathcal{#1}}^{#2}}

\title{A Computer Vision Problem in Flatland}

\author{Sameer Agarwal}
\address{Google Inc.}
\email{sameeragarwal@google.com}
\author{Erin Connelly}
\address{Department of Mathematics, University of Osnabr\"uck}\email{erin.connelly@uni-osnabrueck.de}
\author{Annalisa Crannell} \address{Department of Mathematics, Franklin \& Marshall College}
\email{annalisa.crannell@fandm.edu}
\author{Timothy Duff}
\address{Department of Mathematics University of Missouri}
\email{tduff@missouri.edu}
\author{Rekha R. Thomas} \address{Department of Mathematics, University of Washington}
\email{rrthomas@uw.edu}
\date{\today}
\begin{document}
\begin{abstract}
\vspace{-0.2cm}
    When is it possible to project two sets of labeled points lying in a pair of projective planes to the same image on a  projective line? We give a complete answer to this question and describe the loci of the projection centers that enable a common image. In particular, we find that there exists a solution to this problem if and only if these two sets are themselves images of a common pointset in projective space.
\end{abstract}
\maketitle
 \vspace{-0.35cm}
Imagine you are a robot equipped with a camera, exploring a new world. The images that your camera captures, depend on your position and orientation  as well as the scene that you are looking at. Generally speaking, if you are looking at two different scenes with your camera in two different locations, then the images will also be different, but this begs the question, when is this not the case? i.e., when are the  images of two different scenes from  two different cameras the same? For simplicity we will assume that the world being explored is two dimensional (also known as a flatland), and the images are one dimensional. In this idealized setting we consider the following question:

{\em Given two sets of $n$ labeled points $\{x_i\}$ and $\{y_i\}$
in the plane, when is there a pair of flatland cameras, the first camera imaging $\{x_i\}$ and the second camera imaging $\{y_i\}$,  such that they produce the same image? If such a pair of cameras exists, then describe the set of all such camera pairs.
}  

To answer this question, we first need to formalize it. 

Let $\PP^k$ denote the $k$-dimensional real projective space.
A point in $\PP^k$ may be represented by a nonzero $(k+1)$-vector, and two such vectors $x_1, x_2$ represent the same point if there exists a nonzero scalar $\lambda$ with $x_1 = \lambda x_2,$ in which case we write $x_1\sim x_2.$ 
Abusing notation slightly, we write $A$ for any $(k+1) \times (l+1)$ matrix representing a linear projection\footnote{The notation $\dashrightarrow $ indicates that $A$ is a \emph{rational map}---it may be undefined at certain points in its domain (namely, the points in the right nullspace of $A$).} $A : \PP^l \dashrightarrow \PP^k$. 
When $\rank (A) = k+1 = l,$ the \emph{center} of the projection $A$ is the unique point $a\in \PP^l$ in the right nullspace of $A$ (at which the projection is undefined). A {\em flatland camera} is a rank 2 linear projection from $\PP^2$ to $\PP^1$. If the camera pair $(A,B)$ projects $\mathcal{X}$ and $\mathcal{Y}$ to the same image in $\PP^1$, then so does the camera pair $(HA, HB)$, where $H \in \PGL(2)$. Multiplication by $H$ is equivalent to picking a coordinate system for the image line in $\PP^1$, it does not change the center of projection of the cameras. Put another way, unlike the camera matrix, the center is a projective invariant. 

We can now state our questions precisely.
\begin{question*}
Let $\mathcal{X} = \{x_1,\dots,x_n\}$ and $\mathcal{Y} = \{y_1, \dots, y_n\}$ be two sets of $n$ labeled points in $\PP^2$. 
\begin{enumerate}
    \item[I.] When are there linear projections $A:\PP^2\dashrightarrow \PP^1$ and $B:\PP^2 \dashrightarrow \PP^1$ such that $\forall i$, $Ax_i \sim By_i$? 
    \item[II.] If such a pair $(A,B)$ exists, with $A,B$ full rank, describe the locus of possible centers $(a,b)$.
\end{enumerate}
\end{question*}

Our first main result (\Cref{thm:twocam-existence-general}) provides a complete answer to Question I.
Surprisingly, the answer is that $\mathcal{X}$ and $\mathcal{Y}$ have the same image in $\PP^1$ if and only if they themselves are the images of the same point set in $\PP^3$. In fact, we will see that the question can be posed and answered for two sets of labeled points in any $\PP^k$ that have the same image in $\PP^{k-1}$. We will also see that if there are more than two sets of labeled points in $\PP^2$ that have the same image in $\PP^1$, then we cannot get a theorem as sharp as the one for two sets of labeled points.

Our second set of results characterize the loci of centers of projections (camera centers) for various values of $n$, answering Question II. 
Here, to obtain a clean answer, we must assume that the points in $\X$ and $\Y$ are sufficiently generic.
The following theorem summarizes our results.

\begin{theorem}[Loci Theorem]
    \label{thm:loci of centers}
For generic $\X $ and $\Y$ in $\PP^2$, the loci of centers $(a,b) \in \PP^2 \times \PP^2$ in Question II satisfy the following: 
\begin{itemize}
    \item 
     If $n< 4$, both $a$ and $b$ may be chosen arbitrarily in $\PP^2$ (introduction to \Cref{sec:n-4}).
    \item 
    If $n=4$, then one of the centers may be chosen arbitrarily, and this choice determines a plane conic on which the other center can be chosen arbitrarily (\Cref{thm:n=4 loci invariants,thm:main thm n=4}).
    \item 
    If $n=5$, then one of the centers may be chosen arbitrarily, and this choice determines the other center uniquely (\Cref{thm:n=5 loci invariants,thm:main thm n=5}). 
    \item 
    If $n=6$, $a$ and $b$ must lie on plane cubic curves $C_x$ and $C_y$; if $a\in C_x$ is chosen arbitrarily, then $b\in C_y$ is uniquely determined, and vice-versa (\Cref{thm:n=6 invariant theory,thm:main thm n=6}).
    \item 
    If $n=7$, there are at most three choices for the pair $(a,b)$ (\Cref{thm:n-7}).
    \item 
    If $n>7$, the locus of $(a,b)$ is empty (\Cref{cor:n=8 loci invariants}).
\end{itemize}
\end{theorem}

While the answer to Question I only requires elementary projective geometry and linear algebra, answering Question II calls upon tools from algebraic geometry and invariant theory. The answers themselves come in the 
form of explicit algebraic equations and synthetic geometric constructions that describe the loci of centers.

\begin{remark}\label{construction-remark}
It is worthwhile to clarify the nature of constructions involving
plane algebraic curves---specifically, conics and cubics---as they are key players in our story.
It is a consequence of elementary Galois theory that plane conics other than circles cannot be constructed using only a straightedge and compass.
On the other hand, if we are given five generic points on a conic, then we may construct as many additional points on the conic as desired using the classical \emph{Braikenridge-Maclaurin Theorem} (a converse to Pascal's Theorem, see~\cite[Thm.~9.22]{coxeter1987}.)
Similarly, if we are given a rational point on a rational conic, we may generate infinitely many more rational points using stereographic projection; conics with this property appear in our analysis of $n=7$ points.
For the cubics $C_x, C_y$ appearing in~\Cref{thm:loci of centers}, it is known from algebraic geometry that no rational parametrization exists.
Nevertheless, it is possible to test whether or not ten points lie on a common cubic with a straightedge and compass---see~\cite{TravesWehlau2024} for this recent result and many interesting related questions.
\end{remark}

\textbf{Organization of the paper.} 
In Section~\ref{sec:existence} we will answer Question I. The rest of the paper is devoted to answering Question II. 
In Section~\ref{sec:invariants} we introduce the necessary tools from invariant theory of points in $\PP^1$ and in Section~\ref{sec:epv} we introduce the {\em camera centers variety}. Sections~\ref{sec:n-4} - \ref{sec:7-points} answer Question II for each of $n=4,5,6,7$.

\textbf{Acknowledgments.} Crannell acknowledges a Fulbright award. Duff acknowledges support from NSF DMS-2103310, and Thomas from the Walker Family Professorship.

\section{Answer to Question I}\label{sec:existence}

We first motivate a more careful formulation of Question I.
If we take any scalar $\lambda $ and row vectors $\ba$ and $\bb$ such that $\ba^\top x_i$ and $\bb^\top y_i$ are nonzero for all $i,$ then
\[A=\begin{bmatrix}
    \ba^\top \\ \lambda \ba^\top 
\end{bmatrix}, \quad  
B=\begin{bmatrix} \bb^\top \\ \lambda \bb^\top \end{bmatrix}
\quad 
\Rightarrow
\quad 
Ax_i \sim By_i \sim \begin{bmatrix}1\\ \lambda\end{bmatrix}.
\] 
So, in a trivial sense, the answer to Question I is always ``yes." 
Thus, going forward we exclude rank-one projections.
For any $l\ge 2,$ we define a \emph{camera} to be a full rank linear projection; the main examples of interest to us are the classical \emph{pinhole camera} $\PP^3 \dashrightarrow \PP^2$ and the \emph{flatland camera} $\PP^2 \dashrightarrow \PP^1.$
Thus, we arrive at the following reformulation of Question I.

\begin{question}\label{question:q1}
Let $\mathcal{X} = \{x_1,\dots,x_n\}$ and $\mathcal{Y} = \{y_1, \dots, y_n\}$ be two sets of labeled points in $\PP^2$. When do there exist flatland cameras $A:\PP^2 \dashedrightarrow \PP^1$ and $B:\PP^2 \dashedrightarrow \PP^1$
such that $\forall i , \, Ax_i \sim By_i$?
\end{question}

\Cref{thm:twocam-existence-general} below, in the special case $l=2,$ provides a complete answer to~\Cref{question:q1}.

\begin{theorem}
\label{thm:twocam-existence-general}
    Fix $l\ge 2,$ and let $\mathcal{X} = \{x_1,\dots,x_n\}$ and $\mathcal{Y} = \{y_1, \dots, y_n\}$ be two sets of $n$ labeled points in $\PP^l$. There exist cameras $A:\PP^l\dashedrightarrow \PP^{l-1}$ and $B:\PP^l \dashedrightarrow \PP^{l-1}$ such that $\forall i,\ Ax_i \sim By_i$ if and only if there exist cameras $A':\PP^{l+1}\dashedrightarrow \PP^l$ and $B':\PP^{l+1} \dashedrightarrow \PP^l$ with centers $a',b'$, and a set of points $\mathcal{Z} = \{z_1,\dots,z_n\}\subset  \PP^{l+1}$, none lying on the line connecting  $a', b',$ such that
    \[
    \forall i, \, x_i \sim A'z_i \text{ and } y_i \sim B'z_i.
    \]
\end{theorem}

The triple $(\mathcal{Z},A',B')$ in~\Cref{thm:twocam-existence-general} is said to be a \emph{reconstruction} of $(\X , \Y )$: for $l=2,$ we say it is a \emph{3D reconstruction}. 
Graphically~\Cref{thm:twocam-existence-general} is represented by the following diagram.
\begin{center}
\begin{tikzcd}
& \PP^{l+1}(\mathcal{Z}) \arrow{dl}{}[swap]{A'} \arrow{dr}{B'} & \\
\PP^l(\mathcal{X}) \arrow{dr}{}[swap]{A} & & \PP^l(\mathcal{Y}) \arrow{dl}{B} \\
& \PP^{l-1} &
\end{tikzcd}
\end{center}
After giving a brief, purely algebraic proof of~\Cref{thm:twocam-existence-general}, we will shift our focus towards understanding the case where $l=2$ in greater detail, by exploring its geometric interpretation, connections to computer vision, and a generalization to the case of more than two projections in~\Cref{thm-many-cameras}.

For any $l\ge 2$, we let $\IOMAT{l} $ denote the 0-1 matrix representing the camera $\PP^l \dashrightarrow \PP^{l-1}$ defined by $[x_0: \cdots : x_l] \mapsto [x_0: \cdots : x_{l-1}].$
Similarly, $\OIMAT{l} $ will denote the 0-1 matrix representing the camera $[x_0: \cdots : x_l] \mapsto [x_1: \cdots : x_{l}].$

\begin{lemma}\label{lemma:diamond-special matrix}
    Fix $l\ge 2,$ let $\mathcal{X} = \{x_1,\dots,x_n\}, \mathcal{Y} = \{y_1, \dots, y_n\} \subset \PP^l$ be sets of labeled points, and set 
    \[
A=\OIMAT{l}, \,  B=\IOMAT{l},
\quad 
A'=\IOMAT{l+1}, \, B' = \OIMAT{l+1}.
    \]
We have $\forall i, \, Ax_i \sim By_i$ if and only if there exists a set of points $\mathcal{Z} = \{z_1,\dots,z_n\}\subset  \PP^{l+1}$, none lying on the line connecting the centers $a' = [0:0:\cdots : 1],$ $b'=[1:\cdots : 0 : 0] \in \PP^{l+1},$ such that 
    \[
    \forall i, \, x_i \sim A'z_i \text{ and } y_i \sim B'z_i.
    \]
\end{lemma}

\begin{proof}
To simplify notation, fix a pair of points with the same label and their representatives in homogeneous coordinates, $x = [x_0:\cdots : x_l] \in \X$ and $y = [y_0 : \cdots : y_l] \in \Y $.
We first prove the forward direction. Suppose there exists a scalar $\lambda $ such that $
\OIMAT{l} x
=
\lambda  \IOMAT{l} y \ne 0.$
Consider the point 
\[
z = [
x_0:
x_1:
\cdots :
x_l:
\lambda y_l] =
[
x_0 : \lambda y_0 : \cdots : \lambda  y_{l-1} : \lambda  y_l]
\in \PP^{l+1}.
\]
Points on the line spanned by $a'$ and $b'$ are of the form $[\mu_1:0:\cdots:0:\mu_{l+1}]$.Consequently, $\OIMAT{l} x = [x_1:\cdots:x_l] \ne 0$ implies that $z$ does not lie on the line spanned by $a'$ and $b'$. Moreover,
\[
\IOMAT{l+1}z = 
x,
\quad \text{ and } \quad 
\OIMAT{l+1}z = \lambda y 
\sim 
y.
\]
Thus, if $A=\OIMAT{l}$ and $B=\OIMAT{l}$ project $\X $ and $\Y$ to the same image, then $z$, $A'=\IOMAT{l+1}$ and $B'=\OIMAT{l+1}$ provide the necessary reconstruction.
For the converse, note that if $z\in \PP^{l+1}$ satisfies both
$\IOMAT{l+1} z \sim x$ and $\OIMAT{l+1} z 
\sim y,$
then $\OIMAT{l} x \sim \IOMAT{l} y \sim [z_1: \cdots : z_{l}] \in \PP^{l-1}$.
\end{proof}

\begin{proof}[Proof of~\Cref{thm:twocam-existence-general}]
We reuse the structure of the previous proof, constructing $(A,B)$ from $(A',B')$ and vice-versa. 
Suppose first that $Ax \sim By.$
Choose invertible matrices $S_1, S_2, T$ so that
\[
S_1 A T = \OIMAT{l},
\quad 
S_2 B T = \IOMAT{l}.
\]
Note that $T$ can be any homography of $\PP^l$ with $T a = e_1$ and $T b = e_{l+1}.$ where $e_i$ denotes the $i$th standard unit vector in $\PP^l$. 
Then 
\[
\OIMAT{l} (T^{-1} x) = S_1^{-1} Ax \sim S_1^{-1} By 
= S_1^{-1} S_2^{-1} \IOMAT{l} (T^{-1} y)
= \IOMAT{l} (S T^{-1} y),\\
\]
where $S$ is the invertible matrix
\[
S = \left[\begin{array}{c|c}
S_1^{-1} S_2^{-1} &0\\
\hline 
0^\top &1
\end{array}\right].
\]
By~\Cref{lemma:diamond-special matrix}, $T^{-1} x$ and $S T^{-1}y$ admit a valid reconstruction via $\IOMAT{l+1}$ and $\OIMAT{l+1}.$
Thus $x$ and $y$ also admit a valid reconstruction by setting $A'=T \IOMAT{l+1}$ and $B'=TS^{-1} \OIMAT{l+1}.$

Conversely, suppose that
$A' z \sim x$ and $B' z \sim y.$
Choosing now invertible matrices $S_1,S_2, T$ so that 
\[
S_1 A' T = \IOMAT{l+1},
\quad 
S_2 B' T = \OIMAT{l+1},
\]
we have that 
\[
\IOMAT{l+1} (T^{-1} z) = S_1 A' z
\sim S_1 x,
\quad 
\text{and} 
\quad 
\OIMAT{l+1} (T^{-1} z) = S_2 B' z \sim S_2 y.
\]
Applying~\Cref{lemma:diamond-special matrix}, $S_1x$ and $S_2y$ project to the same image under $\OIMAT{l}$ and $\IOMAT{l}.$
Thus, $x$ and $y$ project to the same image under $A=\OIMAT{l} S_1$ and $B=\IOMAT{l} S_2.$
\end{proof}

Having proven~\Cref{thm:twocam-existence-general} for any $l\ge 2,$ we now focus on the special case $l=2$ and prove another criterion for the existence of projections $A$ and $B$.
\begin{theorem}
\label{thm:fundamental-equivalence}
    Let $\mathcal{X} = \{x_1,\dots,x_n\}$ and $\mathcal{Y} = \{y_1, \dots, y_n\}$ be two sets of labeled points in $\PP^2$. There exist flatland cameras $A:\PP^2\dashedrightarrow \PP^1$ and $B:\PP^2 \dashedrightarrow \PP^1$ with centers $a$ and $b$ such that $\forall i,\ Ax_i \sim By_i$ if and only if there exists a $3\times 3$ matrix $F$ of rank 2 such that 
    \[
    Fa = 
    F^T b = 0, 
    \, \, \text{and } \, \, 
    \forall i , \, \, 
    y_i^T F x_i =0,
    \quad 
    F x_i  \ne 0,
    \quad  
    F^T y_i \ne  0.
    \]
\end{theorem}
\begin{proof}
We first establish the forward direction.
Starting from flatland cameras
\[
A = \begin{bmatrix}
    \mathbf{a}_1^\top\\
    \mathbf{a}_2^\top
\end{bmatrix}, \quad 
B = \begin{bmatrix}
    \mathbf{b}_1^\top\\
    \mathbf{b}_2^\top
\end{bmatrix},
\]
we define $F = \mathbf{b}_2 \mathbf{a}_1^T - \mathbf{b}_1 \mathbf{a}_2^T.$
Clearly $Fa=F^Tb =0.$
Moreover, for any $i,$
\[
y_i^T F x_i = 
\det \begin{bmatrix}
\mathbf{a}_1^T x_i & \mathbf{b}_1^T y_i \\
\mathbf{a}_2^T x_i & \mathbf{b}_2^T y_i
\end{bmatrix}
=
\det \begin{bmatrix}
    Ax_i & By_i 
\end{bmatrix}
=0,
\]
since $Ax_i \sim By_i.$
If $F$ were not rank $2,$ this would imply that $\mathbf{a}_1 \mathbf{b}_2^\top$ was a multiple of $\mathbf{a}_2 \mathbf{b}_1^\top$, or vice-versa; supposing for instance that $\mathbf{a}_1 \mathbf{b}_2^\top = \lambda \mathbf{a}_2 \mathbf{b}_1^\top$ for some scalar $\lambda ,$ we could write $\mathbf{a}_1 = \lambda (b_{1 i}/b_{2 i}) \mathbf{a}_2$ for some $i,$ contradicting the fact that $A$ is full rank.
Since $F$ is rank 2, the matrices $A$ and $F$ have the same nullspace, so $Ax_i \ne 0$ implies $F x_i \ne 0$; similarly, $\nul (B) = \nul (F^T)$ gives $F^T y_i \ne 0.$ 

For the reverse direction, note first that the rank-$2$ matrix $F$ can be factorized as 
\begin{align}
    F = UV^\top = \begin{bmatrix}
    u_1 & u_2 
\end{bmatrix} 
\begin{bmatrix}
    v_1^\top\\
    v_2^\top
\end{bmatrix}
\end{align}
where $U$ and $V$ are $3\times 2$ matrices of full rank. Setting
\begin{align}
A = \begin{bmatrix}
    -v_2^\top \\
    v_1^\top 
\end{bmatrix} \text{ and } B = \begin{bmatrix}
    u_1^\top \\
    u_2^\top
\end{bmatrix},  
\end{align}
we may readily verify $\nul (A) = \nul (F),$ $\nul (B) = \nul (F^T),$ and that $\forall i \, \, A x_i \sim B y_i.$
\end{proof}

The matrix $F$ appearing in~\Cref{thm:fundamental-equivalence} is known in computer vision as the \emph{fundamental matrix}~\cite{hartley2003multiple}. Combining~\Cref{thm:twocam-existence-general} with~\Cref{thm:fundamental-equivalence}, we recover a classical description of this matrix.
\begin{theorem}\label{thm:fundamental-facts}
\cite[Chapter 9]{hartley2003multiple} For a pair of pinhole cameras $A',B':\PP^3\dashedrightarrow\PP^2$ with centers $a',b' \in \PP^3,$  let $a=A' b'$ and $b=B'a'$. There exists a $3\times 3$ matrix $F$ of rank $2$, the \textbf{fundamental matrix} of the pair $(A',B')$, whose left and right nullspaces are spanned by $a$ and $b,$ respectively, and such that, for $x\neq a,y\neq b\in\PP^2$, $y^\top Fx=0$ if and only if there exists $z\in\PP^3$ with $A' z=x$ and $B'z=y$.
Furthermore, these properties determine $F$ up to a scalar multiple.
\end{theorem}
The points $a,b$ are known as the right and left \textit{epipoles} of the fundamental matrix $F$. 
There is a simple, well-known geometric interpretation of the fundamental matrix; namely, there is a bijective correspondence between points $x\neq a\in\PP^2$ and lines through the left epipole $b$ defined by the linear map $\PP_x^2 \dashrightarrow \left( \PP_y^2\right)^*$ that sends $x \mapsto F x$---here $\left( \PP_y^2 \right)^*$ denotes the \emph{dual projective space} whose points represent lines in $\PP_y^2.$
This correspondence is illustrated in~\cite[Fig.~9.5]{hartley2003multiple}.
Similarly, the fundamental matrix $F^T$ maps points $y\ne b \in \PP_y^2$ to lines through the right epipole $a$.

In higher dimensions, the so-called \emph{Grassmann tensors} of projections $\PP^{l} \dashrightarrow \PP^k$ studied by Hartley and Schaffalitzky~\cite{hartley2009reconstruction} play a role comparable to the fundamental matrix.
To avoid technicalities, we limit our focus to the case $(l,k)=(3,2),$ where the geometry is already quite interesting.

\Cref{thm:fundamental-equivalence} admits a partial generalization to the case of $m>2$ sets of labeled points. In particular, we find that the existence of flatland cameras is equivalent to the existence of pairwise compatible fundamental matrices, but, in general, this is not sufficient to obtain a $3D$ reconstruction.

\begin{theorem}
\label{thm-many-cameras}

Suppose we are given $m$ sets of labeled points $\mathcal{X}_j=\{x_1^j,\ldots,x_n^j\}$ in $\PP^2$. There exist flatland cameras $A_j:\PP^2\dashedrightarrow\PP^1$ with centers $a_j\in\PP^2$ such that $\forall i,j,k$, $A_ix_k^i\sim A_jx_k^j$ if and only if there exist fundamental matrices $F^{ij}$ with left and right epipoles $a_i,a_j$ satisfying
\begin{equation}\label{eq:Fundamental matrices multiview1}
\forall i,j,k, \, \, \, (x_k^i)^\top F^{ij}x_k^j=0 .
\end{equation}
\end{theorem}

\begin{proof}
The forward direction follows immediately from our work on the two-view case. We therefore need only prove the reverse direction.

For any $x$ in the $j$-th image, the line $F^{ij}x$ passes through the epipole $a_i$ in the $i$-th image. 
Additionally, the linear map represented by $F^{ij}$ is constant on lines through $a_j$ in the $j$-th image.
Condition \eqref{eq:Fundamental matrices multiview1} is equivalent to the statement that the line $F^{ij}x_k^j$ is the line $\langle a_i,x_k^i \rangle $ spanned by $a_i$ and $x_k^i$. Equivalently, by~\Cref{thm:fundamental-equivalence}, there exists a pair of flatland cameras $A_i^j,A_j^i$ with centers $a_i,a_j$ such that $A_i^jx_k^i \sim A_j^ix_k^j$.

Furthermore, the flatland cameras $A_i^j$, $j=1,\ldots,i-1,i+1,\ldots,m$ all share $a_i$ as their center of projection. Thus, they are equal up to change of coordinates in $\PP^1$. We can therefore assume that $A_i^1=\ldots=A_i^m:=A_i$, finishing our proof.
\end{proof}

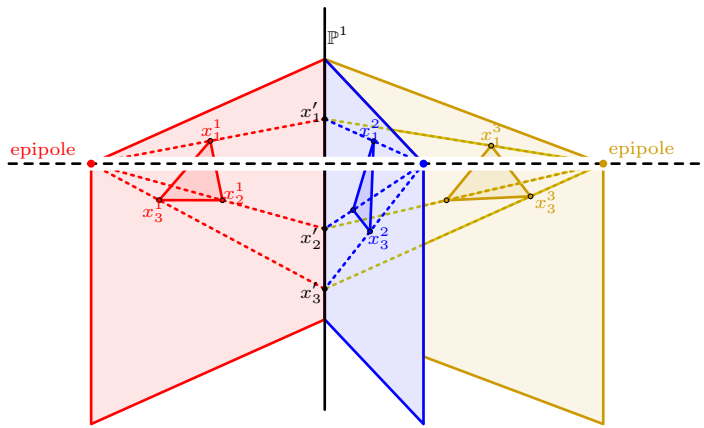
\begin{figure}
    \centering
 \definecolor{XTHREEbehind}{rgb}{0.8, 0.8,0.}
\definecolor{XTWO}{rgb}{0.,0.,1.}
\definecolor{XTHREE}{rgb}{0.8, 0.6,0.}
\definecolor{XONE}{rgb}{1.,0,0.}

\begin{tikzpicture}[scale =0.6,line cap=round,line join=round,>=triangle 45,x=1.0cm,y=1.0cm]

\draw[line width=1.0pt,color=XTHREE,fill=XTHREE,fill opacity=0.10000000149011612] (0.,5.7742237762211195) -- (6.171233006104707,3.455565912139384) -- (6.171233006104707,-2.3186578640817355) -- (2.18818597535383,-0.8221460143878152) -- (2.18818597535383,3.455565912139384) -- cycle;
\draw[line width=1.0pt,color=XTHREE,fill=XTHREE,fill opacity=0.10000000149011612] (3.6875927818500154,3.854100514146576) -- (2.6973237672120485,2.643547349503118) -- (4.560167514788904,2.730269972582122) -- cycle;
\draw [line width=1.0pt,dotted,color=XTHREEbehind] (6.171233006104707,3.455565912139384)-- (0.,0.6772988177744939);
\draw [line width=1.0pt,dotted,color=XTHREEbehind] (0.,4.445826041590958)-- (6.171233006104707,3.455565912139384);
\draw [line width=1.0pt,dotted,color=XTHREE] (1.4773735876549052,4.208760910731003)-- (6.171233006104707,3.455565912139384);
\draw [line width=1.0pt,dotted,color=XTHREE] (2.18818597535383,1.6624123413867864)-- (6.171233006104707,3.455565912139384);
\draw [line width=1.0pt,dotted,color=XTHREEbehind] (0.,2.0130539146446016)-- (6.171233006104707,3.455565912139384);
\draw [line width=1.0pt,dotted,color=XTHREE] (2.18818597535383,2.524537522389644)-- (6.171233006104707,3.455565912139384);

\draw[line width=1.0pt,color=XTWO,fill=XTWO,fill opacity=0.10000000149011612] (0.,5.7742237762211195) -- (2.18818597535383,3.455565912139384) -- (2.18818597535383,-2.3186578640817355) -- (0.,0.) -- cycle;
\draw[line width=1.0pt,color=XONE,fill=XONE,fill opacity=0.10000000149011612] (0.,5.7742237762211195) -- (-5.175819581499535,3.455565912139384) -- (-5.175819581499535,-2.3186578640817355) -- (0.,0.) -- cycle;
\draw[line width=1.0pt,color=XONE,fill=XONE,fill opacity=0.10000000149011612] (-2.5360850511530666,3.960611309603055) -- (-2.2664854806503656,2.6447282700785024) -- (-3.668554603677176,2.646498937444091) -- cycle;
\draw[line width=1.0pt,color=XTWO,fill=XTWO,fill opacity=0.10000000149011612] (1.0847263776469522,3.9549348215361744) -- (0.6288195001768166,2.4275889146310883) -- (1.0043060258130363,1.952432840026672) -- cycle;

\draw [line width=1.0pt] (0.,-2) -- (0.,6.9);

\draw [line width=1.0pt,dotted,color=XONE] (-5.175819581499535,3.455565912139384)-- (0.,0.6772988177744939);
\draw [line width=1.0pt,dotted,color=XONE] (0.,4.445826041590958)-- (-5.175819581499535,3.455565912139384);
\draw [line width=1.0pt,dotted,color=XTWO] (2.18818597535383,3.455565912139384)-- (0.,0.6772988177744939);
\draw [line width=1.0pt,dotted,color=XTWO] (0.,4.445826041590958)-- (2.18818597535383,3.455565912139384);
\draw [line width=1.0pt,dotted,color=XTWO] (2.18818597535383,3.455565912139384)-- (0.,2.0130539146446016);
\draw [line width=1.0pt,dotted,color=XONE] (0.,2.0130539146446016)-- (-5.175819581499535,3.455565912139384);
\draw [line width=5.2pt,color=white,domain=-7.0: 8.5] plot(\x,{(--39.21048812497821-0.*\x)/11.347052587604242});
\draw [line width=1.0pt,dash pattern=on 3pt off 3pt,domain=-7.0: 8.5] plot(\x,{(--39.21048812497821-0.*\x)/11.347052587604242});

\begin{scriptsize}
	\draw[color=black] (0.3, 6.3) node {${\mathbb P}^1$};
\fill [color=XTWO] (2.18818597535383,3.455565912139384) circle (2.5pt);
\fill [color=XONE] (-5.175819581499535,3.455565912139384) circle (2.5pt);
	\draw[color=XONE] (-6.241101468026629,3.7574744881916913) node {epipole};
\fill [color=XTHREE] (6.171233006104707,3.455565912139384) circle (2.5pt);
	\draw[color=XTHREE] (7.018723192190756,3.7816271742758762) node {epipole};

\draw [fill=black] (0.,4.445826041590958) circle (1.5pt);
		\draw[color=black] (-0.3, 4.6) node {$x_1'$};
\draw [fill=black] (0.,2.0130539146446016) circle (1.5pt);
		\draw[color=black] (-0.3, 1.8) node {$x_2'$};
\draw [fill=black] (0.,0.6772988177744939) circle (1.5pt);
			\draw[color=black] (-0.3 , 0.6) node {$x_3'$};

\draw [fill=XONE] (-2.5360850511530666,3.960611309603055) circle (1.5pt);
		\draw[color=XONE] (-2.5,4.2) node {$x_1^1$};
\draw [fill=XONE] (-2.2664854806503656,2.6447282700785024) circle (1.5pt);
		\draw[color=XONE] (-2.0, 2.8) node {$x^1_2$};
\draw [fill=XONE] (-3.668554603677176,2.646498937444091) circle (1.5pt);
		\draw[color=XONE] (-3.8, 2.4) node {$x_3^1$};

\draw [fill=XTWO] (1.0847263776469522,3.9549348215361744) circle (1.5pt);
		\draw[color=XTWO] (1.0, 4.2) node {$x_1^2$};
\draw [fill=XTWO] (0.6288195001768166,2.4275889146310883) circle (1.5pt);
\draw [fill=XTWO] (1.0043060258130363,1.952432840026672) circle (1.5pt);
		\draw[color=XTWO] (1.2, 1.8) node {$x_3^2$};

\draw [fill=XTHREE] (3.6875927818500154,3.854100514146576) circle (1.5pt);
		\draw[color=XTHREE] (3.7, 4.1) node {$x_1^3$};
\draw [fill=XTHREE] (2.6973237672120485,2.643547349503118) circle (1.5pt);

\draw [fill=XTHREE] (4.560167514788904,2.730269972582122) circle (1.5pt);
		\draw[color=XTHREE] (4.9, 2.6) node {$x_3^3$};

\end{scriptsize}
\end{tikzpicture}       \caption{A geometric interpretation of Theorem~\ref{thm-many-cameras}.}
    \label{fig:three-cameras}
\end{figure}

{\bf Remark.}
    Here
 is a geometric way of interpreting Theorem~\ref{thm-many-cameras}.  If the points $\mathcal{X}_j $ satisfy the hypotheses of the theorem, we can find homographies that map each set to a different plane $\PP^2 \subset \PP^3$ so that the flatland cameras (potential epipoles for 3D cameras) are collinear along some line $\ell$, and the ${\mathbb P}^1$ images are not just projectively equivalent, but are actually identical, as in Figure~\ref{fig:three-cameras}.  

In this case, for any two planes ${\mathbb P}_i^2$ and ${\mathbb P}_j^2$, we can choose points $a_i', a_j' \in \ell$ that serve as camera centers allowing us to reconstruct a 3D object for which $\{x_n^i \}$ and $\{x_n^j \}$ are the images.  It is not generally the case, however, that such pairwise reconstruction gives us consistent reconstruction for three or more cameras. To observe this, suppose we reconstructed world points $\mathcal{Z}_{12}$ for the first pair of cameras. Consistency with the the $i$th camera requires that $A_i' z_j\sim x_j^i$ for all $j$. However, we only have the weaker constraint $A_i A_i' z_j \sim A_i x_j^i$, or equivalently, 
\[
\left(A_i' z_j \right)^\top F^{ik}x_j^k=0\quad\forall k=1,\ldots,m.
\]

We note that the projections $A_i A_i ' :\PP^3 \dashrightarrow \PP^1 $ are sometimes referred to as {\em radial cameras}, and the associated constraints have been used for reconstruction (see~\cite{DBLP:conf/cvpr/HrubyKDOPPL23} and the references therein).

\section{Invariant theory of labeled points in $\PP^1$}
\label{sec:invariants}

We now transition to Question II about the loci of centers of a flatland camera pair that can produce the same 
$\PP^1$-image of point sets $\X, \Y$ in $\PP^2$. In Theorem~\ref{thm:loci of centers} we saw a high-level summary of the answers. This relies on two main ingredients --- classical results from the invariant theory of labeled points in $\PP^1$, which we will introduce in this section, and algebraic geometry tools that will be described in~\Cref{sec:epv}. 

To motivate the role of invariant theory, 
observe that if $A$ is a flatland camera with center $a$, then for any $H\in \PGL(2)$, $HA$ is also a flatland camera with center $a$ and conversely any flatland camera with center $a$ is of the form $HA$.
As $H$ varies it generates a $\PGL(2)$-orbit of $\{A x_i\}$ in $\left(\PP^1\right)^n$. Since $H \in \PGL(2)$ is only defined up to scale, we may assume that $\det H =1$ and instead speak of the $\SL(2)$-orbit of $\{Ax_i\}$ in $\left(\PP^1\right)^n$.  We denote this orbit by $\ORBIT{X}{a}$ as this is the set of images of $\mathcal{X}$ by {\em all} flatland cameras with center $a$. Similarly $\ORBIT{Y}{b}$ denotes the $\SL(2)$-orbit of $\{By_i\}$ in $\left(\PP^1\right)^n$. Then we can check that the following lemma is true.

\begin{lemma}  There exists flatland cameras $A$ and $B$ with centers $a$ and $b$ respectively such that $\forall i$, $Ax_i \sim By_i$, if and only if $\ORBIT{X}{a} = \ORBIT{Y}{b}$.
\end{lemma}

The above lemma motivates the following rephrasing of Question II:
\begin{question}\label{question:q2}
    Given two sets of labeled, distinct, and generic points $\X \subset \PP^2$ and  $\Y \subset \PP^2$, 
    what is the locus of $(a,b) \in \PP^2 \times \PP^2$ such that $\ORBIT{X}{a} = \ORBIT{Y}{b}$?
\end{question}

The rest of this section is devoted to developing Theorem~\ref{thm:pull back tool}, which gives a general condition on $a$  and $b$ so that $\ORBIT{X}{a} = \ORBIT{Y}{b}$ using tools from the invariant theory of $n$ labeled points in $\PP^1$.

Spaces that parameterize orbits of an algebraic variety under a group action can be constructed using the framework of {\em Geometric Invariant Theory} (GIT)~\cite{DolgachevGIT}.
 The {\em GIT quotient} of $n$ labeled points in $\PP^1$ by the action of $\SL(2)$ is a well-known projective variety $\M := (\PP^1)^n//\SL(2)$ with (a finitely generated) homogeneous coordinate ring $R$ \cite{DolgachevGIT}.

Let $T = (t_{ij})$ be a generic (symbolic) matrix of size $2 \times n$ with $i$th column 
$$t_i = \begin{pmatrix} t_{i1}\\t_{i2} \end{pmatrix} \,\,\Rightarrow \,\,T = \begin{bmatrix} t_{11} & t_{21} & \cdots & t_{n1} \\ t_{12} & t_{22} & \cdots & t_{n2} \end{bmatrix},$$ 
and let $\CC[T]$ be the polynomial ring in the $2n$ variables $t_{ij}$ that is the algebra of all polynomial functions on $n$ labeled points in $\CC^2$, organized as the columns $t_i$ of $T$.

\begin{definition}
    A polynomial $f(T) \in \CC[T]$ is $\SL(2)$-invariant if 
$f(HT) = f(T)$ for all $H \in \SL(2)$.
\end{definition}

Let the {\em bracket} $[ij]_t$ denote the 
determinant $\det[t_i \,\,t_j]$ for all $1 \leq i < j \leq n$.
When the variables $t$ are understood, we simply write $[ij].$
Since $\det [Ht_i \,\,Ht_j] = \det(H) \det[t_i \,\, t_j] = \det[t_i \,\,t_j]$ for all $H \in \SL(2)$,  $[ij] = \det[t_i \,\,t_j]$ is $\SL(2)$-invariant. By the {\em First Fundamental Theorem} of invariant theory (see eg.~\cite[Theorem 3.2.1]{SturmfelsInvariant}),  $\CC[T]^{\SL(2)}$, the subalgebra of all $\SL(2)$-invariant polynomials in $\CC[T]$, is generated  by the bracket polynomials $[ij]$. .

The ring $R$ that we are interested in is a subalgebra of $\CC[T]^{\SL(2)}$, and hence 
its generators can be described in terms of brackets. 
The ring $\CC[T]^{\SL(2)}$ is multi-graded; assign degree $e_i + e_j$ to $[ij]$,  where $e_i$ is the $i$th standard unit vector in $\ZZ^n$. The multidegree of a monomial in  brackets (called a {\em bracket monomial}) is the sum of the multidegrees of its individual brackets. For example, 
 if $n=4$ then $[12][34]^2$ is a bracket monomial of 
 degree $(1,1,2,2)$, and 
    $[12][34]-[13][24]$ is a (homogeneous) bracket polynomial of degree 
    $(1,1,1,1)$.
Note that the multidegree of a bracket monomial records the number of times the index $i$ appears 
in the monomial, and that all bracket monomials lie in $\CC[T]^{\SL(2)}$.
Let $\ones \in \RR^n$ be the vector of all ones and 
$d \ones$ denote its scaling by $d \in \NN$. When   
$d=2$, we use the symbol 
$\twos$ for the vector $2 \cdot \ones \in \RR^n$ of all $2$'s. 

The following classical result originates from work of Kempe in 1894~\cite{Kempe1894}.
Our statement follows closely a more modern source~\cite[Theorem 2.3]{HMSVDuke2009}.

\begin{theorem} (Kempe's Theorem \cite{Kempe1894}) \label{thm:generators of R v1}
    Let $R_{d \ones} := \CC[T]^{\SL(2)}_{d \ones}$ be the vector space spanned by bracket monomials of degree $d \ones$, and $R$ be the coordinate ring of the GIT quotient $\mathcal{M} = (\PP^1)^n//\SL(2)$. 
    \begin{enumerate}
        \item  The 
        ring $R$ admits a multigrading $ R = \bigoplus_{d \in \NN} R_{d \ones}.$
        \item  The ring $R$ is generated as an algebra by the first nonzero graded piece $R_{d \ones}$; when $n$ is even, $R$ is generated by $R_\ones$ 
        and when $n$ is odd,  by $R_{\twos}$. 
    \end{enumerate}
\end{theorem}
 
Kempe introduced a graphical representation of bracket monomials 
which helps to identify a smaller generating set for $R$.  Arrange $1,\ldots, n$ on the vertices of a 
regular $n$-gon and consider any directed graph $\Gamma$ with edge set $E(\Gamma)$ and vertex set $[n]$. 
The \emph{degree} of $\Gamma$ is the ordered vector of its vertex degrees with 
no difference between in and out degrees. The bracket monomial of $\Gamma$ is:
\begin{align}
    m_\Gamma = \prod_{\overrightarrow{ij} \in E(\Gamma)} [ij].
\end{align}
Note that the degree of $\Gamma$ is the same as the multi-degree of $m_\Gamma$. 
Since the sign of a bracket monomial is unimportant in its role as a generator of $R$,  we may ignore the directions of the arrows in  $\Gamma$ giving rise to the
monomial $m_\Gamma$. However, 
when we want to write down dependences among these bracket monomials we do need 
to be careful about signs. 
If $m_\Gamma$ has degree $\ones$ then $n$ has to be even and 
$\Gamma$ is a perfect matching on $[n]$. 
By Kempe's theorem, if $n$ is even, $R$ is generated by all  $m_\Gamma$ where $\Gamma$ is a perfect matching on $[n]$. 
If $n$ is odd then there are no $m_\Gamma$'s of degree $\ones$, but there are 
$m_\Gamma$'s of degree $\twos$. This explains Part (2) of Kempe's theorem. See~\Cref{fig:six-pentagon-monomials}.

There are only two types of linear relations among the bracket monomials of degree $d \ones$. 
\begin{enumerate}
    \item If we switch the order of indices in a bracket, a bracket monomial changes sign:
\begin{align} \label{eq:relation: switch an edge direction}
    m_\Gamma = - m_{\Gamma'} \,\,\,\textup{ if } \Gamma \textup{ and } \Gamma' \textup{ differ in the direction of one edge}.
\end{align}
In the simplest case, i.e., when $n=2$,  this has the form $m_{\overrightarrow{12}} = - m_{\overrightarrow{21}}$. 

\item The second type of linear relation is the {\em Pl\"ucker relation} 
\begin{align} \label{eq:relation:generalized Plucker}
    m_{\Gamma} = m_{\Gamma_1} + m_{\Gamma_2}
\end{align}
where we identify a pair of edges $\overrightarrow{13}, \overrightarrow{24}$ in $\Gamma$ and obtain $\Gamma_1$ by replacing these edges with 
$\overrightarrow{12},\overrightarrow{34}$ and $\Gamma_2$ by replacing with 
$\overrightarrow{14},\overrightarrow{23}$. This operation ``uncrosses'' the crossing edges $\overrightarrow{13}, \overrightarrow{24}$ in $\Gamma$. When $n=4$, it has the form 
\begin{align}\label{eq:pluecker n=4}
        [12][34]+[14][23]=[13][24].
        \end{align}
\end{enumerate}

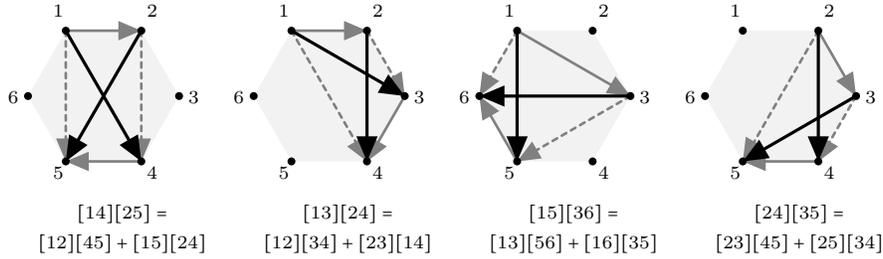
\begin{figure}[ht]
    \definecolor{yqyqyq}{rgb}{0.5019607843137255,0.5019607843137255,0.5019607843137255}
\definecolor{xdxdff}{rgb}{0.49019607843137253,0.49019607843137253,1.}
\definecolor{uuuuuu}{rgb}{0.26666666666666666,0.26666666666666666,0.26666666666666666}
\begin{tikzpicture}[line cap=round,line join=round,>=triangle 45,x=1.0cm,y=1.0cm]
\fill[line width=1.0pt,dotted,color=yqyqyq,fill=yqyqyq,fill opacity=0.10000000149011612] (0.,0.) -- (1.0042170894063625,0.) -- (1.5063256341095437,0.8696775103403789) -- (1.0042170894063625,1.7393550206807578) -- (0.,1.7393550206807582) -- (-0.5021085447031814,0.8696775103403797) -- cycle;
\fill[line width=1.0pt,dotted,color=yqyqyq,fill=yqyqyq,fill opacity=0.10000000149011612] (3.,0.) -- (4.004217089406363,0.) -- (4.506325634109544,0.8696775103403792) -- (4.004217089406363,1.7393550206807586) -- (3.,1.7393550206807589) -- (2.4978914552968186,0.86967751034038) -- cycle;
\fill[line width=1.0pt,dotted,color=yqyqyq,fill=yqyqyq,fill opacity=0.10000000149011612] (6.,0.) -- (7.004217089406362,0.) -- (7.506325634109543,0.8696775103403784) -- (7.004217089406362,1.7393550206807569) -- (6.,1.7393550206807573) -- (5.497891455296819,0.8696775103403792) -- cycle;
\fill[line width=1.0pt,dotted,color=yqyqyq,fill=yqyqyq,fill opacity=0.10000000149011612] (9.,0.) -- (10.004217089406362,0.) -- (10.506325634109542,0.8696775103403792) -- (10.00421708940636,1.7393550206807569) -- (9.,1.7393550206807564) -- (8.497891455296818,0.8696775103403784) -- cycle;

\draw [->,line width=1.0pt,color=yqyqyq](0.,1.7393550206807582) -- (1.0042170894063625,1.7393550206807578);
\draw [->,line width=1.0pt,color=yqyqyq](1.0042170894063625,0.) -- (0.,0.);
\draw [->,line width=1.0pt,color=yqyqyq](3.,1.7393550206807589) -- (4.004217089406363,1.7393550206807586);
\draw [->,line width=1.0pt,color=yqyqyq](4.506325634109544,0.8696775103403792) -- (4.004217089406363,0.);
\draw [->,line width=1.0pt,color=yqyqyq](6.,1.7393550206807573) -- (7.506325634109543,0.8696775103403784);
\draw [->,line width=1.0pt,color=yqyqyq](6.,0.) -- (5.497891455296819,0.8696775103403792);
\draw [->,line width=1.0pt,color=yqyqyq](10.00421708940636,1.7393550206807569) -- (10.506325634109542,0.8696775103403792);
\draw [->,line width=1.0pt,color=yqyqyq](10.004217089406362,0.) -- (9.,0.);
\draw [->,line width=1.0pt,dash pattern=on 2pt off 2pt,color=yqyqyq] (0.,1.7393550206807582) -- (0.,0.);
\draw [->,line width=1.0pt,dash pattern=on 2pt off 2pt,color=yqyqyq] (1.0042170894063625,1.7393550206807578) -- (1.0042170894063625,0.);
\draw [->,line width=1.0pt,dash pattern=on 2pt off 2pt,color=yqyqyq] (3.,1.7393550206807589) -- (4.004217089406363,0.);
\draw [->,line width=1.0pt,dash pattern=on 2pt off 2pt,color=yqyqyq] (4.004217089406363,1.7393550206807586) -- (4.506325634109544,0.8696775103403792);
\draw [->,line width=1.0pt,dash pattern=on 2pt off 2pt,color=yqyqyq] (6.,1.7393550206807573) -- (5.497891455296819,0.8696775103403792);
\draw [->,line width=1.0pt,dash pattern=on 2pt off 2pt,color=yqyqyq] (7.506325634109543,0.8696775103403784) -- (6.,0.);
\draw [->,line width=1.0pt,dash pattern=on 2pt off 2pt,color=yqyqyq] (10.00421708940636,1.7393550206807569) -- (9.,0.);
\draw [->,line width=1.0pt,dash pattern=on 2pt off 2pt,color=yqyqyq] (10.506325634109542,0.8696775103403792) -- (10.004217089406362,0.);

\draw [->,line width=1.2pt] (0.,1.7393550206807582) -- (1.0042170894063625,0.);
\draw [->,line width=1.2pt] (1.0042170894063625,1.7393550206807578) -- (0.,0.);
\draw [->,line width=1.2pt] (3.,1.7393550206807589) -- (4.506325634109544,0.8696775103403792);
\draw [->,line width=1.2pt] (4.004217089406363,1.7393550206807586) -- (4.004217089406363,0.);
\draw [->,line width=1.2pt] (6.,1.7393550206807573) -- (6.,0.);
\draw [->,line width=1.2pt] (7.506325634109543,0.8696775103403784) -- (5.497891455296819,0.8696775103403792);
\draw [->,line width=1.2pt] (10.00421708940636,1.7393550206807569) -- (10.004217089406362,0.);
\draw [->,line width=1.2pt] (10.506325634109542,0.8696775103403792) -- (9.,0.);

\begin{scriptsize}
\foreach \n in {0,...,3} {    
    \fill  (0.+3*\n,1.7393550206807582) circle (1.5pt);
    	\draw (-0.1+3*\n, 2.0) node {1};
    \fill (1.0042170894063625+3*\n,1.7393550206807578) circle (1.5pt);
    	\draw (1.15+3*\n, 2.0) node {2};
     \fill  (1.5063256341095437+3*\n,0.8696775103403789) circle (1.5pt);
   	 \draw (1.7+3*\n, 0.86) node {3};
    \fill  (1.0042170894063625+3*\n,0.) circle (1.5pt);
    \draw  (1.15+3*\n, -0.15) node {4};
    \fill  (0.+3*\n,0.) circle (1.5pt);
  	  \draw (-0.1+3*\n, -0.15) node {5};
   \fill (-0.5021085447031814+3*\n,0.8696775103403797) circle (1.5pt);
	    \draw (-0.7+3*\n, 0.86) node {6};
}	  

\draw (0.75, -0.7) node {$[14][25] =$};
	\draw (0.75, -1.1) node {$[12][45] + [15][24]$};
\draw (3.75, -0.7) node {$[13][24] =$};
	\draw (3.75, -1.1) node {$ [12][34] + [23][14]$};
\draw (6.75, -0.7) node {$[15][36] =$};
	\draw (6.75, -1.1) node {$[13][56] + [16][35]$};
\draw (9.75, -0.7) node {$[24][35] =$};
	\draw (9.75, -1.1) node {$[23][45] + [25][34]$};

\end{scriptsize}
\end{tikzpicture}
    \caption{These figures show several examples of the Pl\"ucker relation being used to ``uncross'' the two (bold) edges  in a bracket monomial on $6$ variables. \label{fig:plucker-relations}.}
\end{figure}

From the preceding discussion, we obtain the following refinement of Kempe's theorem.

\begin{theorem} \label{thm:even odd generators}
When $n$ is even, $R$ is generated as an algebra by the 
bracket monomials $m_\Gamma$ of degree $\ones$ where $\Gamma$ is a non-crossing perfect matching on $[n]$, and when $n$ is odd, by the 
bracket monomials $m_\Gamma$ of degree $\twos$ where $\Gamma$ has no crossing edges. (Adjacent edges are not crossing.) 
\end{theorem}

\begin{example} \label{ex:min gens}
    When $n=5$, the ring $R$ is generated by the bracket monomials $m_\Gamma$  corresponding to the $6$ non-crossing graphs of degree $(2,2,2,2,2)$ that can be drawn 
    on the $5$ vertices of a regular pentagon. See \cite[Figure 6]{HMSVDuke2009}
    or Figure~\ref{fig:six-pentagon-monomials} below.
    
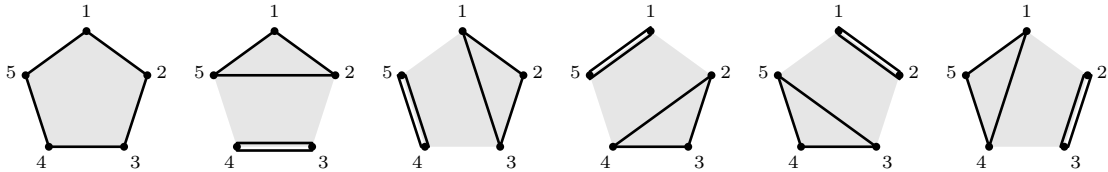
\begin{figure}[ht]
    \centering
    \begin{tikzpicture}[line cap=round,line join=round,>=triangle 45,x=1.0cm,y=1.0cm]

\foreach \n in {0,...,5} {
\fill[fill=black,fill opacity=0.1] 
	(1.+2.5*\n, 0.) -- (2.+2.5*\n, 0.) -- (2.3090169943749475+2.5*\n, 0.9510565162951532) -- (1.5+2.5*\n, 1.5388417685876266) -- (0.6909830056250527+2.5*\n, 0.9510565162951536) -- cycle;
}

\draw [line width=1.0pt] (1.,0.)-- (2.,0.) -- (2.3090169943749475,0.9510565162951532)-- (1.5,1.5388417685876266) -- (0.6909830056250527,0.9510565162951536)-- cycle;

\draw [shift={(0, 0.05)}, line width=1.0pt]  (2. + 2.5, 0.) -- (1. + 2.5, 0.)-- cycle;
\draw [shift={(0, -0.05)}, line width=1.0pt]  (2. + 2.5, 0.) -- (1. + 2.5, 0.)-- cycle;
\draw [line width=1.0pt]  (1.5 + 2.5, 1.5388417685876266) -- (0.6909830056250527 + 2.5, 0.9510565162951536) -- (2.3090169943749475 + 2.5, 0.9510565162951532)-- cycle;

\draw [line width=1.0pt] (2. + 5, 0.) -- (2.3090169943749475 + 5, 0.9510565162951532)-- (1.5 + 5, 1.5388417685876266) -- cycle;
\draw [shift={(-0.05, 0)}, line width=1.0pt] (1. + 5, 0.)-- (0.6909830056250527 + 5, 0.9510565162951536)-- cycle;
\draw [shift={(0.05,0)}, line width=1.0pt] (1. + 5, 0.)-- (0.6909830056250527 + 5, 0.9510565162951536)-- cycle;

\draw [line width=1.0pt] (1. + 7.5, 0.)-- (2. + 7.5, 0.) -- (2.3090169943749475 + 7.5, 0.9510565162951532)-- cycle;
\draw [shift={(0, 0.05)}, line width=1.0pt] (0.6909830056250527 + 7.5, 0.9510565162951536) --(1.5 + 7.5, 1.5388417685876266) -- cycle;
\draw [shift={(0, -0.05)}, line width=1.0pt] (0.6909830056250527 + 7.5, 0.9510565162951536) --(1.5 + 7.5, 1.5388417685876266) -- cycle;

\draw [line width=1.0pt] (1. + 10, 0.)-- (2. + 10, 0.) -- (0.6909830056250527 + 10, 0.9510565162951536) --cycle;
\draw [shift={(0, -0.05)}, line width=1.0pt] (1.5 + 10, 1.5388417685876266)  --(2.3090169943749475 + 10, 0.9510565162951532) -- cycle;
\draw [shift={(0, 0.05)}, line width=1.0pt] (1.5 + 10, 1.5388417685876266)  --(2.3090169943749475 + 10, 0.9510565162951532) -- cycle;

\draw [shift={(-0.05, 0)}, line width=1.0pt] (2.3090169943749475 + 12.5, 0.9510565162951532)-- (2. + 12.5, 0.) -- cycle;
\draw [shift={(0.05, 0)}, line width=1.0pt] (2.3090169943749475 + 12.5, 0.9510565162951532)-- (2. + 12.5, 0.) -- cycle;
\draw [line width=1.0pt] (1. + 12.5, 0.) -- (1.5 + 12.5, 1.5388417685876266) -- (0.6909830056250527 + 12.5, 0.9510565162951536)-- cycle;
\begin{scriptsize}
\foreach \n in {0,...,5} {
\fill   (1.5+2.5*\n,1.5388417685876266) circle (1.5pt);
      \draw (1.5+2.5*\n, 1.8) node {$1$};
\fill   (2.3090169943749475+2.5*\n,0.9510565162951532) circle (1.5pt);
      \draw (2.5+2.5*\n, 1.0) node {$2$};
\fill   (2.+2.5*\n,0.) circle (1.5pt);
      \draw (2.15+2.5*\n, -0.2) node {$3$};
\fill   (1.+2.5*\n,0.) circle (1.5pt);
      \draw (0.9+2.5*\n, -0.2) node {$4$};
\fill   (0.6909830056250527+2.5*\n,0.9510565162951536) circle (1.5pt);
      \draw (0.5+2.5*\n, 1.0) node {$5$};
}
\end{scriptsize}

\end{tikzpicture}
    \caption{The six non-crossing graphs of degree $(2,2,2,2,2)$ on the pentagon.}
    \label{fig:six-pentagon-monomials}
\end{figure}

If $n=6$ then there are $5$ non-crossing perfect matchings on the $6$ vertices of a regular hexagon which produce $5$ monomials $m_1, \ldots, m_5$ that generate $R$, see~\Cref{fig:m0-to-m6}.
When $n=8$ there are $14$ non-crossing perfect matchings on the $8$ vertices of a regular octagon. 
    See~\cite[Figure 7]{HMSVDuke2009}.
\end{example}

\begin{definition} \label{def:g vector}
Let $g = (m_1, \ldots, m_t)$ be a vector of invariants spanning $R_{\ones}$ ($n$ even) or $R_{\twos }$ ($n$ odd).
\end{definition}

The vector $g$ can be used to distinguish between two different $\SL(2)$-orbits of $(\PP^1)^n$, i.e., two different points of 
$\M$, as in the following lemma.

\begin{lemma} \label{lem:signatures of orbits}
Given two sets, $\Pcal$ and $\Qcal$, of $n$ labeled generic points in $\PP^1$, we have that
 $$g(P) := (m_1(P), \ldots, m_t(P)) \sim g(Q) := (m_1(Q), \ldots, m_t(Q))$$ for any of representatives $P$ of $\Pcal$ and 
$Q$ of $\Qcal$ if and only if $\Pcal \simeq \Qcal$, i.e. $\mathcal{P}$ and $\mathcal{Q}$ lie in the same $\SL(2)$-orbit of $(\PP^1)^n$.
\end{lemma}
\begin{proof}
It suffices to prove the result when $m_1, \ldots,  m_t$ are all the bracket monomials of lowest degree, since the conclusion will then follow for invariants with the same linear span.

Suppose $\Pcal \simeq \Qcal$. For any representatives $P$ and $Q$, there exists $H \in \SL(2)$ and scalars $\lambda_i$ with $q_i = \lambda_i H p_i$ for all $i=1, \ldots, n$. 
Set $\lambda = \lambda_1 \cdots \lambda_n.$
If $n$ is even, each $m_i$ comes from a perfect matching on $[n]$ with $n/2$ edges and therefore, $m_i(Q) = m_i(\lambda H P) = \lambda \det(H)^{\frac{n}{2}} m_i(P)$ which means that $g(P) \sim g(Q)$. If $n$ is odd, then 
each $m_i$ has degree $\twos$. Therefore, $m_i(Q) = m_i(\lambda H P) = \lambda^2 \det(H)^{n} m_i(P)$, and again, $g(P) \sim g(Q)$.

To prove the converse, suppose $g(P) \sim g(Q)$ for representatives 
$P$ and $Q$ of $\Pcal$ and $\Qcal$, where $\mathcal{P}$ and $\mathcal{Q}$ are sufficiently generic.
Let $m_{\Gamma_{1234}} $ be the bracket monomial corresponding to a graph $\Gamma_{1234}$ of minimal degree such that $\overrightarrow{12}, \overrightarrow{34} \in E(\Gamma_{1234})$, and $\overrightarrow{13}, \overrightarrow{24} \not \in E (\Gamma_{1234} ).$ 
Let $\Gamma_{1234} '$ be the graph obtained from $\Gamma_{1234} $ by replacing the edges $\overrightarrow{12}, \overrightarrow{34}$ with $\overrightarrow{13}, \overrightarrow{24}$.
Then $$m_{1234} := m_{\Gamma_{1234} } / m_{\Gamma_{1234} '} = \displaystyle\frac{[12][34]}{[13][24]}$$ is the cross-ratio of points $1,\ldots ,4 $, which (by genericity) we may assume is defined and nonzero on $\mathcal{P}$ and $\mathcal{Q}$.
Note that $m_{1234}|_{\mathcal{P}} = m_{1234}|_{\mathcal{Q}}$, since $g(P) \sim g(Q).$
By genericity, there exists a  homography $H\in \PGL_2$ such that $H p_i \sim q_i$ for $i=1,\ldots , 4$---see~\eqref{eq:Hx}~\eqref{eq:Hy}.

Similarly, from a suitably chosen graph $\Gamma_{123i},$ construct the cross ratio $m_{123i}$ where $i>4,$ and observe that $m_{123i}|_\mathcal{P}=m_{123i}|_{\mathcal{Q}}$. 
Since the cross-ratio is invariant under homography, we have 
\[
\displaystyle\frac{\det \begin{bmatrix} q_1 & q_2 \end{bmatrix} \det \begin{bmatrix} q_3 & H p_i \end{bmatrix}}{\det \begin{bmatrix} q_1 & q_3 \end{bmatrix} \det \begin{bmatrix} q_2 & H p_i \end{bmatrix}}
=
\displaystyle\frac{\det \begin{bmatrix} Hp_1 & H p_2 \end{bmatrix} \det \begin{bmatrix} H p_3 & H p_i \end{bmatrix}}{\det \begin{bmatrix} H p_1 & H p_3 \end{bmatrix} \det \begin{bmatrix} H p_2 & H p_i \end{bmatrix}}
=
\displaystyle\frac{\det \begin{bmatrix} q_1 & q_2 \end{bmatrix} \det \begin{bmatrix} q_3 & q_i \end{bmatrix}}{\det \begin{bmatrix} q_1 & q_3 \end{bmatrix} \det \begin{bmatrix} q_2 & q_i \end{bmatrix}},
\]
where the last inequality follows by hypothesis. This rearranges to
\[
\displaystyle\frac{\det \begin{bmatrix} q_2 & q_i \end{bmatrix} \det \begin{bmatrix} q_3 & H p_i \end{bmatrix}}{\det \begin{bmatrix} q_3 & q_i \end{bmatrix} \det \begin{bmatrix} q_2 & H p_i \end{bmatrix}} = 1.
\]
Since the cross ratio of four points equals $1$ only when two points coincide, we deduce that $H p_i \sim q_i$ for all $i>4$, and hence $\mathcal{P} \simeq \mathcal{Q}.$
\end{proof}

We now introduce the tool that will allow us to 
describe the locus of $(a,b)$ by {\em pulling back} brackets from $\PP^1$ to $\PP^2$. 
For clarity, index a 
bracket with $p$ (respectively, $x$) to denote its evaluation at $\Pcal$ (respectively, $\X$). 

\begin{lemma} \label{lem:pullback}
Given a labeled set $\X = \{x_1,\ldots, x_n \} \subset \PP^2$ and a flatland camera $A:\PP^2 \dashrightarrow \PP^1$ with center $a$, let $p_i:=Ax_i$. Then for any collection of two-element subsets   $\{i_1,j_1\}, \{i_2,j_2\}, \ldots,\{i_m,j_m\} \subset[n]$, the following equality holds:
\begin{equation}\label{eq:invariant theory projection}
([i_1 \,\, j_1]_p,\ldots,[i_m \,\, j_m]_p)\sim([i_1 \,\, j_1 \, \,a]_x,\ldots,[i_m \,\, j_m \, \,a]_x).
\end{equation}
\end{lemma}
\begin{proof}
 It suffices to prove that \eqref{eq:invariant theory projection} holds for $m=2$, and the brackets $[12]$ and $[34]$. For ease of notation, fix an affine representation $(u_i,v_i,w_i)\sim x_i$ for $i=1,2,3,4$. We can assume up to change of coordinates on $\PP^2$ and $\PP^1$ that
\begin{equation}
    A=\begin{bmatrix}1 & 0 & 0\\ 0 & 1 & 0\end{bmatrix}
\end{equation}
which means that $Ax_i \sim (u_i,v_i)^\top$.
Then calculate that
\begin{equation}
([12a]_x,[34a]_x)\sim \left(\det \begin{bmatrix}u_1 & u_2 & 0\\ v_1 & v_2 & 0\\ w_1 & w_2 & 1\end{bmatrix}, \det \begin{bmatrix}u_3 & u_4 & 0\\ v_3 & v_4 & 0\\ w_3& w_4 & 1\end{bmatrix} \right)\sim([12]_p,[34]_p).
\end{equation}
\end{proof}

\begin{definition} \label{def:g' vector} 
Recall the vector $g = (m_1, \ldots, m_t)$ from 
Definition~\ref{def:g vector}. 
If there is a flatland camera $\PP^2 \dashrightarrow \PP^1$ with center $a$ that maps $x_i \in \X \mapsto p_i$, then let $g'$ denote the vector obtained from $g$ by replacing every bracket $[ij]$ in each monomial $m_k$ in $g$ with the bracket $[ija]$ where $i,j$ now correspond to $x_i$ and $x_j$. Denote the evaluation of $g'$ at $\X, a$ by 
$g'(\X,a)$.
\end{definition}

From Lemma~\ref{lem:signatures of orbits} and Lemma~\ref{lem:pullback} we obtain the following central tool for answering Question~\ref{question:q2}.

\begin{theorem}\label{thm:pull back tool}
$\ORBIT{X}{a} = \ORBIT{Y}{b}$ if and only if $g'(\X, a) \sim g'(\Y,b)$.
\end{theorem}

The equality $g'(\X, a) \sim g'(\Y,b)$ will give us the equations 
for the centers $(a,b)$ of cameras that can send $\X$ and $\Y$ to the same $\PP^1$-image. 

\section{The camera centers variety}\label{sec:epv}

In this section, we introduce the camera centers variety, a second tool that casts~\Cref{question:q2} in the language of algebraic geometry. 
To make convenient use of this language, we temporarily work in this section over the field of complex numbers; we note, however, that the explicit nature of our geometric constructions in the following sections give us many of the same results over the reals.
As we will see, the tools from invariant theory and the camera centers variety complement each other in our characterization of the loci of camera centers. 

It will be helpful to mark separately two projective planes: $\PP^2_x$ containing $\mathcal{X}$ and the center $a$ of the flatland camera $A$, and $\PP_y^2$ containing $\mathcal{Y}$ and the center $b$ of the flatland camera $B$. Informally, the camera centers variety $\EPV{n}$ consists of all $(a,b)\in \PP_x^2 \times \PP_y^2$ consistent with two sets of $n$ labeled points.
Importantly, this variety comes equipped with projections into $\PP^2_x$ and $\PP^2_y$:
$$
\begin{tikzcd}
 & \EPV{n} \ni (a,b)  \arrow{dr}{\pi_y} \arrow{ld}[swap]{\pi_x}  \\
a \in \PP^2_x && b \in \PP^2_y
\end{tikzcd}
$$
These projections allow us to answer 
questions about the locus of flatland camera centers. 
For example, if we fix $a \in \PP_x^2$, then $\pi_x^{-1}(a)$ (the fiber over $a$) is the locus of corresponding $b\in \PP_y^2$. 

Recall from~\Cref{thm:fundamental-equivalence} that the centers of projection $a$ and $b$ of $A$ and $B$ span, respectively, the right and left nullspace of the fundamental matrix $F$.
Since $F$ is determined only up to scale, we regard it as a point in the projective space of all nonzero $3\times 3$ matrices $\PP(\CC^{3 \times 3}) \cong \PP^8.$

Fix two sets of labeled points $\X $ and $\Y$.
To understand how $\X$ and $\Y$ constrain the projections in~\Cref{thm:fundamental-equivalence}, let us consider the subvariety  
$\FE{\X ,\Y} \subset  \PP^8 \times \PP_x^2 \times \PP_y^2$ of all $(F,a,b)$ such that
\begin{equation}\label{eq:epv-equations}
Fa =0,
 \quad 
 Fb = 0,
 \quad 
 \forall i \, \, 
    y_i^\top F x_i = 0.
\end{equation}
We define the \emph{camera centers variety} $\EPV{\X ,\Y} \subset \PP_x^2 \times \PP_y^2$ associated to the pair $(\X, \Y)$, to be the image of $\FE{\X ,\Y}$ under the coordinate projection $ \PP^8 \times \PP_x^2 \times \PP_y^2 \dashrightarrow \PP_x^2 \times \PP_y^2$. 

Now fixing $n$, but varying the points in $\X$ and $\Y$, we may think of $\FE{\X ,\Y}$ and $\EPV{\X ,\Y} $ as defining \emph{families} of varieties.
More formally, they arise as the fibers of certain projection maps.
For example, the varieties $\FE{\X ,\Y}$ arise as the fibers of a map
\begin{align}
\pi_{\FE{}} : V_n  &\to (\PP_x^2 \times \PP_y^2)^n \nonumber \\
 (F, a, b, x_1, y_1, \ldots , x_n , y_n ) &\mapsto (x_1, y_1, \ldots , x_n, y_n), \label{eq:FE-projection}
\end{align}
where $V_n \subset \PP^8 \times  (\PP_x^2 \times \PP_y^2)^{n+1}$ is a variety where equations~\eqref{eq:epv-equations} are all satisfied.
For technical reasons, we will define $V_n$ so as to ensure that it is \emph{irreducible}.
To do so, we define 
\[
\mathcal{U} = \{ (F, x_1, \ldots , x_n) \in \PP^8 \times \left( \PP_x^2 \right)^n \mid \rank (F) =2, \,  \forall i \,  F x_i \ne 0 \}.
\]
For any fixed $(F, x_1, \ldots , x_n) \in \mathcal{U},$ equations~\eqref{eq:epv-equations} define a product of projective linear spaces in $(a,b,y)$ of dimension $n.$
We define $V_n$ to be the Zariski closure of the set of all $(F, a, b, x, y) \in  \PP^8 \times  (\PP_x^2 \times \PP_y^2)^{n+1}$ such that $(F, x_1, \ldots , x_n) \in \mathcal{U}$ and equations~\eqref{eq:epv-equations}  hold.

Let $\mathcal{I} (V_n)$ denote the vanishing ideal of the variety $V_n.$ 
We may regard $\mathcal{I} (V_n)$ as a prime ideal in the polynomial ring  $S[F, a,b],$ where $S = \CC [ x_1, y_1, \ldots , x_n, y_n].$ 
Let $G = \{ g_1, \ldots , g_s \} $ be a Gr\"{o}bner basis for this ideal with respect to some monomial order $<$ with  $\operatorname{LC}_<(g_i)$ the leading term of $g_i$.
Set $u = \operatorname{LC}_< (g_1) \, \cdots  \, \operatorname{LC}_< (g_s),$ and consider the multiplicatively-closed set 
\begin{equation}\label{eq:generic-freeness}
U = \{ u^k \}_{k\ge 1} \subset S.
\end{equation}
If $n\le 7$, then~\Cref{lemma:dim-En} implies that the map $\pi_{\FE{}}$ is surjective. 
The \emph{generic freeness lemma}, as stated in~\cite[Lemma 10.1]{Kemper2011}, implies that the localized coordinate ring $U^{-1} (S[F,a,b] / \mathcal{I} (V_n)) $ is a free module over the localized polynomial ring $U^{-1} \CC [ x_1, y_1, \ldots , x_n, y_n]$.
This has the following practical consequence: for any specific values of $\X $ and $\Y$ such that $u(\X , \Y ) \ne 0,$ a Gr\"{o}bner basis for $\mathcal{I}(\FE{\X ,\Y})$ can be obtained by specializing $G$ to these specific values. 

In a similar manner, the camera centers varieties $\EPV{\X ,\Y}$ form a family over $\left(\PP_x^2 \times \PP_y^2 \right)^n.$
Generic freeness then implies that several important invariants of these varieties are \emph{constant} over a dense Zariski-open set $\mathcal{U}' \subset (\PP^2 \times \PP^2)^n$.
Specifically, we may take $\mathcal{U}'=\{ (\X , \Y ) \mid u (\X , \Y ) \ne 0 \}$ with $u$ as in~\eqref{eq:generic-freeness}.
For any $(\X, \Y ) \in \mathcal{U}'$, the following invariants of $\EPV{\X,\Y}$ depend only on $n$:
\begin{enumerate}
    \item $\dim (\EPV{\X, \Y})$ (see~\Cref{lemma:dim-En}.)
    \item the \emph{multidegrees} of $\EPV{\X, \Y} \subset \PP_x^2 \times \PP_y^2$ (see~\Cref{def:multidegrees}), and
    \item the bigraded pieces of $\mathcal{I} (\EPV{\X,\Y})$, the vanishing ideal of $\EPV{\X,\Y}$, i.e., for any $j,k \in \ZZ ,$ the dimensions of the vector spaces
    \[
\{
f (a, b) \mid \deg_a (f) = j, \, \, \deg_b (f) = k, \, \, f(a,b) = 0 \, \, \forall (a,b) \in \EPV{\X, \Y}\}.
    \]
\end{enumerate}
This discussion justifies the following abuse of notation: we will write $\FE{n}$ and $\EPV{n}$ to denote any one of the varieties $\FE{\X , \Y}$ and $\EPV{\X , \Y}$ associated to sufficiently generic sets of $n$ points  $\X, \Y$.  

Our next result is a formula for $\dim (\EPV{n}).$

\begin{lemma}\label{lemma:dim-En}
For $n \le 8,$ $\dim \FE{n} = \dim \EPV{n}= \min \left( 4, 7-n\right)$. 
For $n\ge 8,$ $\FE{n} = \EPV{n} = \emptyset .$
\end{lemma}
For the proof of~\Cref{lemma:dim-En}, it is helpful to observe that we can use homographies to fix four points in each of the sets $\X$ and $\Y$, assuming the points in each set are sufficiently generic. For
\begin{equation}\label{eq:standard-position}
\X = \{ x_1, x_2, x_3, x_4, \ldots \} ,
\quad 
\Y = \{ y_1, y_2, y_3, y_4, \ldots \},
\end{equation}
we say that $\X $ and $\Y$ are in \emph{standard position} if $x_i \sim e_i$ for $i=1,\ldots , 4,$
where $e_1,e_2,e_3\in \PP^2$ are represented by the standard basis vectors and $e_4\sim e_1 + e_2 + e_3$.
When $\X$ and $\Y$ are generic, they may be transformed into standard position via homographies $H_x, H_y : \PP^2 \to \PP^2$ defined by
\begin{small}
\begin{align}
H_x = H(x_1, \ldots , x_4) &= 
\left(
\begin{bmatrix}
\det \begin{bmatrix}
x_4  & x_2 & x_3
\end{bmatrix}
& 0 & 0 \\
0 & 
\det \begin{bmatrix}
x_1  & x_4 & x_3
\end{bmatrix}
& 0 \\
0 & 0 &
\det \begin{bmatrix}
x_1  & x_2 & x_4
\end{bmatrix}
\end{bmatrix}
\begin{bmatrix}
x_1 & x_2 & x_3 
\end{bmatrix}
\right)^{-1} \label{eq:Hx},\\
H_y = H(y_1, \ldots , y_4) &= 
\left(
\begin{bmatrix}
\det \begin{bmatrix}
y_4  & y_2 & y_3
\end{bmatrix}
& 0 & 0 \\
0 & 
\det \begin{bmatrix}
y_1  & y_4 & y_3
\end{bmatrix}
& 0 \\
0 & 0 &
\det \begin{bmatrix}
y_1  & y_2 & y_4
\end{bmatrix}
\end{bmatrix}
\begin{bmatrix}
y_1 & y_2 & y_3 
\end{bmatrix}
\right)^{-1}. \label{eq:Hy}
\end{align}
\end{small}
If we assume generic data $\X $ and $\Y$, we can always reduce to the case where $\X$ and $\Y$ are in standard position.
Several of our formulas in subsequent sections assume standard position; general formulas may be recovered using the substitutions
\begin{equation}\label{eq:substitutions}
x_i \gets (H_x)^{-1} x_i,
\quad 
y_i \gets (H_y)^{-1} y_i. 
\end{equation}

\begin{proof}[Proof of ~\Cref{lemma:dim-En}]
We first treat the case of $3\le n \le 7.$
We observe that a fundamental matrix is uniquely determined by its left and right epipoles together with $n\ge 3$ generic point pairs.
Indeed, since the equations~\eqref{eq:epv-equations} are linear in $F,$ an explicit formula for $F$ may be computed when $n=3$ using Cramer's rule. 
For $n=3,$ if $x_i = y_i = e_i$ for $i=1,\ldots , 3,$ we may check that
\begin{equation}\label{eq:ab-to-F}
F = \diag (a) [ a \odot b ]_\times \diag (b),
\end{equation}
where $a \odot b$ denotes the Hadamard (entrywise) product and $[a \odot b]_\times $ denotes the $3\times 3$ skew-symmetric matrix representing the cross-product as the linear transformation $x \mapsto (a \odot b) \times x.$
We note that~\eqref{eq:ab-to-F} is a matrix of rank $2$ for general $a$ and $b.$
For generic $(x_1,y_1, \ldots , x_n, y_n),$ the expression~\eqref{eq:ab-to-F} may be transformed by homographies sending $x_i, y_i \to e_i$; indeed, we may use the explicit homography formulas~\eqref{eq:Hx} and~\eqref{eq:Hy} after choosing generic $x_4, y_4 \in \PP^2.$
In the other direction, $a$ and $b$ may be recovered from $F$ uniquely---explicit formulas follow again from Cramer's rule.
We conclude
\[
\dim (\FE{n}) = \dim (\EPV{n}) \quad \forall n \ge 3.
\]
Now, using the map~\eqref{eq:FE-projection}, since $n\le 7$, we have that
\[
\dim (\FE{n}) = \dim (V_n ) - \dim \left( (\PP^2 \times \PP^2)^n \right)
= 3n + 7 - 4n = 7 - n.
\]
Now, for $n\le 3,$ we have 
\[
4 = \dim (\EPV{3}) \le \dim (\EPV{n}) \le \dim (\PP^2 \times \PP^2) = 4.
\]
It remains to note that $\FE{n}$ and $\EPV{n}$ are empty for $n\ge 8$.
Our proof thus far implies that $\FE{7}$ consists of finitely many points $(F,a,b)$.
For generic $(x_8, y_8) \in \PP^2 \times \PP^2$ we have that $y_8^T F x_8 \ne 0$ for all such $F$. This proves emptiness for all $n\ge 8.$
\end{proof}

\begin{corollary}
\label{cor:n=8 loci invariants}
    Let $\X = \{x_1,\ldots, x_n\} \subset \PP^2_x$ and $\Y = \{y_1, \ldots, y_n\} \subset \PP^2_y$ be two sets of generic labeled points with $n  \ge 8$, then there does not exist a camera pair $(A,B)$ that will project $\mathcal{X}$ and $\mathcal{Y}$ to the same $\PP^1$ image.
\end{corollary}
\begin{proof}
    The proof follows from~\Cref{lemma:dim-En} where it was shown that $n\ge 8,$ $\FE{n} = \EPV{n} = \emptyset .$
\end{proof}
Having determined the dimension of $\EPV{n}$, we briefly recall the next-most important invariant of any variety embedded in a product of projective space: its set of \emph{multidegrees}.

\begin{definition}\label{def:multidegrees}
    The multidegree $d_{i,j}$ of $\EPV{n}\subset \PP_x^2 \times \PP_y^2$ is the number of complex points of intersection of $\EPV{n}$ with $i$ generic hyperplanes (lines) in 
    $\PP^2_x$ and $j$ generic hyperplanes (lines) in $\PP^2_y$ where $i+j = \dim{\EPV{n}}$.
    The multiset of all multidegrees of $\EPV{n}$ is $ \left\{ d_{(i,j)} \,:\, i + j = \dim(\EPV{n}), \, \, d_{(i,j)} > 0 \right\}$.
\end{definition}

The multidegrees of $\EPV{n}$ when $4\le n \le 7$ will further inform our study of camera centers.

Overall, the camera centers variety provides a useful tool that is complementary to the invariant-theoretic tools of the preceding section.
We may now begin to answer~\Cref{question:q2}. 

\section{Camera loci when $n=4$}\label{sec:n-4}
The first non-trivial case for Question~\ref{question:q2} is that of $n=4$. Indeed, since any set of $n \leq 3$ points 
in $\PP^1$ can be sent to any other set of $n \leq 3$ points by a homography, for any pair of cameras $A$ and $B$ we always have 
a homography $H$ such that $\forall i$, 
$HAx_i \sim By_i$.

\begin{theorem} \label{thm:n=4 loci invariants}
    Let $\X = \{x_1,\ldots, x_4\} \subset \PP^2_x$ and $\Y = \{y_1, \ldots, y_4\} \subset \PP^2_y$ be two sets of generic labeled points that can be imaged by flatland cameras to $\Pcal \simeq \Qcal$ in $\PP^1$. Then the first 
    camera center $a \in \PP^2_x$  can be chosen to be any point in 
    $\PP^2_x \setminus \X$. Having fixed $a$, the other camera center $b$ is any point different from $y_1, \dots, y_4$ on the unique conic $\omega$ containing $\Y$ with conic cross-ratio 
    $$(y_1,y_2;y_3,y_4)_\omega = (x_1,x_2;x_3,x_4;a).$$
\end{theorem}

We now define the cross-ratios needed in the statement and proof of Theorem~\ref{thm:n=4 loci invariants}.

\begin{definition}
Setting $[ij] = \det[p_i \,\, p_j]$, the {\bf cross-ratio} of  $p_1,\ldots,p_4\in\PP^1$ is
\begin{equation} \label{eq:cross-ratio}
(p_1,p_2;p_3,p_4):=\frac{[13][24]}{[14][23]}.
\end{equation}
\end{definition}

The cross-ratio is the only projective invariant of 
$4$ points in $\PP^1$ in the following sense. The proof 
is an immediate consequence of Lemma~\ref{lem:signatures of orbits}.

\begin{lemma} \label{lem:cr=homography}
If $\Pcal = \{ p_1,\ldots,p_4 \} \subset \PP^1$ and $\Qcal = \{ q_1,\ldots,q_4 \} \subset \PP^1$ are two labeled sets of points then $(p_1,p_2;p_3,p_4)=(q_1,q_2;q_3,q_4)$ if and only if $\Pcal \simeq \Qcal$. 
\end{lemma}

Permutations of points changes the cross-ratio systematically; for all $\sigma\in S_4$,
\begin{equation}
(p_1,p_2;p_3,p_4)=(q_1,q_2;q_3,q_4) \,\,\,\Leftrightarrow \,\,\, (p_{\sigma(1)},p_{\sigma(2)};p_{\sigma(3)},p_{\sigma(4)})=(q_{\sigma(1)},q_{\sigma(2)};q_{\sigma(3)},q_{\sigma(4)}).
\end{equation}
Thus, we  do not need to consider multiple orderings.
The cross-ratio is $0,1,\infty$, or $\frac{0}{0}$ if and only if the points are not distinct. See \cite{semple-kneebone}[III, \S 4-5].

\begin{definition}
The {\bf planar cross-ratio} of $x_1,\ldots,x_5\in\PP^2$, or  {\bf cross-ratio around} $x_5$, is
\begin{equation} \label{eq:5 point cross-ratio}
 (x_1,x_2;x_3,x_4;x_5) :=\frac{[135][245]}{[145][235]}.
\end{equation}
\end{definition}

Note that planar cross-ratios are preserved under a homography of $\PP^2$.  For $5$ distinct points, the cross-ratio around $x_5$ can be obtained geometrically by drawing the $4$  lines $\overline{x_ix_5}$ for $i=1,2,3,4$, cutting them with a transversal, and computing the cross-ratio of the intersection points. 
A planar cross-ratio is transformed by permutations of $x_1, \ldots,x_4$ similarly to the usual cross-ratio.

\begin{lemma}\label{lem: conic formula}\cite[Ex. 3.4.3]{SturmfelsInvariant}
A collection of $6$ points $x_1,\ldots,x_6 \in \PP^2$ lie on a conic if and only if 
\begin{equation} \label{eq:conic condition}
[135][245][146][236]=[136][246][145][235], 
\end{equation}
or generically, the following cross-ratio equality holds: 
\begin{equation}
(x_1,x_2;x_3,x_4;x_5)=(x_1,x_2;x_3,x_4;x_6).
\end{equation}
\end{lemma}

Lemma~\ref{lem: conic formula} makes the following definition of a conic cross-ratio well-defined. 

\begin{definition}\label{def:conic cross-ratio}
The {\bf conic cross-ratio}
of $4$ points $x_1,\ldots,x_4$ on a (non-degenerate) conic $\omega$  is
\begin{equation}
(x_1,x_2;x_3,x_4)_\omega:=(x_1,x_2;x_3,x_4;x)
\end{equation}
where $x \in \omega$ is any new point.
\end{definition}

\begin{proof}[Proof of Theorem~\ref{thm:n=4 loci invariants}]
For $n=4$ choose the vector of invariants $g=([13][24],[14][23])$ as in Definition~\ref{def:g vector}. Note that these are not non-crossing matchings but is still a linearly independent set of generators of $R_\ones$.  
By Theorem~\ref{thm:pull back tool} and the definition of planar cross-ratios, 
\begin{align} \label{eq:n=4 pullback}
\begin{split}
\Pcal \simeq \Qcal & \,\,\Leftrightarrow \,\,([13a]_x[24a]_x,[14a]_x[23 a]_x)\sim([13b]_y[24b]_y,[14b]_y[23b]_y)\\
& \,\,\Leftrightarrow \,\,(x_1,x_2;x_3,x_4;a)=(y_1,y_2;y_3,y_4;b).
\end{split}
\end{align}

Picking $a \neq x_i$ arbitrarily, we obtain a unique conic $\omega_x$ passing through $a,x_1, \ldots, x_4$ with cross-ratio $\lambda = (x_1,x_2;x_3,x_4;a)$.
Equation~\ref{eq:n=4 pullback} says that $b$ must lie on  the 
unique conic $\omega_y$ containing $y_1, \ldots, y_4$ with cross-ratio
$\lambda = (y_1,y_2;y_3,y_4)_{\omega_y}$. 
Concretely, once $a \in \PP^2_x$ is fixed,  $\omega_y$ is given by the following quadratic equation in the entries of $b$:
\begin{align}
    \frac{[13b]_y[24b]_y}{[14b]_y[23b]_y} = \frac{[13 {a}]_x[24 {a}]_x}{[14 {a}]_x[23 {a}]_x}.
\end{align}
\end{proof}

\definecolor{uuuuuu}{rgb}{0.26666666666666666,0.26666666666666666,0.26666666666666666}
\definecolor{BLUE}{rgb}{0.,0.,1.}
\definecolor{CANVAS}{rgb}{0.5, 0.5, 0.5} 
\definecolor{GRAY}{rgb}{0.5019607843137255,0.5019607843137255,0.5019607843137255}
\definecolor{GREEN}{rgb}{0.,0.7,0.}
\definecolor{YELLOW}{rgb}{1,0.8,0.}
\definecolor{ORANGE}{rgb}{1.,0.4980392156862745,0.}
\definecolor{RED}{rgb}{1.,0.,0.}
\def\ONE#1{\textcolor{RED}{#1}}
\def\TWO#1{\textcolor{ORANGE}{#1}}
\def\THREE#1{\textcolor{BLUE}{#1}}
\def\FOUR#1{\textcolor{GREEN}{#1}}
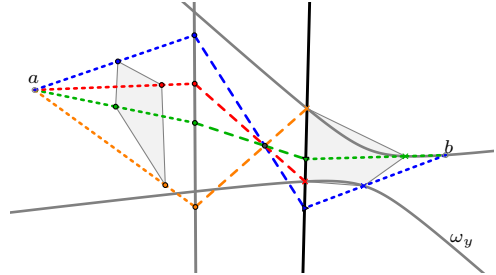
\begin{figure}[h]
\begin{multicols}2
\begin{enumerate}
\item 
Begin with $\X = \{x_1, x_2, x_3, x_4\}$ and\\
$\Y = \{ y_1, y_2, y_3, y_4\}$ in $\PP^2.$  \\
(In the diagram to the right, we connect each set of points by a shaded quadrangle to aid with visualization).  

 \item 
 Choose an arbitrary $a$ and image line $\ell_x$, and let $\ell_y := \ONE{y_1}\TWO{y_2}$. 

\item
Let ${x}_i^\prime$ be the projection of $x_i$ on $\ell_x$ through $a$. The points ${x}_i^\prime$ have a fixed cross-ratio $\lambda$. 

\item
Construct the center of perspectivity $O = ( \ONE{x_1^\prime y_1})\cap( \TWO{x_2^\prime y_2})$.
\end{enumerate}

\vfill \null
\columnbreak

\begin{tikzpicture}[scale =0.8, line cap=round,line join=round,>=triangle 45,x=1.0cm,y=1.0cm]
	\clip(-2.7 ,-0 ) rectangle (8, 6.0);

\fill[line width=1.0pt,color=GRAY,fill=GRAY,fill opacity=0.1] 
	(0.64,4.16) -- (-0.3305456708301174,4.6825528903906175) -- (-0.38,3.68) -- (0.72,1.94) -- cycle;
	\draw [color=GRAY] (0.64,4.16) -- (-0.3305456708301174,4.6825528903906175) -- (-0.38,3.68) -- (0.72,1.94) -- cycle;
\fill[line width=1.0pt,color=GRAY,fill=GRAY,fill opacity=0.1] 
	(3.82,2.02) -- (5.118376078924618,1.9227529675481412) -- (6.04,2.58) -- (3.86,3.64) -- cycle;
	\draw [color=GRAY] (3.82,2.02) -- (5.118376078924618,1.9227529675481412) -- (6.04,2.58) -- (3.86,3.64) -- cycle;

\draw [line width=1.1pt,domain=-2.708101659730916:13.48256455875219] plot(\x,{(--6.1076-1.62*\x)/-0.04});
\draw [line width=1.1pt,color=GRAY,domain= 1 : 2] 
	plot(\x,{(--13.390682369111895-9.597559871781217*\x)/0.05563802824220998});

\draw [line width=1.0pt,dotted,color=ORANGE] (-2.177011390146186,4.04952866062554)-- (0.72,1.94);
\draw [line width=1.0pt,dotted,color=ORANGE] (0.72,1.94)-- (1.3867857082862747,1.4544639172252556);
\draw [line width=1.0pt,dotted,color=GREEN] (-2.177011390146186,4.04952866062554)-- (-0.38,3.68);
\draw [line width=1.0pt,dotted,color=GREEN] (-0.38,3.68)-- (1.3759773287705541,3.318909383686984);
\draw [line width=1.0pt,dotted,color=RED] (-2.177011390146186,4.04952866062554)-- (0.64,4.16);
\draw [line width=1.0pt,dotted,color=RED] (0.64,4.16)-- (1.3709352718220307,4.188664207307314);
\draw [line width=1.0pt,dotted,color=BLUE] (-2.177011390146186,4.04952866062554)-- (-0.3305456708301174,4.6825528903906175);
\draw [line width=1.0pt,dotted,color=BLUE] (-0.3305456708301174,4.6825528903906175)-- (1.3647029754047957,5.263735339280401);

\draw [line width=1.0pt,dash pattern=on 3pt off 3pt,color=RED] (1.3709352718220307,4.188664207307314)-- (2.924305870239852,2.813143512115842);
\draw [line width=1.0pt,dash pattern=on 3pt off 3pt,color=RED] (2.924305870239852,2.813143512115842)-- (3.82,2.02);
\draw [line width=1.0pt,dash pattern=on 3pt off 3pt,color=ORANGE] (1.3867857082862747,1.4544639172252556)-- (2.924305870239852,2.813143512115842);
\draw [line width=1.0pt,dash pattern=on 3pt off 3pt,color=ORANGE] (2.924305870239852,2.813143512115842)-- (3.86,3.64);

\begin{scriptsize}
	\draw [fill=BLUE] (-0.3305456708301174,4.6825528903906175) circle (1.5 pt);
		\draw [color=BLUE] (-0.33, 4.9) node {$x_3$};
    \draw [fill=RED] (0.64,4.16) circle (1.5pt);
    \draw [fill=ORANGE] (0.72,1.94) circle (1.5pt);
    \draw [fill=GREEN] (-0.38,3.68) circle (1.5pt);
   \draw [color=RED] (3.82,2.02)-- ++(-1.5pt,-1.5pt) -- ++(3.0pt,3.0pt) ++(-3.0pt,0) -- ++(3.0pt,-3.0pt);
    \draw [color=BLUE] (5.118376078924618,1.9227529675481412)-- ++(-1.5pt,-1.5pt) -- ++(3.0pt,3.0pt) ++(-3.0pt,0) -- ++(3.0pt,-3.0pt);
    	\draw [color=BLUE] (5.2, 1.7) node {$y_3$};    
    \draw [color=GREEN] (6.04,2.58)-- ++(-1.5pt,-1.5pt) -- ++(3.0pt,3.0pt) ++(-3.0pt,0) -- ++(3.0pt,-3.0pt);
     	\draw [color=GREEN] (6.0, 2.8) node {$y_4$};
   \draw [color=ORANGE] (3.86,3.64)-- ++(-1.5pt,-1.5pt) -- ++(3.0pt,3.0pt) ++(-3.0pt,0) -- ++(3.0pt,-3.0pt);
    \draw [color=GRAY] (-2.177011390146186,4.04952866062554) circle (1.5pt);
    	\draw (-2.2, 4.3) node {$a$};
    \draw [fill=BLUE] (1.3647029754047957,5.263735339280401) circle (1.5pt);
    	\draw [color = BLUE]  (1.65, 5.3) node {$x_3^\prime $};
    \draw [fill=RED] (1.3709352718220307,4.188664207307314) circle (1.5pt);
    \draw [fill=GREEN] (1.3759773287705541,3.318909383686984) circle (1.5pt);
		\draw [color=GREEN] (1.65, 3.3) node {$x_4^\prime $};
     \draw [fill=ORANGE] (1.3867857082862747,1.4544639172252556) circle (1.5pt);
    
    \draw [fill=uuuuuu] (2.924305870239852,2.813143512115842) circle (1.5pt);
    	\draw (2.9, 3.1) node {$O$};
	\draw [color=RED] (0.56, 4.5) node {$ x_1$};
   \draw [color=RED] (1.65, 4.3) node {$ x_1^\prime $};
	\draw [color=ORANGE] (1.65, 1.4) node {$x_2^\prime $};
	\draw [color=ORANGE] (0.6, 1.7) node {$ x_2$}; 
    \draw [color=GRAY] (1.7, 0.3) node {$\ell_x$};
   \draw [color=GRAY] (4.1, 0.3) node {$\ell_y$};
	\draw [color=GREEN] (-0.5, 3.5) node {$ x_4$};
	   \draw [color=RED] (4.1, 1.8) node {$ y_1$};
	\draw [color=ORANGE] (4.2,3.8) node {$ y_2$};
	
\end{scriptsize}
\end{tikzpicture}

\end{multicols}

\begin{multicols}2
\begin{enumerate}
\setcounter{enumi}{4}
\item Let $\ONE{y_1^\prime} =\ONE{y_1}$ and $\TWO{y_2^\prime}  = \TWO{y_2}$.  Use  $O$ to locate $\THREE{y_3^\prime} , \FOUR{y_4^\prime} \in \ell_y$.  By construction, the sets $\{ x_i^\prime \}, \{y_i^\prime \}$ are projectively equivalent. Hence, 
the cross-ratio of the points $\{y_i^\prime \}$ is also $\lambda$.

\item 
Construct  $b = (\THREE{y_3^\prime y_3}) \cap ( \FOUR{y_4^\prime y_4}) $.  
Then $y_i^\prime $ is the projection of $y_i$ through $b$  for all $i$. 
\end{enumerate}
\vfill \null
\columnbreak

\begin{tikzpicture}[scale =0.8, line cap=round,line join=round,>=triangle 45,x=1.0cm,y=1.0cm]
	\clip(-2.7 ,0.5 ) rectangle (8, 6.0);

\fill[line width=1.0pt,color=GRAY,fill=GRAY,fill opacity=0.1] 
	(0.64,4.16) -- (-0.3305456708301174,4.6825528903906175) -- (-0.38,3.68) -- (0.72,1.94) -- cycle;
	\draw [color=GRAY] (0.64,4.16) -- (-0.3305456708301174,4.6825528903906175) -- (-0.38,3.68) -- (0.72,1.94) -- cycle;
\fill[line width=1.0pt,color=GRAY,fill=GRAY,fill opacity=0.1] 
	(3.82,2.02) -- (5.118376078924618,1.9227529675481412) -- (6.04,2.58) -- (3.86,3.64) -- cycle;
	\draw [color=GRAY] (3.82,2.02) -- (5.118376078924618,1.9227529675481412) -- (6.04,2.58) -- (3.86,3.64) -- cycle;

\draw [line width=1.1pt,domain=-2.708101659730916:13.48256455875219] plot(\x,{(--6.1076-1.62*\x)/-0.04});
\draw [line width=1.1pt,color=CANVAS,domain= 1 : 2] 
	plot(\x,{(--13.390682369111895-9.597559871781217*\x)/0.05563802824220998});


\draw [line width=1.0pt,dotted,color=ORANGE] (-2.177011390146186,4.04952866062554)-- (0.72,1.94);
\draw [line width=1.0pt,dotted,color=ORANGE] (0.72,1.94)-- (1.3867857082862747,1.4544639172252556);
\draw [line width=1.0pt,dotted,color=GREEN] (-2.177011390146186,4.04952866062554)-- (-0.38,3.68);
\draw [line width=1.0pt,dotted,color=GREEN] (-0.38,3.68)-- (1.3759773287705541,3.318909383686984);
\draw [line width=1.0pt,dotted,color=RED] (-2.177011390146186,4.04952866062554)-- (0.64,4.16);
\draw [line width=1.0pt,dotted,color=RED] (0.64,4.16)-- (1.3709352718220307,4.188664207307314);
\draw [line width=1.0pt,dotted,color=BLUE] (-2.177011390146186,4.04952866062554)-- (-0.3305456708301174,4.6825528903906175);
\draw [line width=1.0pt,dotted,color=BLUE] (-0.3305456708301174,4.6825528903906175)-- (1.3647029754047957,5.263735339280401);
\draw [line width=1.0pt,dash pattern=on 3pt off 3pt,color=BLUE] (1.3647029754047957,5.263735339280401)-- (2.924305870239852,2.813143512115842);
\draw [line width=1.0pt,dash pattern=on 3pt off 3pt,color=BLUE] (2.924305870239852,2.813143512115842)-- (3.805399695736176,1.4286876773151478);
\draw [line width=1.0pt,dash pattern=on 3pt off 3pt,color=RED] (1.3709352718220307,4.188664207307314)-- (2.924305870239852,2.813143512115842);
\draw [line width=1.0pt,dash pattern=on 3pt off 3pt,color=RED] (2.924305870239852,2.813143512115842)-- (3.82,2.02);
\draw [line width=1.0pt,dash pattern=on 3pt off 3pt,color=ORANGE] (1.3867857082862747,1.4544639172252556)-- (2.924305870239852,2.813143512115842);
\draw [line width=1.0pt,dash pattern=on 3pt off 3pt,color=ORANGE] (2.924305870239852,2.813143512115842)-- (3.86,3.64);
\draw [line width=1.0pt,dash pattern=on 3pt off 3pt,color=GREEN] (1.3759773287705541,3.318909383686984)-- (2.924305870239852,2.813143512115842);
\draw [line width=1.0pt,dash pattern=on 3pt off 3pt,color=GREEN] (2.924305870239852,2.813143512115842)-- (3.8322606795241616,2.516557520728565);
\draw [line width=1.0pt,dotted,color=GREEN] (3.8322606795241616,2.516557520728565)-- (6.04,2.58);
\draw [line width=1.0pt,dotted,color=GREEN] (6.04,2.58)-- (6.933219615781861,2.6056679157876603) ;

\draw [line width=1.0pt,dotted,color=BLUE] (6.933219615781861,2.6056679157876603) -- (5.118376078924618,1.9227529675481412);
\draw [line width=1.0pt,dotted,color=BLUE] (5.118376078924618,1.9227529675481412)-- (3.805399695736176,1.4286876773151478);

\begin{scriptsize}
		\draw [fill=blue] (-0.3305456708301174,4.6825528903906175) circle (1.5 pt);
    \draw [fill=RED] (0.64,4.16) circle (1.5pt);
    \draw [fill=ORANGE] (0.72,1.94) circle (1.5pt);
    \draw [fill=GREEN] (-0.38,3.68) circle (1.5pt);
    \draw [color=RED] (3.82,2.02)-- ++(-1.5pt,-1.5pt) -- ++(3.0pt,3.0pt) ++(-3.0pt,0) -- ++(3.0pt,-3.0pt);
    \draw [color=BLUE] (5.118376078924618,1.9227529675481412)-- ++(-1.5pt,-1.5pt) -- ++(3.0pt,3.0pt) ++(-3.0pt,0) -- ++(3.0pt,-3.0pt);
    
    \draw [color=GREEN] (6.04,2.58)-- ++(-1.5pt,-1.5pt) -- ++(3.0pt,3.0pt) ++(-3.0pt,0) -- ++(3.0pt,-3.0pt);
    \draw [color=ORANGE] (3.86,3.64)-- ++(-1.5pt,-1.5pt) -- ++(3.0pt,3.0pt) ++(-3.0pt,0) -- ++(3.0pt,-3.0pt);
    \draw [fill=CANVAS] (1.3534744019504117,7.200664260161612) ++(-1.5pt,0 pt) -- ++(1.5pt,1.5pt)--++(1.5pt,-1.5pt)--++(-1.5pt,-1.5pt)--++(-1.5pt,1.5pt);
    \draw [fill=CANVAS] (1.4091124301926217,-2.3968956116196045) ++(-1.5pt,0 pt) -- ++(1.5pt,1.5pt)--++(1.5pt,-1.5pt)--++(-1.5pt,-1.5pt)--++(-1.5pt,1.5pt);
    \draw [color=GRAY] (-2.177011390146186,4.04952866062554) circle (1.5pt);
    \draw [fill=BLUE] (1.3647029754047957,5.263735339280401) circle (1.5pt);
    \draw [fill=RED] (1.3709352718220307,4.188664207307314) circle (1.5pt);
    \draw [fill=GREEN] (1.3759773287705541,3.318909383686984) circle (1.5pt);
    \draw [fill=ORANGE] (1.3867857082862747,1.4544639172252556) circle (1.5pt);
    
    \draw [fill=uuuuuu] (2.924305870239852,2.813143512115842) circle (1.5pt);
    \draw [fill=BLUE,shift={(3.805399695736176,1.4286876773151478)},rotate=270] (0,0) ++(0 pt,2.25pt) -- ++(1.9485571585149868pt,-3.375pt)--++(-3.8971143170299736pt,0 pt) -- ++(1.9485571585149868pt,3.375pt);
    \draw [fill=GREEN,shift={(3.8322606795241616,2.516557520728565)},rotate=180] (0,0) ++(0 pt,2.25pt) -- ++(1.9485571585149868pt,-3.375pt)--++(-3.8971143170299736pt,0 pt) -- ++(1.9485571585149868pt,3.375pt);
    	\draw [color=GRAY] (6.933219615781861,2.6056679157876603) circle (1.5pt);
	\draw (7, 2.8) node {$b$};

    	\draw [color = BLUE]  (1.65, 5.3) node {$x_3^\prime $};
    	\draw [color = BLUE]  (4.1, 1.3) node {$y_3^\prime $};
		\draw [color=GREEN] (1.65, 3.0) node {$x_4^\prime $};
		\draw [color=GREEN] (4.1, 2.8) node {$y_4^\prime $};
    	\draw [color=BLUE] (5.2, 1.7) node {$y_3$};    
     	\draw [color=GREEN] (6.0, 2.8) node {$y_4$};
  	\draw (3.0, 3.1) node {$O$};
    	\draw (-2.2, 4.3) node {$a$};

\end{scriptsize}
\end{tikzpicture}

\end{multicols}

\begin{multicols}2
\begin{enumerate}
\setcounter{enumi}{6}
\item
Let $\omega_y$ be the conic through  $\ONE{y_1}, \TWO{y_2}, \THREE{y_3}, \FOUR{y_4}, b$. 
By construction, $\omega_y$ has conic cross-ratio $\lambda$, and  $\lambda = (\ONE{y_1}, \TWO{y_2}; \THREE{y_3}, \FOUR{y_4}; b_y) = (\ONE{y_1}, \TWO{y_2}; \THREE{y_3}, \FOUR{y_4}; b)  $ for all $b_y \in \omega_y\setminus\{y_i\}$.  Hence, $\omega_y$ is the locus of all possible camera centers $b$ corresponding to $a$.
\end{enumerate}
\vfill \null
\columnbreak
\begin{tikzpicture}[scale =0.6, line cap=round,line join=round,>=triangle 45,x=1.0cm,y=1.0cm]
	\clip(-2.7 ,-0 ) rectangle (8, 6.0);

\fill[line width=1.0pt,color=GRAY,fill=GRAY,fill opacity=0.1] 
	(0.64,4.16) -- (-0.3305456708301174,4.6825528903906175) -- (-0.38,3.68) -- (0.72,1.94) -- cycle;
	\draw [color=GRAY] (0.64,4.16) -- (-0.3305456708301174,4.6825528903906175) -- (-0.38,3.68) -- (0.72,1.94) -- cycle;
\fill[line width=1.0pt,color=GRAY,fill=GRAY,fill opacity=0.1] 
	(3.82,2.02) -- (5.118376078924618,1.9227529675481412) -- (6.04,2.58) -- (3.86,3.64) -- cycle;
	\draw [color=GRAY] (3.82,2.02) -- (5.118376078924618,1.9227529675481412) -- (6.04,2.58) -- (3.86,3.64) -- cycle;

\draw [line width=1.1pt,domain=-2.708101659730916:13.48256455875219] plot(\x,{(--6.1076-1.62*\x)/-0.04});
\draw [line width=1.1pt,color=CANVAS,domain= 1 : 2] 
	plot(\x,{(--13.390682369111895-9.597559871781217*\x)/0.05563802824220998});

\draw [samples=50,domain=-0.99:0.99,rotate around={-107.37587681600097:(5.232723278254678,2.304345571199109)},xshift=5.232723278254678cm,yshift=2.304345571199109cm,line width=1.0pt,color=GRAY] plot ({0.39832169154928954*(1+(\x)^2)/(1-(\x)^2)},{0.8868623380858361*2*(\x)/(1-(\x)^2)});
\draw [samples=50,domain=-0.99:0.99,rotate around={-107.37587681600097:(5.232723278254678,2.304345571199109)},xshift=5.232723278254678cm,yshift=2.304345571199109cm,line width=1.0pt,color=GRAY] plot ({0.39832169154928954*(-1-(\x)^2)/(1-(\x)^2)},{0.8868623380858361*(-2)*(\x)/(1-(\x)^2)});

\draw [line width=1.0pt,dotted,color=ORANGE] (-2.177011390146186,4.04952866062554)-- (0.72,1.94);
\draw [line width=1.0pt,dotted,color=ORANGE] (0.72,1.94)-- (1.3867857082862747,1.4544639172252556);
\draw [line width=1.0pt,dotted,color=GREEN] (-2.177011390146186,4.04952866062554)-- (-0.38,3.68);
\draw [line width=1.0pt,dotted,color=GREEN] (-0.38,3.68)-- (1.3759773287705541,3.318909383686984);
\draw [line width=1.0pt,dotted,color=RED] (-2.177011390146186,4.04952866062554)-- (0.64,4.16);
\draw [line width=1.0pt,dotted,color=RED] (0.64,4.16)-- (1.3709352718220307,4.188664207307314);
\draw [line width=1.0pt,dotted,color=BLUE] (-2.177011390146186,4.04952866062554)-- (-0.3305456708301174,4.6825528903906175);
\draw [line width=1.0pt,dotted,color=BLUE] (-0.3305456708301174,4.6825528903906175)-- (1.3647029754047957,5.263735339280401);
\draw [line width=1.0pt,dash pattern=on 3pt off 3pt,color=BLUE] (1.3647029754047957,5.263735339280401)-- (2.924305870239852,2.813143512115842);
\draw [line width=1.0pt,dash pattern=on 3pt off 3pt,color=BLUE] (2.924305870239852,2.813143512115842)-- (3.805399695736176,1.4286876773151478);
\draw [line width=1.0pt,dash pattern=on 3pt off 3pt,color=RED] (1.3709352718220307,4.188664207307314)-- (2.924305870239852,2.813143512115842);
\draw [line width=1.0pt,dash pattern=on 3pt off 3pt,color=RED] (2.924305870239852,2.813143512115842)-- (3.82,2.02);
\draw [line width=1.0pt,dash pattern=on 3pt off 3pt,color=ORANGE] (1.3867857082862747,1.4544639172252556)-- (2.924305870239852,2.813143512115842);
\draw [line width=1.0pt,dash pattern=on 3pt off 3pt,color=ORANGE] (2.924305870239852,2.813143512115842)-- (3.86,3.64);
\draw [line width=1.0pt,dash pattern=on 3pt off 3pt,color=GREEN] (1.3759773287705541,3.318909383686984)-- (2.924305870239852,2.813143512115842);
\draw [line width=1.0pt,dash pattern=on 3pt off 3pt,color=GREEN] (2.924305870239852,2.813143512115842)-- (3.8322606795241616,2.516557520728565);
\draw [line width=1.0pt,dotted,color=GREEN] (3.8322606795241616,2.516557520728565)-- (6.04,2.58);
\draw [line width=1.0pt,dotted,color=GREEN] (6.04,2.58)-- (6.933219615781861,2.6056679157876603) ;

\draw [line width=1.0pt,dotted,color=BLUE] (6.933219615781861,2.6056679157876603) -- (5.118376078924618,1.9227529675481412);
\draw [line width=1.0pt,dotted,color=BLUE] (5.118376078924618,1.9227529675481412)-- (3.805399695736176,1.4286876773151478);

\begin{scriptsize}
		\draw [fill=blue] (-0.3305456708301174,4.6825528903906175) circle (1.5 pt);
    \draw [fill=RED] (0.64,4.16) circle (1.5pt);
    \draw [fill=ORANGE] (0.72,1.94) circle (1.5pt);
    \draw [fill=GREEN] (-0.38,3.68) circle (1.5pt);
    \draw [color=RED] (3.82,2.02)-- ++(-1.5pt,-1.5pt) -- ++(3.0pt,3.0pt) ++(-3.0pt,0) -- ++(3.0pt,-3.0pt);
    \draw [color=BLUE] (5.118376078924618,1.9227529675481412)-- ++(-1.5pt,-1.5pt) -- ++(3.0pt,3.0pt) ++(-3.0pt,0) -- ++(3.0pt,-3.0pt);
    
    \draw [color=GREEN] (6.04,2.58)-- ++(-1.5pt,-1.5pt) -- ++(3.0pt,3.0pt) ++(-3.0pt,0) -- ++(3.0pt,-3.0pt);
    \draw [color=ORANGE] (3.86,3.64)-- ++(-1.5pt,-1.5pt) -- ++(3.0pt,3.0pt) ++(-3.0pt,0) -- ++(3.0pt,-3.0pt);
    \draw [fill=CANVAS] (1.3534744019504117,7.200664260161612) ++(-1.5pt,0 pt) -- ++(1.5pt,1.5pt)--++(1.5pt,-1.5pt)--++(-1.5pt,-1.5pt)--++(-1.5pt,1.5pt);
    \draw [fill=CANVAS] (1.4091124301926217,-2.3968956116196045) ++(-1.5pt,0 pt) -- ++(1.5pt,1.5pt)--++(1.5pt,-1.5pt)--++(-1.5pt,-1.5pt)--++(-1.5pt,1.5pt);
    \draw [color=GRAY] (-2.177011390146186,4.04952866062554) circle (1.5pt);
    \draw [fill=BLUE] (1.3647029754047957,5.263735339280401) circle (1.5pt);
    \draw [fill=RED] (1.3709352718220307,4.188664207307314) circle (1.5pt);
    \draw [fill=GREEN] (1.3759773287705541,3.318909383686984) circle (1.5pt);
    \draw [fill=ORANGE] (1.3867857082862747,1.4544639172252556) circle (1.5pt);
    
    \draw [fill=uuuuuu] (2.924305870239852,2.813143512115842) circle (1.5pt);
    \draw [fill=BLUE,shift={(3.805399695736176,1.4286876773151478)},rotate=270] (0,0) ++(0 pt,2.25pt) -- ++(1.9485571585149868pt,-3.375pt)--++(-3.8971143170299736pt,0 pt) -- ++(1.9485571585149868pt,3.375pt);
    \draw [fill=GREEN,shift={(3.8322606795241616,2.516557520728565)},rotate=180] (0,0) ++(0 pt,2.25pt) -- ++(1.9485571585149868pt,-3.375pt)--++(-3.8971143170299736pt,0 pt) -- ++(1.9485571585149868pt,3.375pt);
    	\draw [color=GRAY] (6.933219615781861,2.6056679157876603) circle (1.5pt);
	\draw (7, 2.8) node {$b$};
    	\draw (-2.2, 4.3) node {$a$};

	\draw (7.3, 0.7) node {$\omega_y$};
	
\end{scriptsize}
\end{tikzpicture}

\end{multicols}
    \caption{A geometric construction that operationalizes Theorem~\ref{thm:n=4 loci invariants}.}
    \label{fig:hyperbola-construction}
\end{figure}

\begin{figure}[h]
    \centering
    \definecolor{BLUE}{rgb}{0.,0.,1.}
\definecolor{aqaqaq}{rgb}{0.6274509803921569,0.6274509803921569,0.6274509803921569}
\definecolor{yqyqyq}{rgb}{0.5019607843137255,0.5019607843137255,0.5019607843137255}
\definecolor{ffqqqq}{rgb}{1.,0.,0.}
\definecolor{GREEN}{rgb}{0.,1.,0.}
\definecolor{ududff}{rgb}{0.30196078431372547,0.30196078431372547,1.}
\definecolor{ORANGE}{rgb}{1.,0.4980392156862745,0.}
\begin{tikzpicture}[line cap=round,line join=round,>=triangle 45,x=1.0cm,y=1.0cm]

\clip(-2.096982880756687,0.852102429262165) rectangle (3.506732433047883,4.417152656887827);

\draw[line width=1.0 pt,color=yqyqyq,fill=yqyqyq,fill opacity=0.10000000149011612] (2.733085169574865,2.4516704199563777) -- (1.7503440511091373,2.117120251968046) -- (1.4576126541193462,2.723492431446897) -- (2.063984833598199,2.744401816946168) -- cycle;
\draw[line width=1.0 pt,color=yqyqyq,fill=yqyqyq,fill opacity=0.10000000149011612] (-0.5287789683113776,2.4203063417074717) -- (-0.6646899740566383,1.615294999985547) -- (-0.9678760637960646,2.1903031012154925) -- (-0.9574213710464292,2.6084908112009075) -- cycle;

\draw [line width=1.0pt,color=aqaqaq,domain=-2.096982880756687:3.506732433047883] plot(\x,{(-1.6811075060628742--0.6063721794788512*\x)/-0.2927313969897911});
\draw [line width=1.0pt,color=yqyqyq,domain=-2.096982880756687:3.506732433047883] plot(\x,{(-0.46034173288833813-0.37817959045636407*\x)/-0.17784753556849364});
\draw [line width=1.0pt,dotted,color=ffqqqq] (-0.0897325197451726,2.397597042643703)-- (1.7503440511091373,2.117120251968046);
\draw [line width=1.0pt,dotted,color=ORANGE] (-0.5145315049120799,1.4942935151409031)-- (1.4576126541193462,2.723492431446897);
\draw [line width=1.0pt,dotted,color=BLUE] (0.08811501582332104,2.775776633100067)-- (2.245067743691548,1.0923354601901984);
\draw [line width=1.0pt,dotted,color=GREEN] (-0.3051443564916225,1.9395397525057043)-- (1.5493740449032818,2.533415264823032);

\draw [line width=1.0pt,dotted,color=GREEN] (2.971802950302511,3.1166004155882816)-- (1.5493740449032818,2.533415264823032);
\draw [line width=1.0pt,dotted,color=BLUE] (2.971802950302511,3.1166004155882816)-- (2.245067743691548,1.0923354601901984);
\draw [line width=1.0pt,dash pattern=on 1pt off 1pt on 3pt off 4pt,color=yqyqyq,domain=-2.096982880756687:3.506732433047883] plot(\x,{(-0.7493426263374849-3.117032358820299*\x)/-0.7385067564687608});
\draw [line width=1.0pt, dash pattern=on 2pt off 2pt,color=BLUE] (-1.7415233272690835,2.483034498205283)-- (0.43022033512646574,2.8305134841885704);
\draw [line width=1.0pt, dash pattern=on 2pt off 2pt,color=ffqqqq] (-1.7415233272690835,2.483034498205283)-- (0.3225980396218014,2.3762695999178245);
\draw [line width=1.0pt, dash pattern=on 2pt off 2pt,color=GREEN] (-1.7415233272690835,2.483034498205283)-- (0.17599243986354907,1.7574879917226667);
\draw [line width=1.0pt, dash pattern=on 2pt off 2pt,color=ORANGE] (-1.7415233272690835,2.483034498205283)-- (0.012931014183288377,1.0692509026659955);
\draw [line width=1.2pt,dotted,color=ORANGE] (1.4576126541193462,2.723492431446897)-- (2.971802950302511,3.1166004155882816);

\begin{scriptsize}
\draw [fill=ORANGE] (-0.6646899740566383,1.615294999985547) circle (1.5pt);
\draw [fill=ududff] (-0.9574213710464292,2.6084908112009075) circle (1.5pt);
\draw [fill=GREEN] (-0.9678760637960646,2.1903031012154925) circle (1.5pt);
\draw [fill=ffqqqq] (-0.5287789683113776,2.4203063417074717) circle (1.5pt);
\draw [fill=ududff] (2.733085169574865,2.4516704199563777) circle (1.5pt);
\draw [fill=ffqqqq] (1.7503440511091373,2.117120251968046) circle (1.5pt);
\draw [fill=ORANGE] (1.4576126541193462,2.723492431446897) circle (1.5pt);
\draw[color=ORANGE] (1.6196603917386947,3.005769135687052) node {$y_2$};
\draw [fill=GREEN] (2.063984833598199,2.744401816946168) circle (1.5pt);
\draw [fill=black] (-1.7415233272690835,2.483034498205283) circle (1.5pt);
\draw[color=black] (-1.668340478021636,2.634627543074996) node {$a$};
\draw [fill=BLUE] (0.08811501582332104,2.775776633100067) circle (1.5pt);
\draw[color=BLUE] (0.0,3.0998613704337705) node {$x_3^\prime $};
\draw [fill=ffqqqq] (-0.0897325197451726,2.397597042643703) circle (1.5pt);
\draw [fill=ORANGE] (-0.5145315049120799,1.4942935151409031) circle (1.5pt);
\draw[color=ORANGE] (-0.4,1.1) node {$ x_2^\prime $};
\draw [fill=GREEN] (-0.3051443564916225,1.9395397525057043) circle (1.5pt);

\draw [color=yqyqyq] (0.7334320182892102,2.272124792476036) circle (1.5pt);

\draw [fill=GREEN] (1.5493740449032818,2.533415264823032) circle (1.5pt);
\draw [fill=BLUE] (2.245067743691548,1.0923354601901984) circle (1.5pt);
\draw[color=BLUE] (2.5605827392058798,1.3748370667439338) node {$y_3^\prime $};
\draw (2.971802950302511,3.1166004155882816) circle (2pt);
	\draw[color=black] (3.1094541085617378,3.3769107282991078) node {$b$};
\draw [fill=black] (0.7514377706520492,4.186283261486294) circle (1.5pt);
\draw[color=black] (0.913968631138306,4.192376762770667) node {$L$};
\draw [fill=ORANGE] (0.012931014183288377,1.0692509026659955) ++(-1.5pt,0 pt) -- ++(1.5pt,1.5pt)--++(1.5pt,-1.5pt)--++(-1.5pt,-1.5pt)--++(-1.5pt,1.5pt);
\draw[color=ORANGE] (0.36509726178244784,1.1448338262519555) node {$\bar x_2$};
\draw [fill=BLUE] (0.43022033512646574,2.8305134841885704) ++(-1.5pt,0 pt) -- ++(1.5pt,1.5pt)--++(1.5pt,-1.5pt)--++(-1.5pt,-1.5pt)--++(-1.5pt,1.5pt);
\draw[color=BLUE] (0.7728302790182281,2.911676900940334) node {$\bar x_3$};
\draw [fill=GREEN] (0.17599243986354907,1.7574879917226667) ++(-1.5pt,0 pt) -- ++(1.5pt,1.5pt)--++(1.5pt,-1.5pt)--++(-1.5pt,-1.5pt)--++(-1.5pt,1.5pt);
\draw[color=GREEN] (0.5114629602773433,1.7512060057308072) node {$\bar x_4$};
\draw [fill=ffqqqq] (0.3225980396218014,2.3762695999178245) ++(-1.5pt,0 pt) -- ++(1.5pt,1.5pt)--++(1.5pt,-1.5pt)--++(-1.5pt,-1.5pt)--++(-1.5pt,1.5pt);
\end{scriptsize}
\end{tikzpicture}
    \caption{If $\ell_x\cap\bar\ell_x = L\in \ell_y$, then $\ell_x$ and $\bar\ell_x$ both construct the same point $b$.}
    \label{fig:four-points-move-ell}
\end{figure}
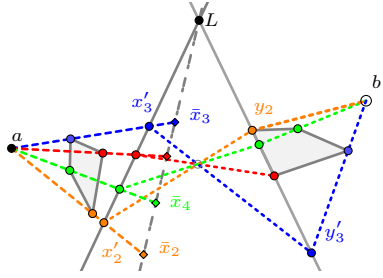

The proof of~\Cref{thm:n=4 loci invariants} can be operationalized by the  geometric construction in Figure~\ref{fig:hyperbola-construction}.

We conclude this section with a description of the camera centers variety $\EPV{4}$. 

\begin{theorem} \label{thm:main thm n=4}
    The camera centers variety $\EPV{4}$ is a hypersurface in $\PP^2 \times \PP^2$. 
    Furthermore,
    \begin{enumerate}
         \item When $\X$ and $\Y$ are in standard position, the equation of $\EPV{4}$ is 
        \begin{equation}\label{eq:E4-equation}
a^T [a \odot b]_\times b = 0.
        \end{equation}
        \item   $\EPV{4}$ projects surjectively onto both $\PP^2$ factors, i.e.,   $\pi_x(\EPV{4}) = \pi_y(\EPV{4})=\PP^2$.
         \item $\EPV{4}$ has bidegree $(2,2)$, ie.~$d_{(2,1)} (\EPV{4}) =d_{(1,2)} ( \EPV{4}) =2$. 
         \item The fibers $\pi_x^{-1}(a) $ and $\pi_y^{-1} (b)$ project onto conics in $\PP^2_y$ and $\PP^2_x,$ respectively.
     \end{enumerate}
\end{theorem}

\begin{remark}
The conic $\omega$ in~\Cref{thm:n=4 loci invariants} is exactly the projection of $\pi_x^{-1} (a)$ into $\PP_y^2.$
\end{remark}

\begin{proof}
\Cref{lemma:dim-En} immediately gives $\dim (\EPV{4}) = 3$, so $\EPV{3} \subset \PP^2 \times \PP^2$ is a hypersurface.
Recall that $e_4 = (1,1,1)$. 
Using~\eqref{eq:ab-to-F}, and that $\X$ and $\Y$ are in standard position, we know that $0  = e_4^\top F e_4 = e_4^T \diag (a) [a\odot  b ]_{\times } \diag (b) e_4 = a^T [a\odot  b ]_{\times } b$. This gives part (1). 
For part (2), we argue that the projection $\pi_x : \EPV{4} \to \PP^2$ defined by $\pi_x (a,b) = a$ is surjective, with a similar argument in the case of $\pi_y.$
If we fix any $a$ in~\eqref{eq:E4-equation}, we obtain a nonzero quadratic equation in $b$.
This immediately gives (4); furthermore, letting $b$ be any point on this conic, we have $\pi_x (a,b) = a,$ proving surjectivity.
Finally, for part (3), let us note that slicing the variety $\EPV{4}$ with two hyperplanes in $a$ and one hyperplane in $b$ is the same as fixing $a$ and slicing the resulting conic in $b$ with a line.
This shows that $d_{(2,1)} (\EPV{4}) = 2$, and a similar argument gives $d_{(1,2)} (\EPV{4}) = 2$.
\end{proof}

\begin{remark}
Although~\Cref{thm:main thm n=4} states that $\EPV{4}$ is three-dimensional, the construction in~\Cref{fig:hyperbola-construction} has four degrees of freedom; two coming from the choice of $a,$ and another two from the choice of $\ell_x.$
This apparent discrepancy may be resolved as follows; the construction of $b$ in terms of $a$ and $\ell_x$ gives a dominant rational map
\begin{align*}
\beta  : \PP^2 \times \left(\PP^2\right)^\ast &\dashrightarrow \EPV{4}, \,\,\,\,\,\,
(a  ,  \ell_x) \mapsto (a, \, b(a, \ell_x)).
\end{align*}
Fixing $(a,b) \in \EPV{4},$ the fiber $\beta^{-1} (a,b)$ is exactly the set of all pairs $(a, \ell_x)$ that construct $b.$
Using the fiber-dimension theorem,
\[
\dim \left( \beta^{-1} (a,b)\right) = \dim \left( \PP^2 \times (\PP^2)^\ast \right) - \dim \left( \EPV{4} \right) = 4- 3 = 1.
\]
Thus, for fixed $a,$ there is a $1$-parameter family of lines $\ell_x$ such that $(a,\ell_x)$ construct the same $b.$
In fact, given the intersection point $L=\ell_x\cap\ell_y$, this family consists of all lines $\bar{\ell}_x$ such that $L=\bar{ \ell}_x\cap\ell_y$.
Why is this? Consider a homography that fixes points on $\ell_y$ and sends $ x^\prime \in\ell_x$ to $\bar x\in\bar \ell_x$ via a perspectivity centered at $a$ (Figure~\ref{fig:four-points-move-ell}). 
This homography preserves incidences; since the four lines $ x_i^\prime y_i^\prime  $ are coincident at the point $O= (\ONE{x_1^\prime y_1})\cap(\TWO{x_2^\prime y_2})$ determined in step (4) of 
Figure~\ref{fig:hyperbola-construction}, the four lines $\bar x_i y_i^\prime  $ must be coincident at some $\bar O = (\ONE{\bar x_1y_1})\cap(\TWO{\bar x_2y_2})$, and so the line $\bar \ell_x$ constructs the same points $y_i^\prime \in\ell_y$, and thus the same camera center $b$.
\end{remark}
\section{Camera loci when $n=5$}\label{sec:5-points}

We now answer Question~\ref{question:q2} when $\X$ and $\Y$ consist of 5 generic labeled points in $\PP^2$. 

\begin{theorem} \label{thm:n=5 loci invariants}
     Let $\X = \{x_1,\ldots, x_5\} \subset \PP^2_x$ and $\Y = \{y_1, \ldots, y_5\} \subset \PP^2_y$ be two sets of generic labeled points that can be imaged by flatland cameras to $\Pcal \simeq \Qcal$ in $\PP^1$. Then the first camera center $a \in \PP^2_x$ can be any point other than $x_1, \ldots, x_5$. Having chosen $a$, the other camera center $b$ is the unique point of intersection of $5$ conics $\omega_y^i$ passing through $\Y \setminus \{y_i\}$.
\end{theorem}
For $5$ labeled points in $\PP^1$, 
the generators of $R$ come from the 
 graded piece $R_\twos$ which is spanned by the $6$ bracket monomials (c.f. Figure~\ref{fig:six-pentagon-monomials}):
 \begin{align}
 \begin{split}
     m_1 = [12][23][34][45][15], \quad  m_2 = [12][25][15][34]^2, \quad m_3 = [12][23][13][45]^2, \\
     m_4 = [23][34][24][15]^2, \quad m_5 = [34][45][35][12]^2, \quad m_6 = [14][45][15][23]^2,
 \end{split}
 \end{align}
 and hence, we may choose $g = (m_1,m_2,m_3,m_4,m_5,m_6)$.

 \begin{lemma} \label{lem:4 wise crs}
     If $\Pcal^i := \Pcal \setminus \{p_i\}$ and $\Qcal^i := \Qcal \setminus \{q_i\}$, then $\Pcal = \{p_1, \ldots, p_5\} \simeq \Qcal = \{q_1, \ldots, q_5\}$ in $\PP^1$ 
     if and only if the $4$-point cross-ratios of $\Pcal^i$ and $\Qcal^i$ coincide for all $i=1,\ldots,5$.
 \end{lemma}

 \begin{proof}
By Lemma~\ref{lem:signatures of orbits}, $\Pcal \simeq \Qcal$ if and only if $g(\Pcal) = g(\Qcal)$. 
Note that $m_2/m_1 = [25][34]/[23][45]$ is the inverse of the 
cross-ratio of points indexed $2,3,4,5$, $m_3/m_1$ is  
the cross-ratio of points indexed $1,3,4,5$, $m_4/m_1$ 
is the inverse of the cross-ratio of the points indexed $1,2,4,5$, 
$m_5/m_1$ is the cross-ratio of the points indexed 
$1,2,3,5$ and finally, $m_6/m_1$ is the inverse of the cross-ratio 
of the points $1,2,3,4$. This proves the statement.
\end{proof}

\begin{proof}[Proof of~\Cref{thm:n=5 loci invariants}]
      Since $\Pcal \simeq \Qcal$, by Lemma~\ref{lem:4 wise crs},  the $4$-point cross-ratio of $\Pcal \setminus \{p_i\}$ coincides 
     with the $4$-point cross-ratio of $\Qcal \setminus \{q_i\}$ for $i=1, \ldots, 5$. Pick the first camera center $a$ arbitrarily in $\PP^2_x \setminus \X$, and let $\omega_x^i$ be the conic passing through $\X^i := \X \setminus \{x_i\}$ and $a$, and suppose its 
     conic cross-ratio with respect to $\X^i$ is $\lambda_i$. Note that $a$ lies on the intersection of the $5$ conics $\omega_x^i$.
     
     For $i=1,\ldots,5$, let $\omega_y^i$ be the unique conic through $\Y^i := \Y \setminus \{y_i\}$ with conic cross-ratio $\lambda_i$ with respect to $\Y^i$. Since the camera $A$ sends $\X^i$ to $\Pcal^i$ and $B$ sends $\Y^i$ to $\Qcal^i$, and $\Pcal^i \simeq \Qcal^i$, we can apply  Theorem~\ref{thm:n=4 loci invariants} to get that $b$ must lie on each of the conics $\omega_y^i$.
     The conics $\omega_y^1$ and $\omega_y^2$ contain the common points $y_3,y_4,y_5$ and hence, generically, must intersect in an additional point $b \neq y_3,y_4,y_5$. Since the point $b$ can be a camera center that correctly projects all the 
     $y$ points on its conic, we get from 
     $\omega_y^1$ that it correctly projects 
     $y_2, y_3, y_4,y_5$ and from $\omega_y^2$ that it correctly projects $y_1$ as well. Therefore, it must be that 
     $b \in \bigcap_{i=1}^5 \omega_y^i$.
     \end{proof}

\begin{figure}[ht]
\begin{center}

\begin{tikzpicture}[scale=0.7,
	line cap=round,line join=round,>=triangle 45,x=1.0cm,y=1.0cm]

\clip(-2.708101659730916,-1.7014202585919824) rectangle (13.48256455875219,6.226998765922938);

\fill[line width=1.0pt,color=GRAY,fill=GRAY,fill opacity=0.1] 
	(-0.3305456708301174,4.6825528903906175) -- (0.64,4.16) -- (0.72,1.94) -- (-1.34,2.82) -- (-0.38,3.68) -- cycle;
	\draw [color=GRAY]  (-0.3305456708301174,4.6825528903906175) -- (0.64,4.16) -- (0.72,1.94) -- (-1.34,2.82) -- (-0.38,3.68) -- cycle;

\fill[line width=1.0pt,color=GRAY,fill=GRAY,fill opacity=0.1] 
	(3.82,2.02) -- (5.118376078924618,1.9227529675481412) -- (6.04,2.58) -- (6.920215665350591,4.520705539136584) -- (3.86,3.64) -- cycle;
	\draw [color=GRAY] (3.82,2.02) -- (5.118376078924618,1.9227529675481412) -- (6.04,2.58) -- (6.920215665350591,4.520705539136584) -- (3.86,3.64) -- cycle;


\draw [line width=1.1pt,domain=-2.708101659730916:13.48256455875219] plot(\x,{(--6.1076-1.62*\x)/-0.04});
\draw [line width=1.1pt,color=CANVAS,domain= 1 : 2] 
	plot(\x,{(--13.390682369111895-9.597559871781217*\x)/0.05563802824220998});

\draw [line width=1.0pt,dotted,color=purple] (-2.177011390146186,4.04952866062554)-- (-1.34,2.82);
\draw [line width=1.0pt,dotted,color=purple] (-1.34,2.82)-- (1.4022213752941473,-1.2081886416327916);
\draw [line width=1.0pt,dash pattern=on 3pt off 3pt,color=purple] (3.9034581251202867,5.400054067371595)-- (2.924305870239852,2.813143512115842);
\draw [line width=1.0pt,dash pattern=on 3pt off 3pt,color=purple] (2.924305870239852,2.813143512115842)-- (1.4022213752941473,-1.2081886416327916);
\draw [line width=1.0pt,dotted,color=purple] (3.9034581251202867,5.400054067371595)-- (6.920215665350591,4.520705539136584);
\draw [line width=1.0pt,dotted,color=purple] (6.920215665350591,4.520705539136584)-- (9.795301671020107,3.682652558769602);
\draw [line width=1.0pt,dotted,color=purple] (9.795301671020107,3.682652558769602)-- (12.901682094179916,2.7771800384174714);
\draw [samples=50,domain=-0.99:0.99,rotate around={-107.37587681600097:(5.232723278254678,2.304345571199109)},xshift=5.232723278254678cm,yshift=2.304345571199109cm,line width=1.0pt,color=purple] plot ({0.39832169154928954*(1+(\x)^2)/(1-(\x)^2)},{0.8868623380858361*2*(\x)/(1-(\x)^2)});
\draw [samples=50,domain=-0.99:0.99,rotate around={-107.37587681600097:(5.232723278254678,2.304345571199109)},xshift=5.232723278254678cm,yshift=2.304345571199109cm,line width=1.0pt,color=purple] plot ({0.39832169154928954*(-1-(\x)^2)/(1-(\x)^2)},{0.8868623380858361*(-2)*(\x)/(1-(\x)^2)});

\draw [line width=1.0pt,dotted,color=ORANGE] (-2.177011390146186,4.04952866062554)-- (0.72,1.94);
\draw [line width=1.0pt,dotted,color=ORANGE] (0.72,1.94)-- (1.3867857082862747,1.4544639172252556);
\draw [line width=1.0pt,dotted,color=GREEN] (-2.177011390146186,4.04952866062554)-- (-0.38,3.68);
\draw [line width=1.0pt,dotted,color=GREEN] (-0.38,3.68)-- (1.3759773287705541,3.318909383686984);
\draw [line width=1.0pt,dotted,color=RED] (-2.177011390146186,4.04952866062554)-- (0.64,4.16);
\draw [line width=1.0pt,dotted,color=RED] (0.64,4.16)-- (1.3709352718220307,4.188664207307314);
\draw [line width=1.0pt,dotted,color=BLUE] (-2.177011390146186,4.04952866062554)-- (-0.3305456708301174,4.6825528903906175);
\draw [line width=1.0pt,dotted,color=BLUE] (-0.3305456708301174,4.6825528903906175)-- (1.3647029754047957,5.263735339280401);
\draw [line width=1.0pt,dash pattern=on 3pt off 3pt,color=BLUE] (1.3647029754047957,5.263735339280401)-- (2.924305870239852,2.813143512115842);
\draw [line width=1.0pt,dash pattern=on 3pt off 3pt,color=BLUE] (2.924305870239852,2.813143512115842)-- (3.805399695736176,1.4286876773151478);
\draw [line width=1.0pt,dash pattern=on 3pt off 3pt,color=RED] (1.3709352718220307,4.188664207307314)-- (2.924305870239852,2.813143512115842);
\draw [line width=1.0pt,dash pattern=on 3pt off 3pt,color=RED] (2.924305870239852,2.813143512115842)-- (3.82,2.02);
\draw [line width=1.0pt,dash pattern=on 3pt off 3pt,color=ORANGE] (1.3867857082862747,1.4544639172252556)-- (2.924305870239852,2.813143512115842);
\draw [line width=1.0pt,dash pattern=on 3pt off 3pt,color=ORANGE] (2.924305870239852,2.813143512115842)-- (3.86,3.64);
\draw [line width=1.0pt,dash pattern=on 3pt off 3pt,color=GREEN] (1.3759773287705541,3.318909383686984)-- (2.924305870239852,2.813143512115842);
\draw [line width=1.0pt,dash pattern=on 3pt off 3pt,color=GREEN] (2.924305870239852,2.813143512115842)-- (3.8322606795241616,2.516557520728565);
\draw [line width=1.0pt,dotted,color=GREEN] (3.8322606795241616,2.516557520728565)-- (6.04,2.58);
\draw [line width=1.0pt,dotted,color=GREEN] (6.04,2.58)-- (12.901682094179916,2.7771800384174714);
\draw [line width=1.0pt,dotted,color=BLUE] (9.795301671020107,3.682652558769602)-- (5.118376078924618,1.9227529675481412);
\draw [line width=1.0pt,dotted,color=BLUE] (5.118376078924618,1.9227529675481412)-- (3.805399695736176,1.4286876773151478);
\draw [rotate around={10.747206189696287:(6.380230955630514,3.2478259068709936)},line width=1.0pt,color=GREEN] (6.380230955630514,3.2478259068709936) ellipse (3.4919208229843264cm and 1.178710504428178cm);

\draw [samples=50,domain=-0.99:0.99,rotate around={103.10387884276047:(4.958379277728053,3.1112224537656856)},xshift=4.958379277728053cm,yshift=3.1112224537656856cm,line width=1.0pt,color=BLUE] plot ({0.7226785594132924*(1+(\x)^2)/(1-(\x)^2)},{2.7686142433350254*2*(\x)/(1-(\x)^2)});
\draw [samples=50,domain=-0.99:0.99,rotate around={103.10387884276047:(4.958379277728053,3.1112224537656856)},xshift=4.958379277728053cm,yshift=3.1112224537656856cm,line width=1.0pt,color=BLUE] plot ({0.7226785594132924*(-1-(\x)^2)/(1-(\x)^2)},{2.7686142433350254*(-2)*(\x)/(1-(\x)^2)});

\begin{scriptsize}
		\draw [fill=blue] (-0.3305456708301174,4.6825528903906175) circle (1.5 pt);
\draw [fill=RED] (0.64,4.16) circle (1.5pt);
\draw [fill=ORANGE] (0.72,1.94) circle (1.5pt);
\draw [fill=purple] (-1.34,2.82) circle (1.5pt);
\draw [fill=GREEN] (-0.38,3.68) circle (1.5pt);
\draw [color=RED] (3.82,2.02)-- ++(-1.5pt,-1.5pt) -- ++(3.0pt,3.0pt) ++(-3.0pt,0) -- ++(3.0pt,-3.0pt);

\draw [color=BLUE] (5.118376078924618,1.9227529675481412)-- ++(-1.5pt,-1.5pt) -- ++(3.0pt,3.0pt) ++(-3.0pt,0) -- ++(3.0pt,-3.0pt);

\draw [color=GREEN] (6.04,2.58)-- ++(-1.5pt,-1.5pt) -- ++(3.0pt,3.0pt) ++(-3.0pt,0) -- ++(3.0pt,-3.0pt);
\draw [color=purple] (6.920215665350591,4.520705539136584)-- ++(-1.5pt,-1.5pt) -- ++(3.0pt,3.0pt) ++(-3.0pt,0) -- ++(3.0pt,-3.0pt);
\draw [color=ORANGE] (3.86,3.64)-- ++(-1.5pt,-1.5pt) -- ++(3.0pt,3.0pt) ++(-3.0pt,0) -- ++(3.0pt,-3.0pt);
\draw [fill=CANVAS] (1.3534744019504117,7.200664260161612) ++(-1.5pt,0 pt) -- ++(1.5pt,1.5pt)--++(1.5pt,-1.5pt)--++(-1.5pt,-1.5pt)--++(-1.5pt,1.5pt);
\draw [fill=CANVAS] (1.4091124301926217,-2.3968956116196045) ++(-1.5pt,0 pt) -- ++(1.5pt,1.5pt)--++(1.5pt,-1.5pt)--++(-1.5pt,-1.5pt)--++(-1.5pt,1.5pt);
\draw [color=GRAY] (-2.177011390146186,4.04952866062554) circle (1.5pt);
\draw [fill=BLUE] (1.3647029754047957,5.263735339280401) circle (1.5pt);
\draw [fill=RED] (1.3709352718220307,4.188664207307314) circle (1.5pt);
\draw [fill=GREEN] (1.3759773287705541,3.318909383686984) circle (1.5pt);
\draw [fill=ORANGE] (1.3867857082862747,1.4544639172252556) circle (1.5pt);
\draw [fill=purple] (1.4022213752941473,-1.2081886416327916) circle (1.5pt);
\draw [fill=uuuuuu] (2.924305870239852,2.813143512115842) circle (1.5pt);
\draw [fill=BLUE,shift={(3.805399695736176,1.4286876773151478)},rotate=270] (0,0) ++(0 pt,2.25pt) -- ++(1.9485571585149868pt,-3.375pt)--++(-3.8971143170299736pt,0 pt) -- ++(1.9485571585149868pt,3.375pt);
\draw [fill=GREEN,shift={(3.8322606795241616,2.516557520728565)},rotate=180] (0,0) ++(0 pt,2.25pt) -- ++(1.9485571585149868pt,-3.375pt)--++(-3.8971143170299736pt,0 pt) -- ++(1.9485571585149868pt,3.375pt);
\draw [fill=purple,shift={(3.9034581251202867,5.400054067371595)},rotate=180] (0,0) ++(0 pt,2.25pt) -- ++(1.9485571585149868pt,-3.375pt)--++(-3.8971143170299736pt,0 pt) -- ++(1.9485571585149868pt,3.375pt);
\draw [color=GREEN] (9.795301671020107,3.682652558769602) circle (1.5pt);
\draw [color=BLUE] (12.901682094179916,2.7771800384174714) circle (1.5pt);
\draw [color=purple] (6.933219615781861,2.6056679157876603) circle (1.5pt);
\fill (8.741241413491174,2.773018958368121) circle (2.5pt);
	\draw(8.7,3.0) node {$b$};
	    	\draw (-2.3, 4.3) node {$a$};
		
	\draw [color=RED] (0.56, 4.5) node {$ x_4$};
    \draw [color=ORANGE] (0.6, 1.7) node {$ x_5$};
	\draw [color=BLUE] (-0.33, 4.9) node {$x_3$};
	\draw [color=GREEN] (-0.8, 3.6) node {$ x_2$};
	\draw [color=purple] (-1.4, 2.6) node {$ x_1$};
		
	\draw [color=BLUE] (9.1, 5.8) node {$\omega^3_y$};
	\draw [color=GREEN] (9.5, 4.5) node {$\omega^2_y$};
	\draw [color=purple] (12.4, 3.4) node {$\omega^1_y$};

\end{scriptsize}
\end{tikzpicture}
\caption{The conics $\omega_y^1$, $\omega_y^2$  and $\omega_y^3$ intersect in the common point $b$. Here $3$ conics are shown, but any two are enough to find $b$.}
\label{fig:five-tuple-construction}
\end{center}
\end{figure}
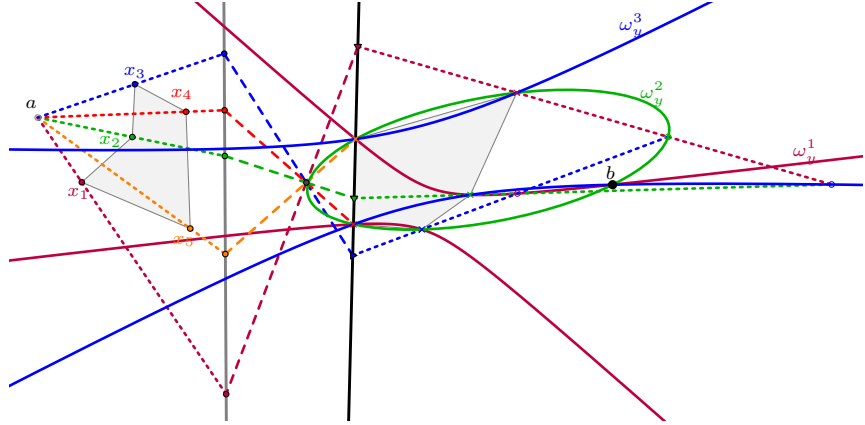

We call the method of determining $b$ from $a$ used in~\Cref{thm:n=5 loci invariants} the {\em intersecting conic construction}, and illustrate it in Figure~\ref{fig:five-tuple-construction}. 
This method was previously observed by Tomas Werner~\cite{werner}.
The construction also works for $n=6$ or $7$ generic points, but subject to the caveat that the point $a$ must then be chosen much more carefully.

In the proof of~\Cref{thm:n=5 loci invariants}, we saw that $b$ is the $4$th point of intersection of two conics that share three points, namely $y_1, y_2, y_3.$
It is known classically that there is an explicit formula for the fourth intersection point as a rational function of $y_1, y_2, y_3$ and the two conics.
To obtain such a formula, it will be convenient to recall the fundamental notion of a Cremona transformation.
\begin{definition}
A birational automorphism $f:\PP^2 \dashedrightarrow \PP^2$ is known as a \textbf{Cremona transformation}. 
It has degree $n$ if it can be defined with forms of degree $n$ having no common factor. 
\end{definition}

The most well-known Cremona transformation is the degree-$2$ \emph{quadratic transformation}
\begin{align}
f_q : \PP^2 &\dashrightarrow \PP^2, \,\,\,\,\,\,
[x_1:x_2:x_3] \mapsto [x_2x_3 : x_1x_3 : x_1x_2].
\end{align}
We now recall some standard facts concerning the quadratic Cremona $f_q.$
Note first that $f_q$ is defined on the set $\PP^2 \setminus \{ e_1, e_2, e_3 \}$.
The points $e_1, e_2, e_3$ where $f_q$ is undefined are known as the \emph{base points} of $f_q.$
Each base point pulls back under $f_q$ to a line:
\begin{equation}\label{eq:exceptional-lines-quadratic-cremona}
f_q^{-1} (e_1) = \langle e_2, e_3 \rangle ,
\quad 
f_q^{-1} (e_2) = \langle e_1, e_3 \rangle ,
\quad 
f_q^{-1} (e_3) = \langle e_1, e_2 \rangle .
\end{equation}
Observe also that for any point $[x_1:x_2:x_3] \notin \{ e_1,e_2,e_3\} \cup f_q^{-1} (\{ e_1, e_2, e_3 \})$, we have
\begin{align*}
f_q ( f_q ( [x_1:x_2:x_3])) &= f_q ([x_2x_3 : x_1x_3 : x_1x_2])\\
&= [ (x_1 x_3) (x_1 x_2) : (x_2 x_3) (x_1 x_2) : (x_2 x_3) (x_1 x_3)]\\
&= [x_1^2 x_2 x_3 : x_1 x_2^2 x_3 : x_1 x_2 x_3^3] = [x_1 : x_2 : x_3].
\end{align*}
From these properties, we see that $f_q$ transforms the intersection of two generic lines $\ell , \ell'$,
\[
\ell \cap \ell ' \, \, :\, \,  c_1 x_1 + c_2 x_2 + c_3 x_3 = d_1 x_1 + d_2 x_2 + d_3 x_3 = 0,
\]
into the fourth intersection point of two generic conics $\omega,$ $\omega'$ passing through $e_1, e_2, e_3$,
\[
c_1 x_2 x_3 + c_2 x_1 x_3 + c_3 x_1 x_2  = 
d_1 x_2 x_3 + d_2 x_1 x_3 + d_3 x_1 x_2 = 0,
\]
and vice-versa. We conclude (cf.~\cite[Footnote 7]{werner}):

\begin{lemma} \label{lem:sturm} If $\omega$ and $\omega '$ are generic conics through $e_1, e_2, e_3,$ they intersect in the fourth point
\begin{equation}\label{eq:sturm}
f_q \left( f_q (\omega) \cap f_q (\omega ') \right) \in \PP^2.
\end{equation}
 \end{lemma}
 In general, if $\omega $ and $\omega '$ are generic conics passing through three given points $x_1, x_2, x_3,$ we may find the fourth point of intersection as follows: choose an additional generic point $x_4,$ change coordinates using the homography $H_x$ in~\eqref{eq:Hy}, apply~\Cref{lem:sturm}, then undo the homography $H_x.$

We conclude this section with an analysis of the camera centers variety $\EPV{5}.$
\begin{theorem} \label{thm:main thm n=5}
    The camera centers variety $\EPV{5}$ is a surface in $\PP^2 \times \PP^2$. 
    Furthermore,
    \begin{enumerate}
        \item Both of the coordinate projections $\pi_x : \EPV{5} \to \PP_x^2$, $\pi_y : \EPV{5} \to \PP_y^2$ are surjective and have rational inverses, $\pi_x^{-1} : \PP_x^2 \dashrightarrow \EPV{5}$, $\pi_y^{-1} : \PP_y^2 \dashrightarrow \EPV{5}$. The composite rational map $\pi_y \circ \pi_x^{-1}:\PP^2_x \dashrightarrow \PP^2_y$ taking $a \rightarrow b$ is a degree-$5$ Cremona transformation of $\PP^2$.
        \item The multidegrees of $\EPV{5}$ are 
        $d_{(2,0)} \left( \EPV{5} \right) = 
d_{(0,2)} \left( \EPV{5} \right) = 1$ and $d_{(1,1)} \left( \EPV{5} \right) = 5.$
     \end{enumerate}
\end{theorem}

\begin{proof}\Cref{lemma:dim-En} gives that  $\dim (\EPV{5})=2,$ ie.~$\EPV{5}$ is a surface.
Let us consider the projection
\begin{align}
\pi_x : \EPV{5} &\to \PP^2_x, \,\,\,\,\,\,
(a,b) \mapsto a . \label{eq:map-pi-a}
\end{align}
For part (1), consider the fiber $\pi_x^{-1} (a)$ over a generic point $a\in \PP^2_x.$
Using the intersecting conic construction, we know that $(a,b) \in \pi_x^{-1} (a)$ is uniquely determined, and $b$ must lie on the intersection of two conics: $\omega $ passing through $a$ and $x_1,x_2,x_3,x_4\in \X$ and $\omega '$ passing through $a$ and $x_1, x_2, x_3, x_5 \in \X .$
Using~\Cref{lem:sturm}, this gives immediately the existence of a rational inverse for $\pi_x$.
Reversing the roles of $a$ and $b$ shows that $\pi_y$ also has a rational inverse; it follows that $\pi_y \circ \pi_x^{-1}$ (and its rational inverse $\pi_x \circ \pi_y^{-1}$) are both Cremona transformations.

To determine the degree of the Cremona transformation $\pi_y \circ \pi_x^{-1}$, assume without loss of generality that
$\X $ and $\Y$ are in standard  position.
Using~\eqref{eq:ab-to-F}, we may take $\omega $ to be the conic whose equation in $b$ is given by
\begin{align}
e_4^T \diag (a) [a\odot  b ]_{\times } \diag (b) e_4 &= e_4^T \left(\!\begin{array}{ccc}
      0&-a_{1}a_{3}b_{2}b_{3}&a_{1}a_{2}b_{2}b_{3}\\
      a_{2}a_{3}b_{1}b_{3}&0&-a_{1}a_{2}b_{1}b_{3}\\
      -a_{2}a_{3}b_{1}b_{2}&a_{1}a_{3}b_{1}b_{2}&0
      \end{array}\!\right) e_4 \nonumber  \\
      &= a_3 (a_1 - a_2) b_1 b_2 + a_2 (a_3 - a_1) b_1 b_3 + a_1 (a_2 - a_3) b_2 b_3 = 0. \label{eq:C}
\end{align}
A similar calculation gives us the equation of $\omega '$:
\begin{align}
&y_5^T \diag (a) [a\odot  b ]_{\times } \diag (b) x_5 = \nonumber \\ 
&a_3y_{5,3} \left(a_{1}x_{5,2}-a_{2}x_{5,1}\right)b_{1}b_{2
      }+a_2y_{5,2}\left(a_{3}x_{5,1}-a_{1}x_{5,3}\right)b_{1}b_{3
      }+y_{5,1}a_1\left(a_{2}x_{5,3}-a_{3}x_{5,2}\right)b_{2}b_{
      3}
= 0. \label{eq:C'}
\end{align}
We now apply~\Cref{lem:sturm}.
The two lines $\ell = f_q (\omega )$ and $\ell ' = f_q (\omega ')$ intersect in the point $f_q (b)$, whose homogeneous coordinates may be obtained by taking the cross product of coefficient vectors obtained from the conic equations~\eqref{eq:C} and~\eqref{eq:C'}; thus
\begin{align}
f_q (b) =
\begin{bmatrix}
a_3 (a_1 - a_2) \\
a_2 (a_3 - a_1)\\
a_1 (a_2 - a_3)
\end{bmatrix}
\times 
\begin{bmatrix}
y_{5,3} a_3 (a_{1}x_{5,2}-a_{2}x_{5,1}) \\
y_{5,2} a_2 (a_{3}x_{5,1}-a_{1}x_{5,3})   \\ 
y_{5,1} a_1 (a_{2}x_{5,3}-a_{3}x_{5,2})\\
\end{bmatrix}
= 
\begin{bmatrix}
a_1  a_2\, p_1 (a; x_5, y_5) \\
a_1 a_3 \, p_2 (a; x_5, y_5) \\
a_2  a_3 \, p_3 (a; x_5, y_5)
\end{bmatrix}, \label{eq:fq-b}
\end{align}
where $p_1, p_2, p_3$ are homogeneous polynomials of degree $2$ in $a$.
Explicitly, we have
\begin{tiny}
\begin{align}
p_1 (a) &= \left(x_{5,2}y_{5,1}-x_{5,1}y_{5,2}\right)a_{3}^{2} +\left(x_{5,2}y_{5,3}-x_{5,1}y_{5,3}\right)a_{1}a_{2}+\left(x_{5,1}y_{5,2} - x_{5,2}y_{5,3}\right)a_{1}a_{3}+\left(x_{5,1}y_{5,2} - x_{5,2}y_{5,1}\right)a_{2}a_{3},\nonumber \\
p_2 (a) &= \left(x_{5,1}y_{5,3}- x_{5,3}y_{5,1}\right)a_{2}^{2} + \left(x_{5,3}y_{5,2}- x_{5,1}y_{5,3}\right)a_{1}a_{2}+\left(x_{5,1}y_{5,2}-x_{5,3}y_{5,2}\right)a_{1}a_{3}+\left(x_{5,3}y_{5,1}-x_{5,1}y_{5,2}\right)a_{2}a_{3},\nonumber \\
p_3 (a) &= \left(x_{5,3}y_{5,2}-x_{5,2}y_{5,3}\right)a_{1}^{2}+\left(x_{5,2}y_{5,3}-x_{5,3}y_{5,1}\right)a_{1}a_{2}+\left(x_{5,2}y_{5,1}-x_{5,3}y_{5,2}\right)a_{1}a_{3}+\left(x_{5,3}y_{5,1}-x_{5,2}y_{5,1}\right)a_{2}a_{3}. \label{eq:pi123}
\end{align}
\end{tiny}
Applying $f_q$ to~\eqref{eq:fq-b}, we obtain
\begin{align}
b = f_q ( f_q (b))  
= 
\begin{bmatrix}
a_1^2 a_2 a_3 \,  p_2 (a ; x_5, y_5) p_3 (a ; x_5, y_5) \\
a_1 a_2^2 a_3 \, p_1 (a ; x_5, y_5) p_3 (a ; x_5, y_5)\\
a_1 a_2 a_3^2 \, p_1 (a ; x_5, y_5) p_2 (a ; x_5, y_5)
\end{bmatrix} 
\sim 
\begin{bmatrix}
a_1 \, p_2 (a ; x_5, y_5) p_3 (a ; x_5, y_5) \\
a_2 \, p_1 (a ; x_5, y_5) p_3 (a ; x_5, y_5)\\
a_3 \, p_1 (a ; x_5, y_5) p_2 (a ; x_5, y_5)
\end{bmatrix}. \label{eq:b-from-a}
\end{align}
In summary,~\eqref{eq:b-from-a} provides an explicit formula for the Cremona transformation
$\pi_y \circ \pi_x^{-1} : \PP^2_x \dashrightarrow \PP^2_y$ that sends $a\mapsto b.$
The polynomials $p_1,p_2,p_3$ are homogeneous of degree $2$ in $a.$
Upon verifying that $p_1, \ldots , p_3$ are irreducible and not multiples of each other, we deduce from~\eqref{eq:b-from-a} that the Cremona transformation  $\pi_y \circ \pi_x^{-1}$ has degree $5.$ 

For part (2), note first that slicing $\PP_x $ by $2$ generic hyperplanes determines a unique point $a,$ from which $b$ is also uniquely determined by~\eqref{eq:b-from-a}.
Thus $d_{(2,0) } (\EPV{5}) = 1,$ and similar remarks give $d_{(0,2)} (\EPV{5}) = 1.$
To obtain $d_{(2,0) } (\EPV{5}) ,$ suppose we slice $\EPV{5}$ with generic line $\ell_x \subset \PP_x $ and a generic line $\ell_y \subset \PP_y$.
Consider a parametric description of $\ell_x$: for fixed, generic $a_1, a_2 \in \PP_x^2$, 
\begin{equation}
\ell_x = 
\left\{ 
s a_1 + t a_2 \mid 
[s:t] \in \PP^1 
\right\} .\label{eq:lx-parametric}
\end{equation}
Consider also an implicit description of the line $\ell_y,$
\begin{equation}
\ell_y = \left\{ b \in \PP^2 \mid 
\begin{bmatrix}
d_1 & d_2 & d_3 
\end{bmatrix}
b
= 0
\right\}.\label{eq:ly-implicit}
\end{equation}
Then, any point $(a,b) \in \EPV{5} \cap \ell_x \cap \ell_y $ must satisfy
\begin{equation}
 \begin{bmatrix}
d_1 & d_2 & d_3 
\end{bmatrix}
\begin{bmatrix}
\pi_y \circ \pi_x^{-1} (s a_1 + t a_2 ) \\
\pi_y \circ \pi_x^{-1} (s a_1 + t a_2 ) \\
\pi_y \circ \pi_x^{-1} (s a_1 + t a_2 ) 
\end{bmatrix} = 0.\label{eq:poly-deg-5}
\end{equation}
We conclude the proof of part (2) by noting that, in the affine chart $s=1,$~\eqref{eq:poly-deg-5}, defines a polynomial of degree $5$ in $t$ with $5$ distinct roots, and thus $d_{(1,1)} (\EPV{5}) = \# (\EPV{5} \cap \ell_x \cap \ell_y) = 5.$
\end{proof}

Werner also notes (\cite[Section 5]{werner}) that part (1) of~\Cref{thm:main thm n=5} is a classical result, citing a 1908 treatise of Rudolf Sturm~\cite{1908sturm}
Our proof provides a detailed, self-contained  justification, and the closely-connected multidegree computation $d_{(1,1)} \left(\EPV{5}\right).$ 
We may also check the conclusions of~\Cref{thm:main thm n=5} and obtain additional information about $\EPV{5}$ by computing the vanishing ideal $\mathcal{I} (\EPV{5})$ with the help of the computer algebra system Macaulay2~\cite{M2}.
The ideal $\mathcal{I} (\EPV{5})$ is generated by $11$ polynomials---3 polynomials of bidegree $(1,3),$ 3 of bidegree $(3,1),$ and 5 of bidegree $(2,2).$

We close by discussing some additional properties of Cremona transformations.
In particular, we will identify properties of the degree-$5$ transformation $\pi_y \circ \pi_x^{-1}$ that will be useful in~\Cref{sec:6-points}.

Much like the standard quadratic transformation, a general Cremona transformation
\begin{align}
  f:\PP^2 &\dashrightarrow \PP^2, \,\,\,\,\,\,
  [x_1:x_2:x_3] \mapsto [f_1 (x) : f_2 (x) : f_3(x) ]
  \label{eq:generic-cremona}
  \end{align}
of degree $n$ has a finite set of \emph{base points} where $f$ is undefined.
To determine the base points of $f$, we associate to $f$ a $2$-parameter family of curves: for $\lambda = [\lambda_1 : \lambda_2 : \lambda_3] \in \PP^2,$ consider the curve
\begin{align}
C_\lambda  = \{ x\in \PP^2 \mid \lambda_1 f_1(x) + \lambda_2 f_2(x) + \lambda_3 f_3(x) = 0 \}.\label{eq:cremona}
\end{align}
The family of curves $\left( C_\lambda \right)_{\lambda \in \PP^2}$ is known classically as the \emph{homoloidal net} associated to $f.$
Observe that each $C_\lambda $ passes through the base points of $f.$
For generic $\lambda $, we define the \emph{multiplicity} of a base point $p$ with respect to $f$ to be the multiplicity of $p$ as a point of the curve $C_\lambda $.
In other words, the multiplicity of a base point $p$ is the largest $k$ such that all $k$-fold partial derivatives of the defining equation of $C_\lambda $ vanish at $p.$
This definition is independent of the choice of generic $\lambda .$

The following result is known classically.
We refer to the nice text~\cite{beltrametti} for a proof in modern language which utilizes B\'{e}zout's theorem and the genus-degree formula.

\begin{proposition}\label{prop:cremona-base points}
(See eg.~\cite[eq.~9.12]{beltrametti})
  If $f: \PP^2 \dashrightarrow \PP^2$ is a degree-$n$ Cremona transformation whose base points $p_1, \ldots , p_k$ have multiplicities $s_1, \ldots , s_k,$ then
  \begin{equation}\label{eq:cremona-base point-count}
\displaystyle\sum_{i=1}^k s_i = 3 (n-1).
    \end{equation}
\end{proposition}

We now determine the base points of the Cremona transformation $\pi_y \circ \pi_x^{-1}.$
The statement of the next result is also due to Werner~\cite{werner}; however, his justification is incomplete, as he asserts that the number of base points in~\Cref{prop:cremona-base points} is generally $n+1$ (which happens to agree with the correct answer when $n\in \{ 2, 5 \} .$) 

\begin{theorem}\label{thm:cremona-6-base points}
  The Cremona transformation $\pi_y \circ \pi_x^{-1}: \PP^2 \dashrightarrow \PP^2$ associated to generic $\X$ and $\Y$ has exactly six base points, each of multiplicity $2,$
  consisting of $x_1, \ldots , x_5,$ and a sixth \textbf{exceptional point} $\hat{x}$.
  When $\X $ and $\Y$ are in standard position, the exceptional point is given by
  \begin{equation}\label{eq:exceptional-point-x}
\hat{x} = \left[\frac{y_{53}-y_{52}}{y_{53}x_{52}-x_{53}y_{52}}
:
\frac{y_{53}-y_{51}}{y_{53}x_{51}-x_{53}y_{51}}
:
\frac{y_{51}-y_{52}}{y_{51}x_{52}-x_{51}y_{52}}\right].
    \end{equation}
\end{theorem}
The formula for $\hat{x}$ in~\eqref{eq:exceptional-point-x} appeared previously in~\cite{connelly2024geometry}.
\begin{proof}
By~\Cref{prop:cremona-base points}, the sum of multiplicities of all base points must equal $3 (5-1) =12.$ 
Thus, once we have shown that $x_1, \ldots , x_5, \hat{x}$ are base points of multiplicity $2$, we may conclude that all base points have been found.
Fixing generic $\X$ and $\Y$ in standard position, let us consider the equation of the homaloidal net associated to $\pi_y \circ \pi_x^{-1}$,
\begin{equation}\label{eq:homaloidal-net}
h(a ; \lambda ) = 
\lambda_1 a_1 p_2 (a) p_3 (a) + 
\lambda_2 a_2 p_1 (a) p_3 (a) +
\lambda_3 a_3 p_1 (a) p_2 (a) = 0,
\end{equation}
where $p_i$ are as in~\eqref{eq:pi123}
For generic values of $\lambda ,$ we may verify that $h$ and its first-order partials $\partial_{a_i} h$ vanish at each base point, but the second-order partials do not.
Verifying that the coordinate functions~\eqref{eq:b-from-a} vanish at $e_1, \ldots , e_4, x_5$, and $\hat{x}$ then gives the result.
\end{proof}

Finally, we point out that the degree 5 Cremona transformation $\pi_y \circ \pi_x^{-1}$ has the property that it maps $6$ particular conics $\omega_{y,1}, \ldots , \omega_{y,5}, \hat{\omega}_y \subset \PP^2$ down to six points $y_1, \ldots , y_5,\hat{y} \in \PP^2 $, much like how the quadratic Cremona transformation $f_q$ maps 3 lines onto its base points. 
Here $\hat{y}$ is another \emph{exceptional point} naturally paired with $\hat{x}.$
If $\X$ and $\Y$ are in standard position, then
\begin{equation}\label{eq:exceptional-point-y}
\hat{y} = \left[
\frac{x_{53}-x_{52}}{y_{53}x_{52}-x_{53}y_{52}}
:
\frac{x_{53}-x_{51}}{y_{53}x_{51}-x_{53}y_{51}}
:
\frac{x_{51}-x_{52}}{y_{51}x_{52}-x_{51}y_{52}}
\right],
\end{equation}
and the conics $\omega_{y,i} = \left( \pi_y \circ \pi_x^{-1} \right)^{-1} (e_i)$ for $i=1,\ldots , 3$ are defined by the polynomials $p_i$ in~\eqref{eq:pi123}.
For generic $\X$ and $\Y,$ not necessarily in standard position, each conic $\omega_{y,i}$ for $i=1,\ldots 5$ is uniquely determined by the requirement that it passes through the five points of $\left(\mathcal{X} \cup \{ \hat{x} \}\right) \setminus \{ x_i \},$ and $\hat{\omega}_y = \left( \pi_y \circ \pi_x^{-1} \right)^{-1} (\hat{y})$ is the unique conic through $\X .$
General algebraic formulas for the exceptional points $\hat{x}$ and $\hat{y}$ follow by composing the homographies~\eqref{eq:Hx} and~\eqref{eq:Hy} with~\eqref{eq:exceptional-point-x} and~\eqref{eq:exceptional-point-y}.
These exceptional points will play a key role in our geometric construction for $n=6$ points.

\section{Camera loci when $n=6$}\label{sec:6-points}

We now come to the case of $n=6$ where we will see some 
intricate invariant theory helping to answer Question~\ref{question:q2}. The first theorem of this section is the following. 

\begin{theorem} \label{thm:n=6 invariant theory}
Let $\X = \{x_1,\ldots, x_6\} \subset \PP^2_x$ and $\Y = \{y_1, \ldots, y_6\} \subset \PP^2_y$ be two sets of generic labeled points that can be imaged by flatland cameras to $\Pcal \simeq \Qcal$ in $\PP^1$. Then the camera centers $a$ and $b$ lie on cubic curves $C_x \subset \PP^2_x$  and $C_y \subset \PP^2_y$ passing through the points in $\X$ and $\Y$. Once $a$ is chosen, $b$ is determined. The equations of $C_x$ and $C_y$ can be computed explicitly from $\X$ and $\Y$.
\end{theorem}

We will now prepare  to prove Theorem~\ref{thm:n=6 invariant theory}. 
In this case it will be convenient to work with a particular set of non-minimal generators of the ring $R$ (c.f.~\Cref{thm:even odd generators}). We first enlarge the set non-crossing generators from~\Cref{ex:min gens} to include the degree $\ones$ bracket monomial 
\begin{align}
    m_0 = [14][25][36].
\end{align}
The graphs corresponding to the $6$ bracket monomials $m_0, \ldots, m_5$ are shown in Figure~\ref{fig:m0-to-m6}

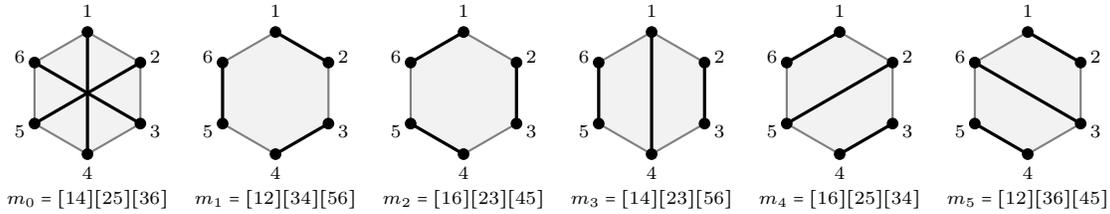
\begin{figure}[htbp]
    \centering
    \definecolor{yqyqyq}{rgb}{0.5019607843137255,0.5019607843137255,0.5019607843137255}
\begin{tikzpicture}[line cap=round,line join=round,>=triangle 45,x=1.0cm,y=1.0cm]
\draw[line width=0.8pt,color=yqyqyq,fill=yqyqyq,fill opacity=0.1] (-0.7035793656460445,0.40621173615201) -- (0.,0.) -- (0.7035793656460446,0.4062117361520099) -- (0.7035793656460447,1.21863520845603) -- (0.,1.62484694460804) -- (-0.7035793656460442,1.2186352084560306) -- cycle;
\draw[line width=0.8pt,color=yqyqyq,fill=yqyqyq,fill opacity=0.1] (1.7964206343539555,0.40621173615201) -- (2.5,0.) -- (3.2035793656460445,0.40621173615201) -- (3.2035793656460445,1.21863520845603) -- (2.5,1.62484694460804) -- (1.7964206343539557,1.2186352084560306) -- cycle;
\draw[line width=0.8pt,color=yqyqyq,fill=yqyqyq,fill opacity=0.1] (4.2964206343539555,0.40621173615201) -- (5.,0.) -- (5.7035793656460445,0.40621173615201) -- (5.7035793656460445,1.21863520845603) -- (5.,1.62484694460804) -- (4.2964206343539555,1.2186352084560306) -- cycle;
\draw[line width=0.8pt,color=yqyqyq,fill=yqyqyq,fill opacity=0.1] (6.7964206343539555,0.40621173615201) -- (7.5,0.) -- (8.203579365646045,0.40621173615201) -- (8.203579365646045,1.21863520845603) -- (7.5,1.62484694460804) -- (6.7964206343539555,1.2186352084560306) -- cycle;
\draw[line width=0.8pt,color=yqyqyq,fill=yqyqyq,fill opacity=0.1] (9.296420634353955,0.40621173615201) -- (10.,0.) -- (10.703579365646048,0.4062117361520082) -- (10.70357936564605,1.21863520845603) -- (10.,1.6248469446080418) -- (9.296420634353957,1.2186352084560323) -- cycle;
\draw[line width=0.8pt,color=yqyqyq,fill=yqyqyq,fill opacity=0.1] (11.796420634353957,0.40621173615201) -- (12.5,0.) -- (13.203579365646045,0.4062117361520077) -- (13.203579365646046,1.2186352084560268) -- (12.5,1.6248469446080376) -- (11.796420634353959,1.2186352084560297) -- cycle;

\draw [line width=1.2pt] (0.,1.62484694460804)-- (0.,0.);
\draw [line width=1.2pt] (0.7035793656460447,1.21863520845603)-- (-0.7035793656460445,0.40621173615201);
\draw [line width=1.2pt] (0.7035793656460446,0.4062117361520099)-- (-0.7035793656460442,1.2186352084560306);
\draw [line width=1.2pt] (2.5,1.62484694460804)-- (3.2035793656460445,1.21863520845603);
\draw [line width=1.2pt] (3.2035793656460445,0.40621173615201)-- (2.5,0.);
\draw [line width=1.2pt] (1.7964206343539555,0.40621173615201)-- (1.7964206343539557,1.2186352084560306);
\draw [line width=1.2pt] (5.,1.62484694460804)-- (4.2964206343539555,1.2186352084560306);
\draw [line width=1.2pt] (5.7035793656460445,1.21863520845603)-- (5.7035793656460445,0.40621173615201);
\draw [line width=1.2pt] (5.,0.)-- (4.2964206343539555,0.40621173615201);
\draw [line width=1.2pt] (7.5,1.62484694460804)-- (7.5,0.);
\draw [line width=1.2pt] (8.203579365646045,1.21863520845603)-- (8.203579365646045,0.40621173615201);
\draw [line width=1.2pt] (6.7964206343539555,1.2186352084560306)-- (6.7964206343539555,0.40621173615201);
\draw [line width=1.2pt] (10.,1.6248469446080418)-- (9.296420634353957,1.2186352084560323);
\draw [line width=1.2pt] (10.70357936564605,1.21863520845603)-- (9.296420634353955,0.40621173615201);
\draw [line width=1.2pt] (10.703579365646048,0.4062117361520082)-- (10.,0.);
\draw [line width=1.2pt] (12.5,1.6248469446080376)-- (13.203579365646046,1.2186352084560268);
\draw [line width=1.2pt] (11.796420634353959,1.2186352084560297)-- (13.203579365646045,0.4062117361520077);
\draw [line width=1.2pt] (11.796420634353957,0.40621173615201)-- (12.5,0.);

\begin{scriptsize}
    \foreach \n in {0,...,5} {    
    \draw[color=black] (0+2.5*\n, 1.9) node {1};
    \draw[color=black] (0.9+2.5*\n, 1.3) node {2};
    \draw[color=black] (0.9+2.5*\n, 0.3) node {3};
    \draw[color=black] (0+2.5*\n, -0.25) node {4};
    \draw[color=black] (-0.9+2.5*\n, 0.3) node {5};
    \draw[color=black] (-0.9+2.5*\n, 1.3) node {6};
   \draw [fill=black] (0+2.5*\n,0.) circle (2.0pt);
      \draw [fill=black] (-0.7035793656460445+2.5*\n,0.40621173615201) circle (2.0pt);
    \draw [fill=black] (0.7035793656460446+2.5*\n,0.4062117361520099) circle (2.0pt);
    \draw [fill=black] (0.7035793656460447+2.5*\n,1.21863520845603) circle (2.0pt);
    \draw [fill=black] (0.+2.5*\n,1.62484694460804) circle (2.0pt);
    \draw [fill=black] (-0.7035793656460442+2.5*\n,1.2186352084560306) circle (2.0pt);
    }
    \draw (0, -0.6) node {$m_0 = [14][25][36]$};
    \draw (2.5, -0.6) node {$m_1=[12][34][56]$};
    \draw (5, -0.6) node {$m_2=[16][23][45]$};
    \draw (7.5, -0.6) node {$m_3=[14][23][56]$};
    \draw (10, -0.6) node {$m_4=[16][25][34]$};
    \draw (12.5, -0.6) node {$m_5=[12][36][45]$};

\end{scriptsize}
\end{tikzpicture}
    \caption{The generators $m_1$ through $m_5$, with $m_0$.}
    \label{fig:m0-to-m6}
\end{figure}

The monomial $m_0$ is a perfect matching with crossing edges and hence is linearly dependent on $m_1, \ldots, m_5$. In fact, $m_0 =m_1+m_2+m_3+m_4+m_5$ 
which can be seen by successively uncrossing edges 
via the Pl\"ucker relations 
in Figure~\ref{fig:plucker-relations}.

Applying an invertible linear change of coordinates to $m_0, \ldots, m_6$ yields the {\em Joubert invariants} $A,B,C,D,E,F$ of $6$ labeled points in $\PP^1$, defined as: 
\begin{align}
    \begin{bmatrix}
        A \\ 
        B \\ 
        C \\ 
        D \\ 
        E \\ 
        F 
    \end{bmatrix} = 
    2 \begin{bmatrix} 
    -1 & -1 & -1 & 0 & 0 & 0 \\
    0 & 0 & 0 & -1 & -1 & 1 \\ 
    0 & 0 & 0 & 1 & -1 & -1 \\
    1 & 1 & -1 & 0 & 0 & 0 \\
    1 & -1 & 1& 0 & 0 & 0 \\ 
    0 & 0 & 0 & -1 & 1 & -1 \end{bmatrix}
    \begin{bmatrix}
        m_0 \\ m_1 \\ m_2 \\ m_3 \\ m_4 \\ m_5
    \end{bmatrix}.
\end{align}

Setting $g = (A,B,C,D,E,F)$, by Lemma~\ref{lem:signatures of orbits},  we have that since $\Pcal \simeq \Qcal$,
$g(\Pcal) \sim g(\Qcal)$. 
In order to pull back the Joubert invariants to $\PP^2$ as in Lemma~\ref{lem:pullback}, we need explicit bracket 
expressions. The classical expressions involve another change of coordinates and is the following:
\begin{small}
\begin{align} \label{eq:Joubert invariants}
\begin{split}
	A&= [12][34][56] + [13][46][25] + [14][26][35] + [15][24][36] + [16][23][45] \\
	B &= [45][31][26] + [43][16][52] + [41][56][32] + [42][51][36] + [46][53][26] \\
	C &=[15][64][23] + [16][43][52] + [14][53][62] + [12][54][63] + [13][56][42] \\
	D& =[14][36][52] + [13][62][45] + [16][42][35] + [15][46][32] + [12][43][65] \\
	E &=[16][32][54] + [13][24][65] + [12][64][35] + [15][62][34] + [14][63][25] \\
	F &=[42][61][53] + [46][13][25] + [41][23][65] + [45][21][63] + [43][26][15] 
\end{split}
\end{align}
\end{small}
Note that these expressions are not written using $m_0,\ldots,m_6$. 
For example, $A$
is the sum of the following crossing and non-crossing perfect matchings.

\begin{center}
\definecolor{GRAY}{rgb}{0.5,0.5,0.5}

\begin{tikzpicture}[scale = 1.1,
	line cap=round,line join=round,>=triangle 45,x=1.0cm,y=1.0cm]
\fill[ color=GRAY,fill=GRAY,fill opacity=0.10000000149011612] (-0.5628634925168357,0.324969388921608) -- (0.,0.) -- (0.5628634925168357,0.32496938892160787) -- (0.5628634925168357,0.9749081667648238) -- (0.,1.299877555686432) -- (-0.5628634925168354,0.9749081667648244) -- cycle;
\fill[ color=GRAY,fill=GRAY,fill opacity=0.10000000149011612] (1.4371365074831646,0.324969388921608) -- (2.,0.) -- (2.5628634925168354,0.32496938892160787) -- (2.5628634925168354,0.9749081667648236) -- (2.,1.2998775556864315) -- (1.4371365074831648,0.9749081667648241) -- cycle;
\fill[ color=GRAY,fill=GRAY,fill opacity=0.10000000149011612] (3.4371365074831646,0.324969388921608) -- (4.,0.) -- (4.562863492516835,0.32496938892160787) -- (4.562863492516835,0.9749081667648236) -- (4.,1.2998775556864315) -- (3.4371365074831646,0.9749081667648241) -- cycle;
\fill[ color=GRAY,fill=GRAY,fill opacity=0.10000000149011612] (5.437136507483165,0.324969388921608) -- (6.,0.) -- (6.562863492516835,0.32496938892160787) -- (6.562863492516835,0.9749081667648236) -- (6.,1.2998775556864315) -- (5.437136507483165,0.9749081667648241) -- cycle;
\fill[ color=GRAY,fill=GRAY,fill opacity=0.10000000149011612] (7.437136507483165,0.324969388921608) -- (8.,0.) -- (8.562863492516835,0.324969388921608) -- (8.562863492516835,0.9749081667648237) -- (8.,1.2998775556864317) -- (7.437136507483165,0.9749081667648241) -- cycle;
\draw [line width=1.0pt] (0.,1.299877555686432) -- (0.5628634925168357,0.9749081667648238);
\draw [line width=1.0pt] (0.5628634925168357,0.32496938892160787) -- (0.,0.);
\draw [line width=1.0pt] (-0.5628634925168357,0.324969388921608) -- (-0.5628634925168354,0.9749081667648244);
\draw [line width=1.0pt] (2.,1.2998775556864315) -- (2.5628634925168354,0.32496938892160787);
\draw [line width=1.0pt] (2.,0.) -- (1.4371365074831648,0.9749081667648241);
\draw [line width=1.0pt] (2.5628634925168354,0.9749081667648236) -- (1.4371365074831646,0.324969388921608);
\draw [line width=1.0pt] (4.,1.2998775556864315) -- (4.,0.);
\draw [line width=1.0pt] (4.562863492516835,0.9749081667648236) -- (3.4371365074831646,0.9749081667648241);
\draw [line width=1.0pt] (4.562863492516835,0.32496938892160787) -- (3.4371365074831646,0.324969388921608);
\draw [line width=1.0pt] (6.,1.2998775556864315) -- (5.437136507483165,0.324969388921608);
\draw [line width=1.0pt] (6.562863492516835,0.9749081667648236) -- (6.,0.);
\draw [line width=1.0pt] (6.562863492516835,0.32496938892160787) -- (5.437136507483165,0.9749081667648241);
\draw [line width=1.0pt] (8.,1.2998775556864317) -- (7.437136507483165,0.9749081667648241);
\draw [line width=1.0pt] (8.562863492516835,0.9749081667648237) -- (8.562863492516835,0.324969388921608);
\draw [line width=1.0pt] (8.,0.) -- (7.437136507483165,0.324969388921608);
\begin{scriptsize}
\draw [fill=black] (0.,0.) circle (2.0pt);
\draw [fill=black] (8.,0.) circle (2.0pt);
\draw [fill=black] (4.,0.) circle (2.0pt);
\draw [fill=black] (6.,0.) circle (2.0pt);
\draw [fill=black] (2.,0.) circle (2.0pt);
\draw [fill=black] (-0.5628634925168357,0.324969388921608) circle (2.0pt);
\draw [fill=black] (7.437136507483165,0.324969388921608) circle (2.0pt);
\draw [fill=black] (3.4371365074831646,0.324969388921608) circle (2.0pt);
\draw [fill=black] (1.4371365074831646,0.324969388921608) circle (2.0pt);
\draw [fill=black] (5.437136507483165,0.324969388921608) circle (2.0pt);
\draw [fill=black] (0.5628634925168357,0.32496938892160787) circle (2.0pt);
\draw [fill=black] (0.5628634925168357,0.9749081667648238) circle (2.0pt);
\draw [fill=black] (0.,1.299877555686432) circle (2.0pt);
\draw [fill=black] (-0.5628634925168354,0.9749081667648244) circle (2.0pt);
\draw [fill=black] (2.5628634925168354,0.32496938892160787) circle (2.0pt);
\draw [fill=black] (2.5628634925168354,0.9749081667648236) circle (2.0pt);
\draw [fill=black] (2.,1.2998775556864315) circle (2.0pt);
\draw [fill=black] (1.4371365074831648,0.9749081667648241) circle (2.0pt);
\draw [fill=black] (4.562863492516835,0.32496938892160787) circle (2.0pt);
\draw [fill=black] (4.562863492516835,0.9749081667648236) circle (2.0pt);
\draw [fill=black] (4.,1.2998775556864315) circle (2.0pt);
\draw [fill=black] (3.4371365074831646,0.9749081667648241) circle (2.0pt);
\draw [fill=black] (6.562863492516835,0.32496938892160787) circle (2.0pt);
\draw [fill=black] (6.562863492516835,0.9749081667648236) circle (2.0pt);
\draw [fill=black] (6.,1.2998775556864315) circle (2.0pt);
\draw [fill=black] (5.437136507483165,0.9749081667648241) circle (2.0pt);
\draw [fill=black] (8.562863492516835,0.324969388921608) circle (2.0pt);
\draw [fill=black] (8.562863492516835,0.9749081667648237) circle (2.0pt);
\draw [fill=black] (8.,1.2998775556864317) circle (2.0pt);
\draw [fill=black] (7.437136507483165,0.9749081667648241) circle (2.0pt);

	\draw[color=black] (0, 1.55) node {$1$};
	\draw[color=black] (2, 1.55) node {$1$};
	\draw[color=black] (4, 1.55) node {$1$};
	\draw[color=black] (6, 1.55) node {$1$};
	\draw[color=black] (8, 1.55) node {$1$};

	\draw[color=black] (0.75, 0.9) node {$2$};
	\draw[color=black] (2.8, 0.2) node {$3$};
	\draw[color=black] (4, -0.25) node {$4$};
	\draw[color=black] (5.25, 0.2) node {$5$};
	\draw[color=black] (7.2, 1.0) node {$6$};

\end{scriptsize}
\end{tikzpicture}
\end{center}

Pulling $A, \ldots, F$  back to $\PP^2$ we obtain the cubic polynomials (in $u=(u_0,u_1,u_3)$):
\begin{small}
\begin{align} \label{eq:covariant cubics}
\begin{split}
a(u)&=[25u][13u][46u]+[51u][42u][36u]+[14u][35u][26u]+[43u][21u][56u]+[32u][54u][16u]\\
b(u)&=[53u][12u][46u]+[14u][23u][56u]+[25u][34u][16u]+[31u][45u][26u]+[42u][51u][36u]\\
c(u)&=[53u][41u][26u]+[34u][25u][16u]+[42u][13u][56u]+[21u][54u][36u]+[15u][32u][46u]\\
d(u)&=[45u][31u][26u]+[53u][24u][16u]+[41u][25u][36u]+[32u][15u][46u]+[21u][43u][56u]\\
e(u)&=[31u][24u][56u]+[12u][53u][46u]+[25u][41u][36u]+[54u][32u][16u]+[43u][15u][26u]\\
f(u)&=[42u][35u][16u]+[23u][14u][56u]+[31u][52u][46u]+[15u][43u][26u]+[54u][21u][36u]
\end{split}
\end{align}
\end{small}

By Theorem~\ref{thm:pull back tool}, $g'(\X,a) \sim g'(\Y,b)$ which we write as 
\begin{align} \label{eq:equality of bumped joubert}
(a_x(a),\ldots,f_x(a))\sim(a_y(b),\ldots,f_y(b))
\end{align}
where $a_x(a)$ is the evaluation of $a(u)$ at $\X,a$.

Now suppose we focus on $\X$.
It is known that the closure of the parameterized variety 
\begin{align} \label{eq:Sx} 
\left\{ [a_x(u) : b_x(u) : c_x(u) : d_x(u) : e_x(u) : f_x(u))] \in \PP^5 \,:\, u \in \PP^2 \right\}
\end{align}
is a cubic surface in $\PP^5$ cut out by the {\em Cremona hexahedral equations} \cite{coble}: 
\begin{align}\label{eq:canonical cubic surface}
\begin{split}
    z_1^3+z_2^3+z_3^3+z_4^3+z_5^3+z_6^3&=0\\
    z_1 + z_2 + z_3 + z_4 + z_5 + z_6 & = 0\\
    \bar{a}_x z_1+\bar{b}_x z_2 +\bar{c}_x z_3+\bar{d}z_4 +\bar{e}_x z_5 +\bar{f}_x z_6 &=0. 
    \end{split}
\end{align}
The scalars $\bar{a}_x,\ldots,\bar{f}_x$ are defined as follows. For the $6$ points $x_1, \ldots, x_6$ in $\PP^2_x$ define 
\begin{equation}[(ij)(kl)(rs)]:=[ijr][kls]-[ijs][klr].
\end{equation}
The vanishing of this invariant expresses that the three lines 
$\overline{x_ix_j}$, $\overline{x_kx_l}$ and $\overline{x_rx_s}$ 
meet in a point \cite[pp 169]{coble}. Using these invariants, Coble 
defines the  $6$ scalars \cite[pp 170]{coble}:
\begin{small}
\begin{align}\label{eq:bumped up Joubert invariants}
\begin{split}
\bar{a}_x&=[(25)(13)(46)]+[(51)(42)(36)]+[(14)(35)(26)]+[(43)(21)(56)]+[(32)(54)(16)]\\
\bar{b}_x&=[(53)(12)(46)]+[(14)(23)(56)]+[(25)(34)(16)]+[(31)(45)(26)]+[(42)(51)(36)]\\
     \bar{c}_x&=[(53)(41)(26)]+[(34)(25)(16)]+[(42)(13)(56)]+[(21)(54)(36)]+[(15)(32)(46)]\\
\bar{d}_x&=[(45)(31)(26)]+[(53)(24)(16)]+[(41)(25)(36)]+[(32)(15)(46)]+[(21)(43)(56)]\\
\bar{e}_x&=[(31)(24)(56)]+[(12)(53)(46)]+[(25)(41)(36)]+[(54)(32)(16)]+[(43)(15)(26)]\\
\bar{f}_x&=[(42)(35)(16)]+[(23)(14)(56)]+[(31)(52)(46)]+[(15)(43)(26)]+[(54)(21)(36)]
\end{split}
\end{align}
\end{small}

We now have all the ingredients to prove 
Theorem~\ref{thm:n=6 invariant theory}.

\begin{proof}[Proof of Theorem~\ref{thm:n=6 invariant theory}]
By \eqref{eq:equality of bumped joubert}, the vector 
$(a_y(b),\ldots,f_y(b))$ satisfies the equations in 
\eqref{eq:canonical cubic surface} where the scalars where computed from $\X$. Therefore, 
\begin{align} \label{eq:mixed eqn 1}
    \bar{a}_x a_y(b)+\bar{b}_x b_y(b) +\bar{c}_x c_y(b)+\bar{d}_x d_y(b) +\bar{e}_x e_y(b) +\bar{f}_x f_y(b) =0.
\end{align}
Switching the roles of $\X$ and $\Y$ we also have that 
\begin{align} \label{eq:mixed eqn 2}
    \bar{a}_y a_x(a)+\bar{b}_y b_x(a) +\bar{c}_y c_x(a)+\bar{d}_y d_x(a) +\bar{e}_y e_x(a) +\bar{f}_y f_x(a) =0.
\end{align}

Therefore, $a$ lies on the cubic curve $C_x$ in $\PP^2$ cut out by equation~\ref{eq:mixed eqn 2} and $b$ lies on the cubic curve $C_y$ in $\PP^2$ cut out by equation~\ref{eq:mixed eqn 1}. Check that $x_1, \ldots, x_6$ lie on $C_x$ and $y_1, \ldots, y_6$ lie on $C_y$. 

Next we argue that if we fix $a \in C_x$ then $b$ is uniquely determined. Denote the cubic surface \eqref{eq:Sx} in $\PP^5$ by $S_x$, and define $S_y$ similarly. The maps $\PP^2_x \dashrightarrow S_x$ and $\PP^2_y \dashrightarrow S_y$ are birational automorphisms and are one-to-one except at $x_i$ and $y_i$ respectively. Fix $a\in C_x$. Then the vector
\begin{equation}
g'(\mathcal{X},a)\sim(a_x(a),\ldots,f_x(a))
\end{equation}
represents the unique point in $S_x$ corresponding to $a$. Since it satisfies the equation $\bar{a}_yz_1+\ldots+\bar{f}_yz_6=0$, it is also a point in $S_y$; there is therefore a unique point $b\in C_y$ satisfying $g'(\mathcal{Y},b)\sim g'(\mathcal{X}, a)$.
\end{proof}

\begin{remark}
Note that if we pick $a \in C_x$, then we could also do the intersecting conic construction to locate $b$. This would work by constructing all the $15={6 \choose 4}$ conics in $\PP^2_x$ that pass through $a \in C_x$ and $\X \setminus \{x_i,x_j\}$, and then constructing the conics in $\PP^2_y$ passing through $\Y \setminus \{y_i,y_j\}$ with the corresponding conic cross-ratios. These $15$ conics in $\PP^2_y$ will intersect at a unique point which is $b$. However, this construction requires knowing the cubic $C_x$ on which to locate $a$. 
\end{remark}

\definecolor{BLUE}{rgb}{0.,0.,1.}
\definecolor{GRAY}{rgb}{0.5019607843137255,0.5019607843137255,0.5019607843137255}
\definecolor{ORANGE}{rgb}{1.,0.4980392156862745,0.}
\definecolor{YELLOW}{rgb}{1.,0.8,0.}
\definecolor{GREEN}{rgb}{0.,0.8,0.5}
\definecolor{RED}{rgb}{1.,0.,0.}
\definecolor{PINK}{rgb}{1.,0.,1.}

\def\RX{\textcolor{RED}{x_5}}
\def\BX{\textcolor{BLUE}{x_2}}
\def\YX{\textcolor{YELLOW}{x_4}}
\def\GX{\textcolor{GREEN}{x_3}}
\def\OX{\textcolor{ORANGE}{x_6}}
\def\PX{\textcolor{PINK}{x_1}}

\def\RZ{\textcolor{RED}{y_5}}
\def\YZ{\textcolor{YELLOW}{y_4}}
\def\BZ{\textcolor{BLUE}{y_2}}
\def\GZ{\textcolor{GREEN}{y_3}}
\def\OZ{\textcolor{ORANGE}{y_6}}
\def\PZ{\textcolor{PINK}{y_1}}
\def\PE{\textcolor{PINK}{\hat{y}_1}}

To geometrically construct the locus of flatland camera centers when $n=6,$ we may apply a variant of the  intersecting conic construction described in the previous section to 5-point subsets.
The exceptional points $\hat{x}, \hat{y}$ arising from this construction will play a key role.

Fix an ordered pair of indices $(i,j)$ with $1\le i , j \le 6,$ and $i\ne j.$
Similarly to the previous section, let $\hat{y}_i$ denote the exceptional point determined by the set $\Y \setminus \{ y_i \}.$
We define $\omega_{y,ij}$ to be the unique conic passing through the five-point set $\left(\Y \cup \{ \hat{y}_i \}\right) \setminus \{ y_i , y_j\}.$
For each $i=1,\ldots , 6,$ we then have
\begin{equation}
\label{eq:conic-6-intersection}
\displaystyle\bigcap_{\substack{1\le j \le 6\\j\ne i}} \omega_{y,ij}
=
\{ 
\hat{y}_i
\}.
\end{equation}
\Cref{fig:many-conics} below illustrates the intersection formula~\eqref{eq:conic-6-intersection} when $i=1.$
As Werner observes~\cite{werner},   the conics $\omega_{y,ij}$ may be defined independently of the exceptional points $\hat{y}_i$.
Indeed, the conic $\omega_{x,i}$ passing through $\X \setminus \{ x_i \}$ maps to $\omega_{y,ij}$
under the homography that sends $x_k \to y_k$ for $k\notin \{ i, j\}$.
\Cref{fig:many-conics} also illustrates that any two of the conics are sufficient to determine this point, since they will intersect in three points of $\Y.$ 
This furnishes a construction of six exceptional points $\hat{y}_1, \ldots , \hat{y}_6\in \PP_y^2$ from $\X$ and $\Y $ alone.
A similar procedure constructs $6$ exceptional points $\hat{x}_1, \ldots , \hat{x}_6 \in \PP_x^2$.

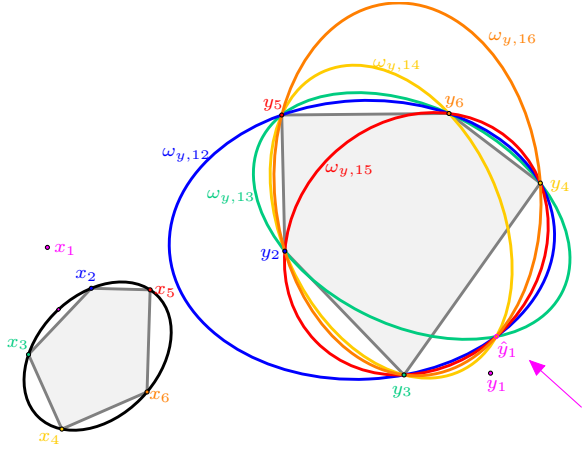
\begin{figure}[ht]
\begin{center}
	\begin{tikzpicture}[scale=0.7,
	line cap=round,line join=round,>=triangle 45,x=1.0cm,y=1.0cm]
\clip(-0.4, 0.2) rectangle (11., 8.7);
\fill[line width=1.1pt,color=GRAY,fill=GRAY,fill opacity=0.10000000149011612] 
	(2.3095922573070555,3.2293717316174972) -- (1.1921691758905404,3.2587776021810897) -- (0.,2.) -- 
	(0.6334576351822827,0.5828433808941588) -- (2.250780516179872,1.2885842744203828) -- cycle;
\fill[line width=1.1pt,color=GRAY,fill=GRAY,fill opacity=0.10000000149011612] 
	(7.984925276080415,6.581640975867061) -- (4.809091255212421,6.552235105303469) -- 
	(4.867902996339604,3.9645184957073134) -- (7.132155029736231,1.612048850619902) -- 
	(9.719871639332373,5.258376800505392) -- cycle;
\draw [line width=1.1pt,color=GRAY] (2.3095922573070555,3.2293717316174972)-- (1.1921691758905404,3.2587776021810897);
\draw [line width=1.1pt,color=GRAY] (1.1921691758905404,3.2587776021810897)-- (0.,2.);
\draw [line width=1.1pt,color=GRAY] (0.,2.)-- (0.6334576351822827,0.5828433808941588);
\draw [line width=1.1pt,color=GRAY] (0.6334576351822827,0.5828433808941588)-- (2.250780516179872,1.2885842744203828);
\draw [line width=1.1pt,color=GRAY] (2.250780516179872,1.2885842744203828)-- (2.3095922573070555,3.2293717316174972);
\draw [line width=1.1pt,color=GRAY] (7.984925276080415,6.581640975867061)-- (4.809091255212421,6.552235105303469);
\draw [line width=1.1pt,color=GRAY] (4.809091255212421,6.552235105303469)-- (4.867902996339604,3.9645184957073134);
\draw [line width=1.1pt,color=GRAY] (4.867902996339604,3.9645184957073134)-- (7.132155029736231,1.612048850619902);
\draw [line width=1.1pt,color=GRAY] (7.132155029736231,1.612048850619902)-- (9.719871639332373,5.258376800505392);
\draw [line width=1.1pt,color=GRAY] (9.719871639332373,5.258376800505392)-- (7.984925276080415,6.581640975867061);
\draw [rotate around={114.68902407968969:(6.886466828133069,4.525092469171441)},line width=1.1pt,color=YELLOW] 
	(6.886466828133069,4.525092469171441) ellipse (3.1294399224753064cm and 2.080058368627179cm);
\draw [rotate around={153.03600530288222:(7.276906417412576,4.633019789315503)},line width=1.1pt,color=GREEN] 
	(7.276906417412576,4.633019789315503) ellipse (3.204180981296815cm and 2.0642472009769848cm);
\draw [rotate around={-177.14549329486542:(6.3311158420769935,4.1786318150991955)},line width=1.1pt,color=BLUE] 
	(6.3311158420769935,4.1786318150991955) ellipse (3.661062188612483cm and 2.648316632417248cm);
\draw [rotate around={44.120321812824415:(7.3648447883247705,4.104244497381486)},line width=1.1pt,color=RED] 
	(7.3648447883247705,4.104244497381486) ellipse (2.6609727852289726cm and 2.3266218440921147cm);
\draw [rotate around={95.61954272787143:(7.197416952126105,5.139531741465944)},line width=1.1pt,color=ORANGE] 
	(7.197416952126105,5.139531741465944) ellipse (3.5514072743350056cm and 2.5236279525691354cm);
\draw [rotate around={-134.27589778940393:(1.3122202521785304,1.9698190312887291)},line width=1.1pt] 
	(1.3122202521785304,1.9698190312887291) ellipse (1.6147406180477664cm and 1.1537980351325563cm);
\begin{scriptsize}
\draw [fill=RED] (2.3095922573070555,3.2293717316174972) circle (1.1pt);
	\draw (2.6, 3.2) node {$\RX$};
\draw [fill=BLUE] (1.1921691758905404,3.2587776021810897) circle (1.1pt);
	\draw (1.1, 3.5) node {$\BX$};
\draw [fill=GREEN] (0.,2.) circle (1.1pt);
	\draw (-0.2, 2.2) node {$\GX$};
\draw [fill=YELLOW] (0.6334576351822827,0.5828433808941588) circle (1.1pt);
	\draw (0.4, 0.4) node {$\YX$};
\draw [fill=ORANGE] (2.250780516179872,1.2885842744203828) circle (1.1pt);
	\draw (2.5, 1.2) node {$\OX$};
\draw [fill=ORANGE] (7.984925276080415,6.581640975867061) circle (1.1pt);
	\draw (8.1, 6.8) node {$\OZ$};
\draw [fill=RED] (4.809091255212421,6.552235105303469) circle (1.1pt);
	\draw (4.65, 6.75) node {$\RZ $};
\draw [fill=BLUE] (4.867902996339604,3.9645184957073134) circle (1.1pt);
	\draw (4.6,3.9) node {$\BZ$};
\draw [fill=GREEN] (7.132155029736231,1.612048850619902) circle (1.1pt);
	\draw (7.1, 1.3) node {$\GZ$};	
\draw [fill=YELLOW] (9.719871639332373,5.258376800505392) circle (1.1pt);
	\draw (10.1, 5.2) node {$\YZ$};	

	\draw [color=PINK] (8.884919322745228,2.3416073525155885) circle (1.1pt);
		\draw (9.1, 2.1) node {$\PE$};
  \draw [->,color=PINK] (10.5, 1)--(9.5, 1.9);
	\draw [fill=PINK] (8.774397502378596,1.645740979998882) circle (1.1pt);
		\draw (8.9, 1.4) node {$\PZ$};		
	\draw [fill=PINK] (0.3586418322364089,4.035815590319276) circle (1.1pt);
		\draw (0.7, 4.0) node {$\PX$};
	\draw [fill=PINK] (0.5698247232086053,2.8568232200661283) ++(-1.1pt,0 pt) -- ++(1.1pt,1.1pt)--++(1.1pt,-1.1pt)--++(-1.1pt,-1.1pt)--++(-1.1pt,1.1pt);

 \draw [color=RED] (6.1,5.5) node {$\omega_{y,15}$};
 \draw [color=BLUE] (3,5.8) node {$\omega_{y,12}$};
 \draw [color=ORANGE] (9.2,8) node {$\omega_{y,16}$};
 \draw [color=GREEN] (3.85,5) node {$\omega_{y,13}$};
 \draw [color=YELLOW] (7,7.5) node {$\omega_{y,14}$};
\end{scriptsize}
\end{tikzpicture}
\caption{Constructing 5 conics from one, and the exceptional point $\hat{y}_1.$}
\label{fig:many-conics}
\end{center}
\end{figure}

As it turns out, the cubic curve $C_x$ (resp.~$C_y$) is uniquely determined by the requirement that it passes through the six points of $\X $ (resp.~$\Y$) and the six exceptional points $\hat{x}_1, \ldots, \hat{x}_6$ (resp.~$\hat{y}_1, \ldots , \hat{y}_6.$)
In fact, $\X$ together with any three of the exceptional points suffice to determine $C_x$ uniquely.

For $n=6$ in this section, and $n=7$ in the next, we consider the linear space of $3\times 3$ matrices consistent with the constraints imposed by $\X = \{ x_1, \ldots , x_n\}, \,   \Y = \{ y_1, \ldots , y_n \} \subset \PP^2$:
\begin{equation}\label{eq:LXY}
L_{\X , \Y } = \{ F \in \PP^8 \mid y_i^T F x_i =0, \, i=1, \ldots , n \} .
\end{equation}
When $n\le 9$ and $\X $ and $\Y$ are generic, we have that $\dim (L_{\X , \Y }) = 8 - n.$ 
When $n=6,$ we have $\dim (L_{\X , \Y }) =2,$ so $L_{\X , \Y}$ is spanned by three $3\times 3$ matrices $F_1, F_2, F_3$:
\begin{equation}\label{eq:LXY-6}
L_{\X , \Y } = \langle F_1, F_2, F_3 \rangle .
\end{equation}
With the parametrization~\eqref{eq:LXY-6}, the cubic curves may be characterized as follows:
\begin{align}
C_x &= \left\{ a \in \PP^2 \mid  \det \left[
\begin{array}{c|c|c}
F_1 a & F_2 a & F_3 a
\end{array}
\right] = 0 \right\}, \\
C_y &= \left\{ b\in \PP^2 \mid \det \left[
\begin{array}{c|c|c}
F_1^T b & F_2^T b & F_3^T b
\end{array}
\right] = 0\right\}.
\end{align}
Since $y_i^T F_j x_i=0$ for $i=1, \ldots , 6$, $j=1,\ldots , 3,$ this gives $\X \subset C_x$, $\Y \subset C_y.$
The fact that $C_x$ and $C_y$ contain the exceptional points follows from~\cite[Lemma 6.1]{connelly2024geometry}.

We conclude this section with 
\Cref{thm:main thm n=6}; its statement summarizes our results for the case $n=6$ in the language of the camera centers variety.

\begin{theorem} \label{thm:main thm n=6}
    The camera centers variety $\EPV{6}$ is a curve in $\PP^2 \times \PP^2$. 
    Furthermore,
    \begin{enumerate}
        \item Both of the coordinate projections $\pi_x : \EPV{6} \to \PP_x^2$, $\pi_y : \EPV{6} \to \PP_y^2$ project $\EPV{6}$ onto  nonsingular cubic plane curves $C_x \subset \PP_x^2, $ $C_y \subset \PP_y^2,$ with $\X \subset C_x,$ $\Y \subset C_y.$
        \item The multidegrees of $\EPV{6}$ are 
        $d_{(1,0)} \left( \EPV{6} \right) = 
d_{(0,1)} \left( \EPV{6}\right) = 3.$ 
     \end{enumerate}
\end{theorem}

We note additionally that the vanishing ideal $\mathcal{I} (\EPV{6})$ is generated by the equations of $C_x$ and $C_y,$ of respective bidegrees $(3,0)$ and $(0,3),$ and $3$ polynomials of bidegree $(1,1).$

\section{Camera loci when $n=7$}\label{sec:7-points}

Finally, we address the case where $\X $ and $\Y $ consist of $n=7$ generic points each.
We begin by describing a simple geometric construction that locates the three camera centers. 
Using the construction outlined in the previous section, we can construct two cubic curves corresponding to two 6-point subsets. In the left panel of~\Cref{fig:implicit-curve-7pts-2curves}, we see two cubic curves corresponding to $\{x_1,\dots,x_5\}$ and either $x_6$ (blue) or $x_7$ (red). These cubics intersect in nine complex points.
Six of these points are given by $x_1,\dots,x_5$ and the corresponding exceptional point $\hat{x}.$ 
The remaining three intersection points $a_1, a_2, a_3,$ are the possible cameras.
The right panel further illustrates that this construction does not rely on the choice of 2 cubics: all 7 pass through $a_1,a_2,a_3.$ 

\begin{figure}
    \centering  \includegraphics[width=0.45\linewidth]{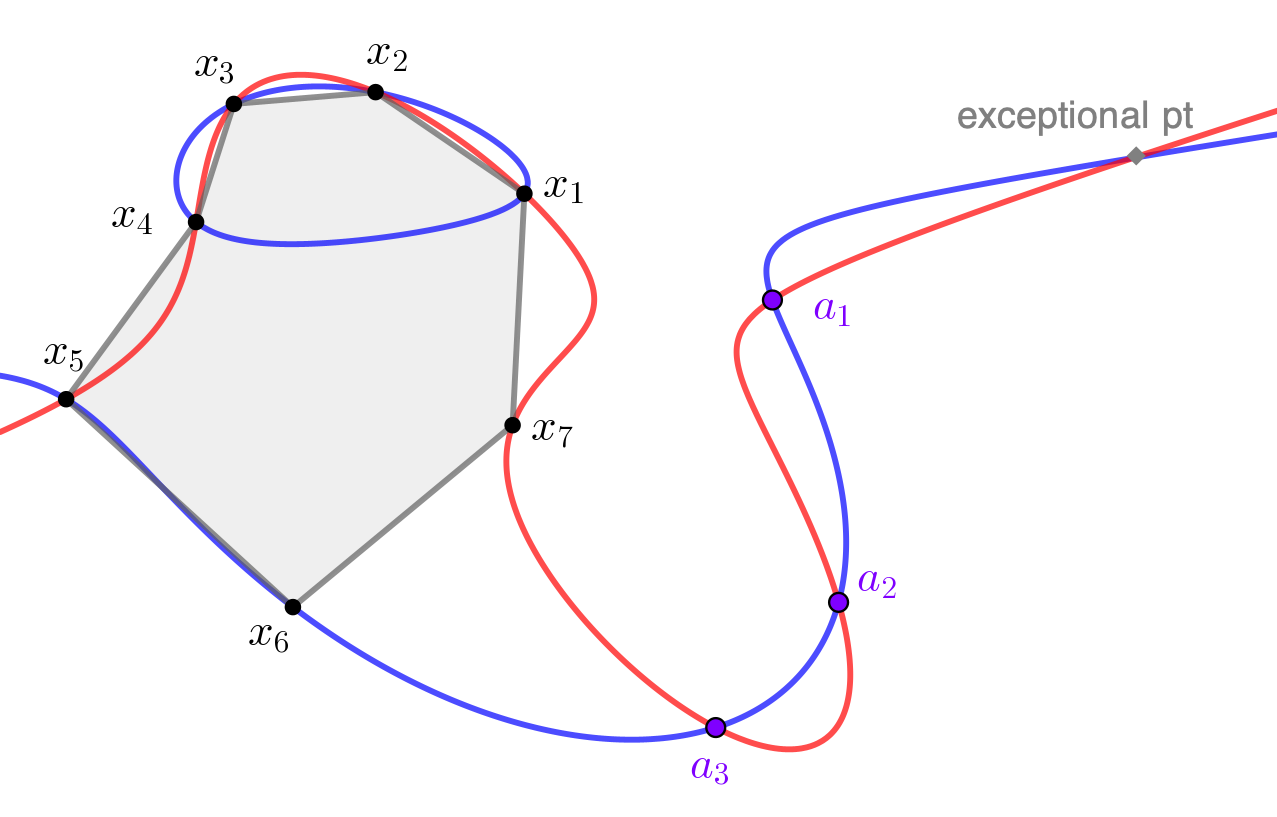}
    \quad
\includegraphics[width=0.45\linewidth]{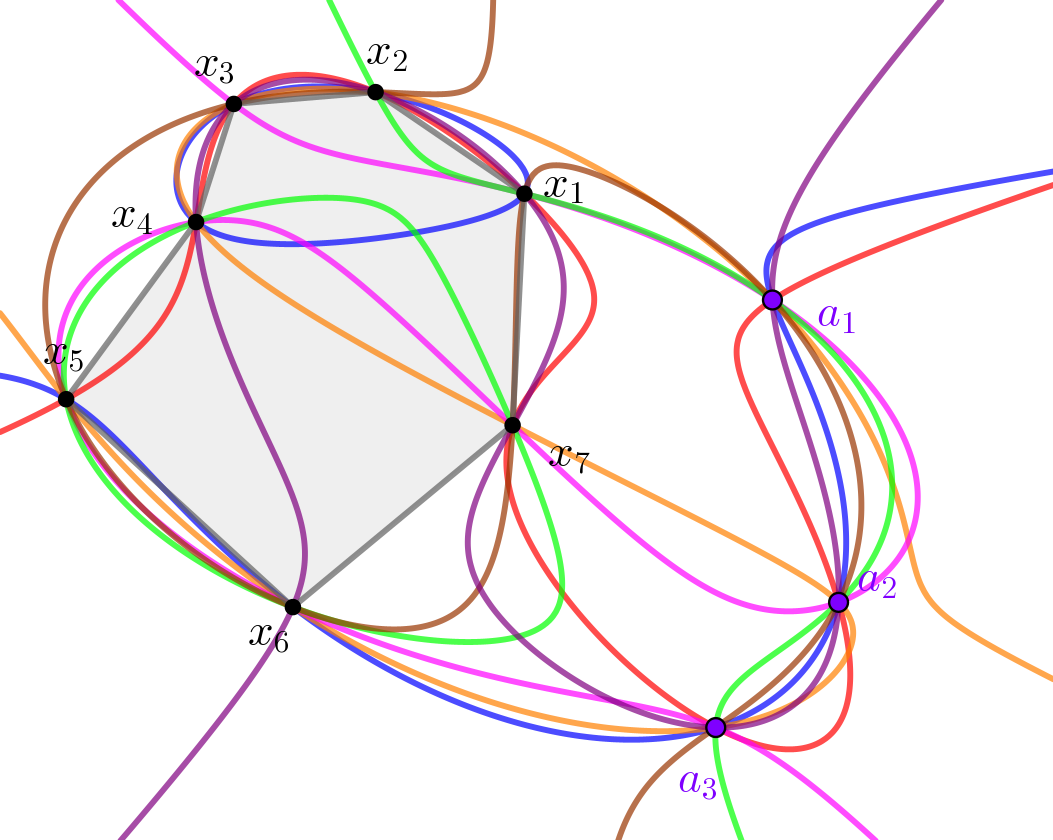}
    \caption{ Left: For a set of 7 labeled points $\mathcal{X} = \{ x_1, \ldots , x_7 \}$, we can use two cubics associated with  6-point subsets to locate the camera centers $a_1$, $a_2$, and $a_3$.  Right: All seven such cubics pass through the same three points.}
    \label{fig:implicit-curve-7pts-2curves}
\end{figure}

Despite the simplicity of this construction, the cubics it produces  are surprisingly \emph{not the simplest equations} vanishing on $\EPV{7}$, nor its projections into $\PP_x^2$ and $\PP_y^2.$
Indeed, computing the vanishing ideal $\mathcal{I} (\EPV{7})$, we find that it is generated by three types of equations: 3 of bidegree $(2,0)$, 3 of bidegree $(0,2)$, and 6 of bidegree $(1,1).$
We use the remainder of this section to explain the first two types of equations, which give \emph{conics} constraining the locations of the camera centers $a$ and $b$.

Let us consider again the linear subspace $L_{\X , \Y}$ defined in~\eqref{eq:LXY}. 
Since $n=7,$ we have $\dim (L_{\X , \Y}) = 1.$
Fix two distinct points $F_1, F_2 \in L_{\X , \Y }$ so that
\begin{equation}\label{eq:pencil-generators-F-7pt}
L_{\X , \Y } = \langle F_1, F_2 \rangle .
\end{equation}
Consider now some fundamental matrix $F\in L_{\X , \Y}$, so that 
\begin{equation}
F = s_1 F_1 + s_2 F_2 ,
\quad 
\det F = 0,
\end{equation}
for some $[s_1:s_2] \in \PP^1.$
Letting $a$ be the right epipole of $F$, we have
\begin{equation}\label{eq:sFA}
s_1 F_1 a + s_2 F_2 a = 0.
\end{equation}
Using the cross product, we obtain three quadratic equations from~\eqref{eq:sFA}:
\begin{equation}\label{eq:crossproduct-conics}
F_1 a \times F_2 a  = 0.
\end{equation}
Thus, $a$ must lie on the intersection
of the three plane conics in $a$ defined by~\eqref{eq:crossproduct-conics}.
We claim that this intersection is a set of three points. 
To see this, let us write
\begin{equation}\label{eq:F1-F2}
F_1 = \begin{bmatrix}
f_{11}^T \\
f_{12}^T \\
f_{13}^T
\end{bmatrix}
,
\quad
F_2 = \begin{bmatrix}
f_{21}^T\\
f_{22}^T \\
f_{23}^T
\end{bmatrix}
\end{equation}
so that~\eqref{eq:sFA} is equivalent to the rank constraint
\begin{align}
\rank \, M(a) \le 1, \text{ where} \label{eq:rank-constraint}\\
\label{eq:linear-form-mat}
M(a) = 
\begin{bmatrix}
F_1 a & F_2 a 
\end{bmatrix}
=
\begin{bmatrix}
a^T
\begin{bmatrix}
f_{11} & f_{21}    
\end{bmatrix} \\
a^T \begin{bmatrix}
f_{12} & f_{22}    
\end{bmatrix}\\
a^T\begin{bmatrix}
f_{13} & f_{23}    
\end{bmatrix}
\end{bmatrix}.
\end{align}
Since the data is generic, we may assume that each row of $M(a)$ is nonzero.
Each conic equation in~\eqref{eq:crossproduct-conics} is equivalent to requiring some pair of rows of $M(a)$ to be linearly dependent.
Thus, any pair of minors of $M(a)$ gives two conics intersecting in four points.
Among these points of intersection, we claim that exactly three lie on all three conics.
To see this, let us consider the pair of conics that involve the third row of $M(a).$
These conics intersect at $a$ if there exist scalars $\lambda_1, \lambda_2$ such that
\begin{equation}\label{eq:conic-rows}
a^T \begin{bmatrix}
f_{13} & f_{23}    
\end{bmatrix}
=
\lambda_1 a^T \begin{bmatrix}
f_{11} & f_{21}    
\end{bmatrix},
\quad 
a^T
\begin{bmatrix}
f_{13} & f_{23}    
\end{bmatrix} 
=
\lambda_2 a^T \begin{bmatrix}
f_{12} & f_{22}   
\end{bmatrix} .
\end{equation}
Among the four points where these conics intersect, there is the distinguished point $a = f_{13} \times f_{23}$, allowing us to take $\lambda_1=\lambda_2=0$ in~\eqref{eq:conic-rows}.
By the genericity of $\X $ and $\Y$, we may assume that this  point does \emph{not} lie on the conic expressing the linear dependence of the first two rows of $M(a).$
If we now let $a$ be any one of the other three points of intersection, then both $\lambda_1$ and $\lambda_2$ in~\eqref{eq:conic-rows} are nonzero, implying the first two rows of $M(a)$ are dependent.
Thus $a$ lies on all three conics.

Note that the same argument applies to the remaining conic pairs for $a$,
and an analogous triple of conics for $b.$
Moreover, from the $n=5$ case, we know that $a$ and $b$ are related by a degree $5$ Cremona transformation.
We summarize our discussion with the following theorem.
\begin{theorem}\label{thm:n-7}
When $n=7$, each of the camera centers $a,b \in \PP^2$ must lie on the intersection of three conics, $\omega_{a,i}$ and $\omega_{b,i}$ respectively, for $i=1,\ldots ,3$, which can be explicitly computed from the data $\X $ and $\Y.$
Moreover, each triple of conics intersects in a set of exactly three complex points,
\[
\# \displaystyle\bigcap_{i=1}^3 \omega_{a,i}
=
\# \displaystyle\bigcap_{i=1}^3 \omega_{b,i}
= 3.
\]
Once one camera center is fixed, the other is uniquely determined. Thus, the camera centers variety $\EPV{7}$ is a finite set of three points, $d_{(0,0)} (\EPV{7}) = 3,$ and there are at most $3$ choices for the pair $(a,b).$
\end{theorem}

As noted in~\cite[Remark 3]{connelly2024lines}, the $2 \times 2$ minors of the matrix $M(a)$ define a quadratic Cremona transformation $\PP^2 \dashrightarrow \PP^2$ sending $x_i \to y_i$ for $i=1,\ldots , 7.$
The three base points of this transformation are the three possible camera centers in $\pi_x (\EPV{7})=\cap\omega_{a,i}.$
Moreover, each conic $\omega_{a,i}$ is the pre-image of the line $\langle e_j,e_k\rangle\subset\PP^2_y$, where $j,k\ne i$, under this transformation.

When the data $\X $ and $\Y$ are defined by sufficiently generic \emph{rational numbers}, the three distinguished points $f_{1i} \times f_{2i}$, $i=1,\ldots 3$ that arise in the proof of~\Cref{thm:n-7} are also rational.
This in turn implies that the conics $\omega_{a,i}$ and $\omega_{b,i}$ appearing in~\Cref{thm:n-7} contain infinitely many rational points.
This can be proved by \emph{stereographic projection}: if $\ell $ is any rational line passing through the distinguished point that is not tangent to the conic, then its points of intersection with the conic are determined by the roots of a quadratic polynomial: since one of the roots corresponding to the distinguished point is rational, it follows that the other root (and hence the other point of intersection) must also be.
Thus, we may construct infinitely many rational points on the conics $\omega_{a,i}, \omega_{b,i}.$

\bibliographystyle{plain}
\bibliography{references}
\end{document}